\documentclass{amsart}

\usepackage{amsfonts,amssymb,verbatim,amsmath,amsthm,latexsym,textcomp,amscd, hyperref}
\usepackage{epsfig}
\usepackage{amsmath,amsthm,graphics}
\usepackage{graphicx}
\usepackage{color}
\usepackage{url}

\begin{document}

\newtheorem{theorem}{Theorem}[section]
\newtheorem{prop}[theorem]{Proposition}
\newtheorem{lemma}[theorem]{Lemma}
\newtheorem{cor}[theorem]{Corollary}
\newtheorem{cond}[theorem]{Condition}
\newtheorem{ing}[theorem]{Ingredients}
\newtheorem{conj}[theorem]{Conjecture}
\newtheorem{claim}[theorem]{Claim}
\newtheorem{constr}[theorem]{Construction}
\newtheorem{rem}[theorem]{Remark}
\newtheorem{scheme}[theorem]{Scheme}
\newtheorem{assume}[theorem]{Standing Assumption}

\newtheorem*{theorem*}{Theorem}
\newtheorem*{modf}{Modification for arbitrary $n$}
\newtheorem{qn}[theorem]{Question}
\newtheorem{condn}[theorem]{Condition}

\theoremstyle{definition}
\newtheorem{defn}[theorem]{Definition}
\newtheorem{eg}[theorem]{Example}
\newtheorem{rmk}[theorem]{Remark}

\newcommand{\map}{\rightarrow}
\newcommand{\boundary}{\partial}
\newcommand{\C}{{\mathbb C}}
\newcommand{\integers}{{\mathbb Z}}
\newcommand{\natls}{{\mathbb N}}
\newcommand{\ratls}{{\mathbb Q}}
\newcommand{\reals}{{\mathbb R}}
\newcommand{\proj}{{\mathbb P}}
\newcommand{\lhp}{{\mathbb L}}
\newcommand{\tr}{{\operatorname{Tread}}}
\newcommand{\rs}{{\operatorname{Riser}}}
\newcommand{\tube}{{\mathbb T}}
\newcommand{\cusp}{{\mathbb P}}
\newcommand\AAA{{\mathcal A}}
\newcommand\BB{{\mathcal B}}
\newcommand\CC{{\mathcal C}}
\newcommand\ccd{{{\mathcal C}_\Delta}}
\newcommand\DD{{\mathcal D}}
\newcommand\EE{{\mathcal E}}
\newcommand\FF{{\mathcal F}}
\newcommand\GG{{\mathcal G}}
\newcommand\HH{{\mathcal H}}
\newcommand\II{{\mathcal I}}
\newcommand\JJ{{\mathcal J}}
\newcommand\KK{{\mathcal K}}
\newcommand\LL{{\mathcal L}}
\newcommand\MM{{\mathcal M}}
\newcommand\NN{{\mathcal N}}
\newcommand\OO{{\mathcal O}}
\newcommand\PP{{\mathcal P}}
\newcommand\QQ{{\mathcal Q}}
\newcommand\RR{{\mathcal R}}
\newcommand\SSS{{\mathcal S}}
\newcommand\TT{{\mathcal T}}
\newcommand\ttt{{\mathcal T}_T}
\newcommand\tT{{\widetilde T}}
\newcommand\UU{{\mathcal U}}
\newcommand\VV{{\mathcal V}}
\newcommand\WW{{\mathcal W}}
\newcommand\XX{{\mathcal X}}
\newcommand\YY{{\mathcal Y}}
\newcommand\ZZ{{\mathcal Z}}
\newcommand\CH{{\CC\HH}}
\newcommand\TC{{\TT\CC}}
\newcommand\EXH{{ \EE (X, \HH )}}
\newcommand\GXH{{ \GG (X, \HH )}}
\newcommand\GYH{{ \GG (Y, \HH )}}
\newcommand\PEX{{\PP\EE  (X, \HH , \GG , \LL )}}
\newcommand\MF{{\MM\FF}}
\newcommand\PMF{{\PP\kern-2pt\MM\FF}}
\newcommand\ML{{\MM\LL}}
\newcommand\mr{{\RR_\MM}}
\newcommand\tmr{{\til{\RR_\MM}}}
\newcommand\PML{{\PP\kern-2pt\MM\LL}}
\newcommand\GL{{\GG\LL}}
\newcommand\Pol{{\mathcal P}}
\newcommand\half{{\textstyle{\frac12}}}
\newcommand\Half{{\frac12}}
\newcommand\Mod{\operatorname{Mod}}
\newcommand\Area{\operatorname{Area}}
\newcommand\ep{\epsilon}
\newcommand\hhat{\widehat}
\newcommand\Proj{{\mathbf P}}
\newcommand\U{{\mathbf U}}
 \newcommand\Hyp{{\mathbf H}}
\newcommand\D{{\mathbf D}}
\newcommand\Z{{\mathbb Z}}
\newcommand\R{{\mathbb R}}
\newcommand\s{{\Sigma}}
\renewcommand\P{{\mathbb P}}
\newcommand\Q{{\mathbb Q}}
\newcommand\E{{\mathbb E}}
\newcommand\til{\widetilde}
\newcommand\length{\operatorname{length}}
\newcommand\BU{\operatorname{BU}}
\newcommand\gesim{\succ}
\newcommand\lesim{\prec}
\newcommand\simle{\lesim}
\newcommand\simge{\gesim}
\newcommand{\simmult}{\asymp}
\newcommand{\simadd}{\mathrel{\overset{\text{\tiny $+$}}{\sim}}}
\newcommand{\ssm}{\setminus}
\newcommand{\diam}{\operatorname{diam}}
\newcommand{\pair}[1]{\langle #1\rangle}
\newcommand{\T}{{\mathbf T}}
\newcommand{\inj}{\operatorname{inj}}
\newcommand{\pleat}{\operatorname{\mathbf{pleat}}}
\newcommand{\short}{\operatorname{\mathbf{short}}}
\newcommand{\vertices}{\operatorname{vert}}
\newcommand{\collar}{\operatorname{\mathbf{collar}}}
\newcommand{\bcollar}{\operatorname{\overline{\mathbf{collar}}}}
\newcommand{\I}{{\mathbf I}}
\newcommand{\tprec}{\prec_t}
\newcommand{\fprec}{\prec_f}
\newcommand{\bprec}{\prec_b}
\newcommand{\pprec}{\prec_p}
\newcommand{\ppreceq}{\preceq_p}
\newcommand{\sprec}{\prec_s}
\newcommand{\cpreceq}{\preceq_c}
\newcommand{\cprec}{\prec_c}
\newcommand{\topprec}{\prec_{\rm top}}
\newcommand{\Topprec}{\prec_{\rm TOP}}
\newcommand{\fsub}{\mathrel{\scriptstyle\searrow}}
\newcommand{\bsub}{\mathrel{\scriptstyle\swarrow}}
\newcommand{\fsubd}{\mathrel{{\scriptstyle\searrow}\kern-1ex^d\kern0.5ex}}
\newcommand{\bsubd}{\mathrel{{\scriptstyle\swarrow}\kern-1.6ex^d\kern0.8ex}}
\newcommand{\fsubeq}{\mathrel{\raise-.7ex\hbox{$\overset{\searrow}{=}$}}}
\newcommand{\bsubeq}{\mathrel{\raise-.7ex\hbox{$\overset{\swarrow}{=}$}}}
\newcommand{\tw}{\operatorname{tw}}
\newcommand{\base}{\operatorname{base}}
\newcommand{\trans}{\operatorname{trans}}
\newcommand{\rest}{|_}
\newcommand{\bbar}{\overline}
\newcommand{\UML}{\operatorname{\UU\MM\LL}}
\renewcommand{\d}{\operatorname{dia}}
\newcommand{\hs}{{\operatorname{hs}}}
\newcommand{\vT}{{V(T)}}
\newcommand{\EL}{\mathcal{EL}}
\newcommand{\tsum}{\sideset{}{'}\sum}
\newcommand{\tsh}[1]{\left\{\kern-.9ex\left\{#1\right\}\kern-.9ex\right\}}
\newcommand{\Tsh}[2]{\tsh{#2}_{#1}}
\newcommand{\qeq}{\mathrel{\approx}}
\newcommand{\Qeq}[1]{\mathrel{\approx_{#1}}}
\newcommand{\qle}{\lesssim}
\newcommand{\Qle}[1]{\mathrel{\lesssim_{#1}}}
\newcommand{\simp}{\operatorname{simp}}
\newcommand{\vsucc}{\operatorname{succ}}
\newcommand{\vpred}{\operatorname{pred}}
\newcommand\fhalf[1]{\overrightarrow {#1}}
\newcommand\bhalf[1]{\overleftarrow {#1}}
\newcommand\sleft{_{\text{left}}}
\newcommand\sright{_{\text{right}}}
\newcommand\sbtop{_{\text{top}}}
\newcommand\sbot{_{\text{bot}}}
\newcommand\sll{_{\mathbf l}}
\newcommand\srr{_{\mathbf r}}
\newcommand\geod{\operatorname{\mathbf g}}
\newcommand\mtorus[1]{\boundary U(#1)}
\newcommand\A{\mathbf A}
\newcommand\Aleft[1]{\A\sleft(#1)}
\newcommand\Aright[1]{\A\sright(#1)}
\newcommand\Atop[1]{\A\sbtop(#1)}
\newcommand\Abot[1]{\A\sbot(#1)}
\newcommand\boundvert{{\boundary_{||}}}
\newcommand\storus[1]{U(#1)}
\newcommand\Momega{\omega_M}
\newcommand\nomega{\omega_\nu}
\newcommand\twist{\operatorname{tw}}
\newcommand\modl{M_\nu}
\newcommand\MT{{\mathbb T}}
\newcommand\dw{{d_{weld}}}
\newcommand\dt{{d_{te}}}
\newcommand\Teich{{\operatorname{Teich}}}
\renewcommand{\Re}{\operatorname{Re}}
\renewcommand{\Im}{\operatorname{Im}}
\newcommand{\mc}{\mathcal}
\newcommand{\ccs}{{\CC(S)}}
\newcommand{\mtdw}{{({M_T},\dw)}}
\newcommand{\tmtdw}{{(\til{M_T},\dw)}}
\newcommand{\tmldw}{{(\til{M_l},\dw)}}
\newcommand{\mtdt}{{({M_T},\dt)}}
\newcommand{\tmtdt}{{(\til{M_T},\dt)}}
\newcommand{\tmldt}{{(\til{M_l},\dt)}}
\newcommand{\trvw}{{\tr_{vw}}}
\newcommand{\ttrvw}{{\til{\tr_{vw}}}}
\newcommand{\but}{{\BU(T)}}
\newcommand{\ilkv}{{i(lk(v))}}
\newcommand{\pslc}{{\mathrm{PSL}_2 (\mathbb{C})}}
\newcommand{\tttt}{{\til{\ttt}}}
\newcommand{\bcomment}[1]{\textcolor{blue}{#1}}
\newcommand{\jfm}[1]{\marginpar{#1\quad -jfm}}

\title{Tight trees and model geometries of surface bundles over graphs}

\author{Mahan Mj}

\address{School of Mathematics, Tata Institute of Fundamental Research, Mumbai-40005, India}
\email{mahan@math.tifr.res.in}
\email{mahan.mj@gmail.com}

\subjclass[2010]{ 20F65, 20F67 (Primary), 22E40, 57M50}
\keywords{model geometry, ending lamination theorem, convex cocompact subgroup, curve complex, tight tree, tight geodesic}

\thanks{The author was supported by  the Department of Atomic Energy, Government of India, under project no.12-R\&D-TFR-5.01-0500;
	and in part by the National Science Foundation 
	under Grant No. DMS-1440140 at the Mathematical Sciences Research Institute in Berkeley
	during the Fall 2016 program in Geometric Group Theory. Research 
	also partly supported by  a DST JC Bose Fellowship, Matrics research project grant  MTR/2017/000005, CEFIPRA  project No. 5801-1, as also  by an endowment of the Infosys Foundation.} 
\date{\today}

\begin{abstract}   We generalize the notion of tight geodesics in the curve complex to tight trees. We then use tight trees to construct model geometries for certain surface bundles over graphs. This extends some aspects of the
 combinatorial model for doubly degenerate hyperbolic 3-manifolds  developed by   Brock, Canary and Minsky  during the course of their proof of the Ending Lamination Theorem. Thus we obtain uniformly Gromov-hyperbolic geometric model spaces  equipped with geometric $G-$actions, where  $G$ admits an exact sequence of the form $$1 \to \pi_1(S) \to G \to Q \to 1.$$  Here $S$ is a closed surface of genus $g > 1$ and $Q$ belongs to a special class of free convex  cocompact subgroups of the mapping class group $MCG(S)$.
\end{abstract}

\maketitle

\tableofcontents

\section{Introduction} \label{sec-intro}
A combinatorial model for doubly degenerate hyperbolic 3-manifolds was developed by   Brock, Canary and Minsky in \cite{minsky-elc1,minsky-elc2} during the course of their proof of the Ending Lamination Theorem. The combinatorial machinery guiding the construction of the combinatorial model in \cite{minsky-elc1,minsky-elc2} is based on the technology of hierarchy paths developed by Masur and Minsky in \cite{masur-minsky,masur-minsky2}. Let $\CC(S)$ denote the curve complex of a closed surface $S$. Then  the boundary $\partial \CC(S)$ consists of the ending laminations $\EL (S)$ \cite{klarreich-el}. For
a pair of ending laminations $\LL_\pm \in \partial \CC(S) = \EL (S)$, let $\gamma$ be a tight geodesic  in the curve complex $\CC(S)$ joining $\LL_\pm$.
A hierarchy of paths joining $\LL_\pm$ is then constructed in \cite{masur-minsky2,minsky-elc1} with $\gamma$ as the base tight geodesic. The hierarchy  forms the combinatorial backbone for the model (see also \cite{bowditch-model,ohshika-tams} for some alternate treatments).\\
  
 \noindent {\bf Convex cocompact subgroups of the mapping class group:} We shall extend some aspects of the combinatorial model to treat a class of free convex  cocompact subgroups of the mapping class group $MCG(S)$. 
 A subgroup $Q$ of  $MCG(S)$ is said to be {\bf convex cocompact}  \cite{farb-coco}  if some  orbit of $Q$ in the Teichm\"uller space $\Teich(S)$
 is quasiconvex.  We shall say that  $Q$ is $K-$convex  cocompact if the weak hull of the limit set of $Q$ quotiented by $Q$ has diameter at most $K$; equivalently some $Q$ orbit is $K-$quasiconvex.
 
 Associated to any $Q \subset MCG(S)$, there is an
  exact sequence \cite[Section 1.2]{farb-coco} of the form $$1 \to \pi_1(S) \to G \to Q \to 1.$$ It follows from work of Farb-Mosher  \cite{farb-coco}
  Hamenstadt \cite{hamen} and Kent-Leininger \cite{kl-coco} that the following are equivalent:
  \begin{enumerate}
  \item  $Q$ is  convex cocompact,
  \item the extension $G$ occurring in the 
  above exact sequence   is hyperbolic (see also  \cite{mahan-sardar} for an extension  to surfaces with punctures),
  \item Any orbit of $Q$ in $\CC(S)$ is qi-embedded.
  \end{enumerate}
 Our principal aim in this paper is to construct uniformly Gromov-hyperbolic geometric model spaces  equipped with geometric $G-$actions, where $G$ is as above and $Q$ belongs to a special class of free convex  cocompact subgroups of the mapping class group $MCG(S)$.

   Identifying $Q$ with an  orbit in $\CC(S)$,  	the Gromov boundary
 	$\partial Q$ of $Q$ can be canonically identified with a Cantor set in $\EL(S)$ as well as in the Thurston boundary $\PML (S) = \partial \Teich(S)$ of Teichm\"uller space. 
 	In order to construct a model space for $G$, we shall first need to construct tight geodesics and
 	a hierarchy of paths  for every pair of points $p, q$
 	in $\partial(Q)$. \\
 	
  \noindent {\bf Model geometries:}	A crucial issue that arises in the process is to check for consistency: When two such tight geodesics $\gamma_1,$ (resp. $\gamma_2$) joining $p_1, q_1$ (resp. $p_2, q_2$) cross at a vertex $v $ then the hierarchy paths joining $p_1, q_1$ (resp. $p_2, q_2$) subordinate to $v$ need to be consistent. This is one of the new and somewhat subtle features that appears when $\partial Q$ is a Cantor set as opposed to the case where $Q=\Z$ and $\partial Q$ has exactly two points.
 	There are
 	 two cases in which we can handle this problem corresponding to the following two model geometries of doubly degenerate 3-manifolds:
 	 
 	 \begin{enumerate}
 	 	\item  bounded geometry \cite{minsky-bddgeom},
 	 	\item a special case of the split geometry model investigated in \cite{mahan-split,mahan-amalgeo}.
 	 \end{enumerate}
 	 
 \subsection{Statement of results}
	  The first case that we address is that
 of {\bf bounded geometry,} where	 
 	 we assume that there exists $\ep >0$ such that for all $p, q \in \partial Q \subset \partial \Teich (S)$, the Teichm\"uller geodesic 
 	 joining $p, q$ lies in the $\ep-$thick part of  Teichm\"uller  space.
 	Suppose now that $Q$ is free. Let $\Gamma_Q$ be a Cayley graph of $Q$ with respect to a free generating set and $\Phi: \Gamma_Q \to Teich(S)$ be a piecewise geodesic equivariant map. The pull-back of the universal bundle to $\Gamma_Q$ will be denoted as $M_{Q,\Phi}$.
 	Then (see Proposition \ref{thickbdlhyp}) we have:
 	 
 	 \begin{prop}\label{thickbdlhyp-intro} Given $K, \ep \geq 0$, there exists $\delta >0$ such that the following holds:\\
 	 	Let $Q$ be a free $K-$convex cocompact subgroup and  let $o \in Teich(S)$ with $Q.o \subset Teich_\ep(S)$. There exists $\Phi: \Gamma_Q \to Teich(S)$ such that 
 	  the universal cover $\til{M_{Q,\Phi}}$ is $\delta-$hyperbolic.
 	 \end{prop}

 Generalizing the notion of a tight geodesic from \cite{masur-minsky,masur-minsky2}, we say that	 a simplicial map $i: T \to \CC (S)$ from a (not necessarily regular) simplicial tree $T$ of bounded valence defines an {\bf $L-$tight tree of non-separating curves} if  
 	 for every vertex $v$ of $T$, $i(v)$ is non-separating, and for every pair of distinct vertices $u \neq w$ adjacent to $v$ in $T$,
 	 $$d_{\CC(S \setminus i(v))} (i(u), i(w)) \geq L.$$
 	 The following (see Proposition \ref{isometrictighttree}, essentially due to Bromberg) shows that 
 	 $L$--tight trees  are isometrically embedded.

  \begin{prop}\label{isometrictighttree-intro} 
  	There exists $L\geq 3$,such that the following holds.  Let $S$ be a closed surface of genus at least $3$, and let $i: T\to \CC(S)$ define an $L$--tight tree of non-separating curves.  Then $i$ is an isometric embedding.
  \end{prop}

Let $i: T \to \CC(S)$ be  a tight tree of non-separating curves and let  $v$ be a vertex of $T$. The link of $v$ in $T$  is denoted as $lk(v)$.
Let $W_v = S \setminus i(v)$. Then $i(lk(v))$ consists of a uniformly bounded number of vertices in $\CC(W_v)$. Hence the weak convex hull 
$CH(i(lk(v)) $ of 
$i(lk(v))$ in  $\CC(W_v)$ 
admits a uniform approximating tree $T_v$. We refer to $T_v$ as the {\bf tree-link} of $v$.
The {\bf blow-up} $\BU (T)$ of $T$ is a {\bf metric tree} obtained from $T$ by replacing the $\half-$neighborhood of each $v \in T$ by the tree-link $T_v$.

 An $L-$tight tree is said to be {\bf $R-$thick} if
for any vertices $u,v, w$ of $T$ and any proper essential subsurface $W$ of $S \setminus  i(v)$ (including essential annuli), \[ d_W ((i(u), i(w)) \leq R,\] where $d_W(\cdot\ , \cdot)$ denotes distance in $\CC(W)$ between subsurface projections onto $W$. For an $L-$tight, $R-$thick tree $T$ we construct a bundle
 $P: M_T \to \but$ over the blow-up $\but$  of $T$. $M_T$ will take the place of the model manifold of \cite{minsky-elc1}. The pre-image $P^{-1} (T_v)$ will be called the {\bf building block} corresponding to $v$ and will be denoted as $M_v$. Inside every $M_v$, there is a natural copy of $S^1 \times T_v$ corresponding to the simple closed curve $i(v) \subset S$. We refer to it as the {\bf Margulis riser} corresponding to $v$ and denote it by $\RR_v$. Margulis risers  in $M_T$ take the place of Margulis tubes in hyperbolic 3-manifolds. (The terminology "riser" is borrowed from \cite{mms} where they form parts of tracks). One can think of the geometry of Margulis tubes in \cite{minsky-elc1} as a consequence of performing hyperbolic Dehn surgery on a thickened neighborhood of Margulis risers.  Equivalently, one thickens the Margulis risers, removes the interior and performs hyperbolic Dehn  filling. For convex cocompact free subgroups of $MCG(S)$, there is no such canonical filling. Thus Margulis risers are the best replacement we could find for Margulis tubes.

 For $l$ a bi-infinite geodesic in $T$, let $l_\pm$ denote the ending laminations given by the ideal end-points of $i(l)$ in the boundary of  $\CC(S)$ and let $N_l$ denote the doubly degenerate hyperbolic 3-manifold with ending laminations  $l_\pm$. 
 We  denote the vertices of $T$ occurring along $l$ by $\VV(l)$. If $L$ is large enough, then each $i(v)$   gives a  Margulis tube  $\T_v$  in $N_l$. Let $N_l^0 = N_l \setminus \bigcup_{v \in \VV(l)} \T_v$.
 
 Let $\BU(l)$ denote the bi-infinite geodesic in $\but$ after blowing up $l$ in $T$. 
 Also let $M_l$ denote the bundle over $\BU(l)$ induced from $\Pi: M_T \to \but$. Let $M_l^0 = M_l \setminus \bigcup_{v \in \VV(l)} \RR_v$.

 \begin{theorem} \label{model-str-intr} (See Theorem \ref{model-str})
 	Given $R \geq 0$, there exist $K \geq 1, e >0$ such that if $i: T \to \CC(S)$ is  an $L-$tight $R-$thick tree of non-separating curves, then the following holds:\\ There exists a  metric $\dw$ on $M_T$ such that $P: M_T \to \but$  satisfies the following properties:
 	\begin{enumerate}
 		\item The induced metric on a Margulis riser $\RR_v$ is the metric product $S^1_e \times T_v$, where $S^1_e$ is a round circle with radius $e$. 
 		\item  For any bi-infinite geodesic $l$ in $T$, $N_l^0$ and $M_l^0$ are $K-$bi-Lipschitz homeomorphic.
 		\item Further, if there exists a subgroup $Q$ of $MCG(S)$ acting cocompactly and geometrically on $i(T)$, then this action can be lifted to an isometric fiber-preserving isometric action of $Q$ on $(M_T,\dw)$.
 		\item $P: (M_T,\dw) \to \but$ is uniformly proper.
 	\end{enumerate}
 \end{theorem}  

The universal cover $\tmtdw$ contains flat strips $\R \times T_v$ coming from the universal covers of
the Margulis risers $\RR_v= S^1_e \times T_v$. We show that this is the only obstruction to effectively hyperbolizing $\til{M_T}$. Equip each $\RR_v$ with a product pseudometric that is zero on the first factor $S^1$ and agrees with the metric on $T_v$ on the second. This replacement of a product metric by a pseudo-metric is called
 partial electrification in \cite{mahan-reeves} and in the specific context of Margulis tubes, it is called tube-electrification in \cite{mahan-split}.
 The resulting pseudometric on $M_T$ is denoted as $\dt$. The main Theorem of the paper is the following (see Theorem \ref{maintech}):

\begin{theorem}\label{tel-intro} Given $R$, there exists $\delta$ such that  if $i: T \to \CC(S)$ is  an $L-$tight $R-$thick tree of non-separating curves, then 
	$\tmtdt$ is $\delta-$hyperbolic. Further, $\tmtdw$ is strongly $\delta-$hyperbolic relative to the collection $\til \RR$ of lifts  of Margulis risers.
\end{theorem}

A coarse version of the model theorem of \cite{minsky-elc1, minsky-elc2} would say that there exists $\delta > 0$ such that  the model geometry $M$ corresponding to any doubly degenerate hyperbolic manifold satisfies the property that $\til M$ is $\delta-$hyperbolic. Let $\til \TT$ denote the collection of lifts of models of Margulis tubes to $\til M$. It follows that  $\til M$ is {\it  strongly $\delta-$hyperbolic relative to} the collection $\til \TT$ (see Definition  \ref{def-srh}). The second statement of Theorem \ref{tel-intro} above generalizes this statement to the coarse model $\mtdw$ for bundles over  tight trees. In fact, the hypothesis on non-separating curves can be removed completely (Corollary \ref{effectivehypliptreecor}) for the second statement. The first statement of Theorem \ref{tel-intro} is  finer and captures one of the parameters of model Margulis tubes, viz.\ the imaginary coefficient of  model Margulis tubes in \cite{minsky-elc1, minsky-elc2}). For this statement,  the hypothesis on non-separating curves can be relaxed somewhat (see Definition \ref{def-balancedtree} and Theorem \ref{maintech}) but cannot be removed altogether (see the examples in Section \ref{sec-imb}).

 \noindent {\bf Steps of the proof and technical issues:} 
Theorem \ref{tel-intro} is an
 effective hyperbolization theorem for surface bundles over trees. The broad strategy is as follows:
 
\begin{enumerate} 
	\item First, a geometric model is constructed for the  bundle $M_T$ over  $T$ with fiber $S$ (see the discussion before Theorem \ref{model-str-intr} for a summary).
	\item For any bi-infinite geodesic $l$ in $T$, we would have liked
	 to show that the restriction $M_l$ of the  bundle $M_T$ to $l$  is uniformly bi-Lipschitz to the combinatorial model of \cite{minsky-elc1} for a doubly degenerate hyperbolic 3-manifold. This is not quite true and the construction needs to be modified (see Item (1) of Theorem \ref{model-str-intr} above for a precise statement). We think of $M_l$ as a  bundle over a line.
\item Use the converse to the Bestvina-Feighn combination theorem to extract	 effective and uniform flaring constants for the bundles $M_l$ over  lines.
\item 	Feed the  uniform flaring constants  back into the bundle over $T$ to obtain effective hyperbolization.
\end{enumerate}

A number of   difficulties arise in making the above strategy work as stated. We have already mentioned the consistency check that needs to be done when tight geodesics $\gamma_1,$  $\gamma_2$ cross at a vertex $v $. 
We briefly elaborate on the  difficulty alluded to in Item (2) above. In the case we shall be most interested in this paper, the vertex $v$ will give rise to Margulis tubes $\T_1, \T_2$ in the doubly degenerate manifolds $M_1, M_2$ corresponding to $\gamma_1,$  $\gamma_2$. It turns out that 
gluing the Margulis tubes $\T_1, \T_2$, even partially, in  $M_1, M_2$ to construct a hyperbolic model over $\gamma_1 \cup \gamma_2$ is simply not possible. We can nevertheless partially glue the boundaries $\partial \T_1, \partial \T_2$.
 The precise process involved is a certain welding construction introduced by the author in \cite{mahan-split}. This construction, however,
 gives rise to flat strips obstructing effective and uniform hyperbolization of the bundle as mentioned before Theorem \ref{tel-intro}. To circumvent this, we tube-electrify the Margulis risers
  to finally obtain a uniformly hyperbolic pseudometric. \\

{\bf Outline of the paper:}
 In Section \ref{sec-tree}, we introduce the notion of tight trees in $\CC(S)$ and show that such trees $T$ are necessarily isometrically embedded.
 For links of vertices in $T$, we describe a blowup construction: we replace a small neighborhood of a  vertex $v$ by an associated finite tree called a tree-link $T_v$. The topological building blocks for the model we construct later in the paper are of the form $M_v=S \times T_v$. The blown up tree is denoted as $\but$.
 
 A geometric structure for the building blocks  $M_v$ is introduced in Section \ref{sec:bb}. Motivated by the model geometries of doubly degenerate hyperbolic 3-manifolds constructed by Minsky \cite{minsky-bddgeom,minsky-jams,minsky-elc1} and adapted in \cite{mahan-amalgeo,mahan-split} we describe the model geometry of $M_v$.  Assembling these together give us a metric $\dw$ on the bundle $M_T$ over the blowup $\but$ of $T$. An auxiliary partially electrified version $\dt$ of $\dw$ is also defined here.

 In Section \ref{sec:prelims} we recall and adapt some basic technical tools that we require for the proof of Theorems \ref{model-str} and \ref{tel-intro} (see Theorem  \ref{maintech}). We describe
   an effective version of the Bestvina-Feighn combination theorem and its converse for hyperbolic spaces. We also describe relatively hyperbolic analogs.

  Uniform  hyperbolicity 
    of $\tmtdt$ is established in Section \ref{sec-minsky}.

 \section{Trees in the Curve Complex}\label{sec-tree} The aim of this section is twofold:
 \begin{enumerate}
 \item To use subsurface projections \cite{masur-minsky2} to give a sufficient condition for an isometric embedding of a tree $T$ in the curve complex $\CC(S)$ (Lemma \ref{ltightimpliesgeod}, Propositions \ref{isometrictighttree} and \ref{isometrictighttree-sep}).
 \item To describe the topological structure of building blocks $M_v$ corresponding to vertices $v$ of $T$. The main point here is to construct a blow up of the vertex $v$ to a finite tree $T_v$, called the {\bf tree-link} of $v$,  and hence a blown-up tree $\but$ from $T$. 	 
 \end{enumerate}

A remark on notation. We shall use $MCG(S)$ to denote the mapping class group of a closed surface $S$ and $Mod(S)$ to denote its moduli space.
 
 \subsection{Subsurface projections}\label{ssp}
 
 	The {\bf complexity} of a surface $Y$ of genus $g$ with $b$ boundary components is given by
 	\[ \xi(Y) = 3g + b - 3.\]
 	
 	The curve complex of $Y$ is denoted as $\CC(Y)$ and the arc-and-curve complex of $Y$ is denoted as $\AAA\CC(Y)$.  There is a coarsely defined $2$--Lipschitz retraction $\psi_Y$ from $\AAA\CC(Y)$ to $\CC(Y)$, given by performing surgery using boundary curves  \cite[Lemma 2.2]{masur-minsky2}. In particular for any arc $a\in \AAA\CC(Y)$, $\d_Y(\psi_Y(a)) \leq 2$.
 
 \begin{defn} \label{def-subsurfproj}
 	Let $Y\subset S$ be an essential proper subsurface.   If $\gamma\in \CC(S)$ can be homotoped to be disjoint from $Y$, define $\pi_Y(\gamma) = \emptyset$.  If $\gamma$ is homotopic to an essential curve in $Y$, then $\pi_Y(\gamma) = \gamma$.
 	Else homotope $\gamma$  to intersect $\partial Y$ minimally.  Then $\gamma\cap Y$ is a set of vertices of a simplex in $\AAA\CC(Y)$.  Define \[\pi_Y(\gamma) = \bigcup_i \psi_Y(a_i).\]  
 \end{defn}

 Definition \ref{def-subsurfproj} can be easily extended to laminations. For $\LL$ a geodesic lamination on $S$, and $Y$ an essential subsurface of $S$, let $\LL|_Y=\LL\cap W$. Then $\LL|_Y$ gives an element of the arc-and-curve complex $\AAA\CC(Y)$ after identifying the arcs and closed curves of $\LL|_Y$  with their relative isotopy classes (see \cite[Section 2.2]{minsky-bddgeom} for instance). By performing surgery on the arcs along boundary components of $Y$ we obtain elements of $\CC(Y)$. Hence $\pi_Y (\LL)$ may be defined as in Definition \ref{def-subsurfproj}.
 \begin{defn}\label{def-dists}
 	Let $Y\subset S$ be an essential subsurface with $\xi(Y) > 1$. For any collection of vertices $\VV$ of $\CC(S)$, define
 	\[ \d_Y(\VV) = \d_Y\left(\bigcup\{\pi_Y(v)\vert v\in \VV\}\right).\]
 	If $F$ is a subgraph of $\CC(S)$, define \[\d_Y(F) = \d_Y(F^{(0)}).\]
 		Finally, for $X, Z$ proper essential subsurfaces of $S$, define \[\d_Y(X,Z)= \d_Y(\partial X,\partial Z).\]
 \end{defn}
 The same definitions work for laminations.
 
 \begin{theorem}\label{bgit} {\bf (Bounded geodesic image theorem)} \cite[Theorem 3.1]{masur-minsky2}, \cite[Corollary 1.3]{webb}
 	There exists $M>0$ satisfying the following.
 	Let $S$ be a surface of finite type, and let $Y\subset S$ be an essential subsurface of complexity at least $2$.  Let $\gamma$ be a (finite or infinite) geodesic segment in $\CC(S)$ so that $\pi_Y(v)\neq \emptyset$ for every vertex $v$ of $\gamma$.  
 	Then $\d_Y(\gamma)\leq M$.
 \end{theorem}
Note that $M$ in Theorem \ref{bgit} above is a universal constant.

\begin{theorem}\label{bi} {\bf (Behrstock Inequality) \cite{behrstock-asymptotic}:}
	For a surface $S$ of finite type, there exists $D\ge 0$ such that for any three essential subsurfaces $X, Y, Z$ of $S$,
	\[\min\{\d_Y(X,Z), \d_Z(X,Y) \} \leq D.\]
\end{theorem}

Two essential subsurfaces $Y, Y'$ of $S$ are said to fill $S$ if there exists no simple closed curve in $S$ that can be homotoped off $Y$ as well as off $Y'$. 
For multicurves $v, w$ on $S$, the subsurface filled by $v, w$ is denoted as $F(v,w)$. We adapt the notion of a tight sequence from  \cite{masur-minsky2}
below (and caution the reader that what we call a tight geodesic here is referred to as a tight sequence in  \cite{masur-minsky2}).

\begin{defn} 
A sequence of multicurves $\{v_i\}$ is said to be a {\bf tight geodesic}   if:
\begin{enumerate}
	\item for any simple closed curve $\alpha_i\in v_i$ and $\alpha_j\in v_j$, $d_{\CC(S)} (\alpha_i, \alpha_j) = |i-j|$.
	\item $v_i = \partial F(v_{i-1},v_{i+1})$.
\end{enumerate}
\end{defn}

We shall now furnish a sufficient condition for proving that a sequence of multicurves is a  tight geodesic.
We are grateful to Ken Bromberg for telling us a proof of the following:
\begin{lemma}\label{ltightimpliesgeod} There exists $L\geq 3$ such that the following holds.\\
	Let $v_0, \cdots, v_n$ be a sequence of multicurves in $\CC(S)$ such that
	\begin{enumerate}
		\item For all $i$, there exists an essential subsurface $Y_i$ of $S$ such that $\partial Y_i = v_i$ and $v_{i-1},v_{i+1} \subset Y_i$.
		\item $d_{\CC(Y)} (v_{i-1},v_{i+1}) \geq L $. 
	\end{enumerate}
	Then $\{v_0, \cdots, v_n\}$ is a tight geodesic.
\end{lemma}
\begin{proof} 
	First, observe that since $L\geq 3$, it follows that $\partial Y(v_{i-1},v_{i+1}) = v_i$, i.e.\ $v_{i-1},v_{i+1}$ fill $Y_i$. Choosing $L\geq 5$, it follows that for any simple closed curves $\sigma_{i-1} \in v_{i-1}$ and $\sigma_{i+1} \in v_{i+1}$, $\sigma_{i-1},  \sigma_{i+1} $ fill $Y_i$.  It suffices therefore to prove that for simple closed curves $\sigma_i \in v_i$, $\{\sigma_0, \cdots, \sigma_n\}$ is a geodesic.
	Choose $L \geq 4D +1$ (where $D$ is as in Theorem \ref{bi}).
	
	Recall the notation of Definition \ref{def-dists}.	
	We now use the Behrstock inequality Theorem \ref{bi} to show that if $i<j<k$ then $d_{Y_j}(v_i, v_k)$ is uniformly coarsely
	equal to $d_{Y_j}(v_{j-1},v_{j+1})$. More precisely for $D$  as in Theorem \ref{bi}, for $i<j<k$, 
	\[ |d_{Y_j}(v_i, v_k)-d_{Y_j}(v_{j-1},v_{j+1}) | \leq 2D.\] 
	
	We  argue by induction. Assuming that the statement is true for $k\le m$  we shall show that if $i<j<m+1$ then the statement holds. By induction $d_{Y_{j-1}}(v_i, v_j) $ is coarsely (up to an additive $2D$) equal to $d_{Y_{j-1}}(v_{j-2}, v_j) \geq L$ (by hypothesis).  Hence $d_{Y_{j-1}}(v_i, v_j) \geq L-2D \geq 2D + 1$. By Theorem \ref{bi} this means that $d_{Y_j}(v_{i}, v_{j-1})$ is uniformly small, bounded by $D$ (for $i=j-1$ this is trivial.) Similarly we have $d_{Y_j}(v_{j+1}, v_{m+1})$ is  uniformly small, bounded by $D$. Hence by the triangle inequality if $i<j<m+1$, \[|d_{Y_j}(v_{i}, v_{m+1}) - d_{Y_j}(v_{j+1}, v_{j-1})| \leq 2D.\] This proves the claim by induction.
	
	\begin{claim}\label{claim-fill}
		$Y_i$ and $Y_j$ fill if $i\neq j$.
	\end{claim}	
	We complete the proof modulo this claim. Choose $L > 2M$, where $M$ is as in the Bounded Geodesic Image Theorem.
	By the Bounded Geodesic Image Theorem, for every $0<i<n$ any geodesic between $v_0$ and $v_n$ must pass through a curve $\eta_i$ that does not intersect $Y_i$. 
	By Claim \ref{claim-fill}, $\eta_i$ intersects $Y_j$ for all $j \neq i$.
	Hence $\eta_i \neq \eta_j$ for all $i\neq j$ and hence any geodesic between $v_0$ and $v_n$ has length $n-1$. This implies that the original sequence $\{v_0, \cdots, v_n\}$ is a tight geodesic. 
\end{proof}

\noindent {\bf Proof of Claim \ref{claim-fill}:}
To show that any two $Y_i$ and $Y_j$ fill, we first observe that since $\partial Y_i$ is contained in $Y_{i+1}$ and $\partial Y_{i+1}$ is contained
in $Y_i$, $Y_i$ and $Y_{i+1}$  fill $S$. 

Next assume $i+1<j$. If $Y_i$ and $Y_j$ do not fill there is a curve $c$ disjoint from both $Y_i$ and $Y_j$.
In particular, $c$ is contained in $Y_{i+1}$ (since $Y_i$ and $Y_{i+1}$  fill $S$).
Hence the (subsurface) projections of both $v_i$ and $v_j$ to $Y_{i+1}$ will be disjoint from
$c$ (as proper arcs) or at a uniformly bounded distance from $c$ (if we turn them into curves a la Masur-Minsky). This contradicts $d_{Y_{j+1}}(v_i,v_j) \geq L$. \hfill $\Box$\\

\subsection{Tight trees}\label{sec-tightree} Let $S$  be a surface of finite type and $\CC(S)$  its curve-complex. 
The collection of simplices in $\CC(S)$ will be denoted as $\ccd(S)$. For a tree $T$, the set of vertices of $T$ will be denoted as $V(T)$. We generalize the notion of a tight geodesic to an isometric embedding of a tree as follows:

\begin{defn}\label{def-isomembedding}
	For any geodesic (finite, semi-infinite, or bi-infinite) $\gamma =\{\cdots, v_{-1}, v_0, v_1, \cdots\}$ in $T$, and a map $i:\vT \to \ccd(S)$, a choice of  simple closed curves $\sigma_i \in i(v_i)$ will be called a path in $\CC(S)$ induced by $\gamma$.
	
	A  map  $i: V(T) \to \ccd (S)$ will be called an isometric embedding if any path induced in $\CC(S)$ by a geodesic $\gamma$ in $T$ is a geodesic in $\CC(S)$.
\end{defn}

Much of the discussion in this subsection and Section \ref{sec-topbb} gets simplified if we assume that we are dealing with a sequence of simple non-separating curves. We therefore define  this special case first.

\begin{defn}\label{def-tighttree}
An {\bf $L-$tight tree of non-separating curves} in the curve complex $\CC(S)$ consists of a (not necessarily regular) simplicial tree $T$ of bounded valence and a simplicial map $i: T \to \CC (S)$ such that 
 for every vertex $v$ of $T$ and for every pair of distinct vertices $u \neq w$ adjacent to $v$ in $T$,
$$d_{\CC(S \setminus i(v))} (i(u), i(w) \geq L.$$

An $L-$tight tree of non-separating curves for some $L\geq 3$ will simply be called a tight tree of non-separating curves.
\end{defn}

The proof of Lemma \ref{ltightimpliesgeod} immediately gives us the following (Chris Leininger first told us the proof of this special case  of Lemma \ref{ltightimpliesgeod}):

\begin{prop}\label{isometrictighttree}
	There exists  $L\geq 3$ such that the following holds.  Let $S$ be a closed surface, and let $i: T\to \CC(S)$ define an $L$--tight tree of non-separating curves.  Then $i$ is an isometric embedding.
\end{prop}

 We now extend the above definition  to allow the possibility of 
multicurves,
as well as
separating curves.

\begin{defn}\label{def-tighttree-sep}
	An {\bf $L-$tight tree} in the curve complex $\CC(S)$ consists of a (not necessarily regular) simplicial tree $T$ of bounded valence and a  map $i: V(T) \to \ccd (S)$ such that 
	\begin{enumerate}
		\item 
		for every vertex $v$ of $T$,
		$S\setminus i(v)$ consists of exactly one or two components. Further, if $S\setminus i(v)$ consists of  two components and $i(v)$ contains more than one simple closed curve, then each component of $i(v)$ is individually non-separating. In this situation, $v$ is called a {\bf separating vertex} of $T$.
		\item for every pair of adjacent vertices $u \neq v$   in $T$, and any vertices $u_0, v_0$ of the simplices $i(u), i(v)$ respectively,
		\[d_{\CC(S)} (u_0, v_0) =1.\]
		\item There is a {\bf distinguished component} $Y_v$ of $S\setminus i(v)$ such that for any vertex $u$ adjacent to $v$ in $T$, $i(u) \subset Y_v$ (automatic if $i(v)$ is non-separating).
		For $i(v)$ separating, we shall refer to $Y_v':=S \setminus Y_v$ as the {\bf secondary} component for $v$.  
		\item  for every pair of distinct vertices $u \neq w$ adjacent to $v$ in $T$, and any vertices $u_0, w_0$ of the simplices $i(u), i(w)$ respectively,
		\[d_{\CC(Y_v)} (u_0, w_0) \geq L.\]
	\end{enumerate}	
	An $L-$tight tree for some $L\geq 3$ will simply be called a tight tree.
\end{defn}

Lemma \ref{ltightimpliesgeod} gives us the following generalization of 
Proposition \ref{isometrictighttree}:
\begin{prop}\label{isometrictighttree-sep}
	Let $L> \max{(2M, 4D)}$, where $M$ is the constant from the Bounded Geodesic Image Theorem and $D$ is the Behrstock constant from Theorem \ref{bi}.  Let $S$ be a closed surface of genus at least $2$, and let $i: V(T)\to \ccd(S)$ define an $L$--tight tree (as in Definition \ref{def-tighttree-sep}).  Then $i$ is an isometric embedding.
\end{prop}

\begin{proof} It suffices to show 
	(cf.\ Definition \ref{def-isomembedding}) that
	any path induced in $\CC(S)$ by a geodesic $\gamma$ in $T$ is a geodesic in $\CC(S)$. But this last statement follows immediately from Lemma \ref{ltightimpliesgeod}.
\end{proof}

\begin{assume}\label{assumption}
	We shall henceforth assume throughout the paper that whenever we refer to an $L-$tight tree, $L> \max{(2M, 4D)}$ as in the hypothesis of Proposition \ref{isometrictighttree-sep}.
\end{assume}

\subsection{Topological building blocks from links} \label{sec-topbb} In this subsection, we shall first describe a construction of building blocks from a tight tree of non-separating curves motivated by Minsky's construction in \cite{minsky-elc1}. We shall then proceed to indicate the modifications necessary for more general tight trees. In this section, we shall describe only the topological part of the construction, postponing the geometric aspect of it to  Section \ref{sec:bb}.

For $(X,d)$ a hyperbolic metric space, and $\VV\subset X$,
$CH(\VV)$ will denote the union of all geodesics joining $v_i, v_j \in \VV$ and will be called the {\bf weak convex hull} of $\VV$.

Let $i: T \to \CC(S)$ be  a tight tree of non-separating curves and let  $v$ be a vertex of $T$. The link of $v$ in $T$  is denoted as $lk(v)$. Then $i(lk(v))$ consists of a uniformly bounded number of vertices in $\CC(S)$ (since $T$ has bounded valence). Let $m_T$ denote this bound. 

Since $S$ is fixed, there exists $\delta_0 >0$ such that for any essential connected subsurface $W$ of $S$, $\CC(W)$ is $\delta_0-$hyperbolic. In fact, there is a universal $\delta_0\leq 17$ independent even of $S$ \cite{slim-unicorns}, but we shall not need this.   It follows that for any essential connected subsurface $W$ of $S$ and any collection $\VV = \{v_1, \cdots, v_k\}$ of $k\leq m_T$ vertices  of  $\CC(W)$, there exists a finite tree $T_\VV \subset \CC(W)$ uniformly approximating  $CH(\VV)$, i.e.\ there exists a surjective map $\P: CH(\VV) \to T_\VV$ such that
\begin{enumerate}
	\item the pre-image of any point in $T_\VV$ under $\P$ has diameter uniformly bounded by $(2\delta_0 +1) m_T$ (the exact constant is not important; it will suffice for our purposes to have a uniform bound in terms of $\delta_0$ and $m_T$).
	\item $d_{\CC(W)}(v_i,v_j) = d_{T_\VV}(\P(v_i), \P(v_j))$.
	\item The vertices $\{\P(v_i)\}$ are precisely the extremal/leaf vertices of ${T_\VV}$, i.e.\ ${T_\VV}$ is precisely the convex hull of the collection of points $\{\P(v_i)\}$ in ${T_\VV}$.
\end{enumerate}

Note that the tree $T_\VV$ constructed from $\VV$ is not unique, but only coarsely so, in the sense that any two such trees are uniformly quasi-isometric to $CH(\VV)$ by maps taking $\Pi(v_i) $ to $v_i$. 

In the light of Proposition \ref{isometrictighttree} we define:

\begin{defn}\label{def-treelink}
	For a tight tree  $i: T \to \CC(S)$ of non-separating curves,
	there exists $k\geq 1$ such that for all $v \in T$ there exists a tree $T_v$ (by the above discussion) satisfying the following: \\For  $W=S\setminus i(v)$, there exists a surjective $k-$quasi-isometry $$\P_W :CH(i(lk(v))  \to T_v,$$ where $CH(i(lk(v))$ denotes the weak convex hull  of  $i(lk(v))$ in $\CC(W)$.
	
	We shall refer to $T_v$ as the {\bf tree-link} of $v$.
\end{defn}

\begin{defn}\label{def-topbb} Let $i: T \to \CC(S)$ be  a tight tree of non-separating curves and let  $v$ be a vertex of $T$.
	The {\bf topological building block corresponding to $v$ } is $$M_v =S \times T_v.$$
\end{defn}

Thus the topological building block corresponding to $v$ is the trivial $S-$bundle over its tree-link. Note that $M_v$ contains a distinguished `annulus' $i(v)\times  T_v$, where, as before, $i(v)$ is identified with a non-separating simple closed curve on $S$. We shall refer to $i(v)\times  T_v \subset M_v$ as the {\bf Margulis riser} in $M_v$ or simply as the Margulis riser corresponding to $v$. The reason for this terminology will become clearer when we describe the geometric structure on $M_v$.  

In order to assemble the building blocks  corresponding to vertices together, we shall need an auxiliary `blow-up' construction of the tree $T$. We pass to the first barycentric subdivision $S_1(T)$ of $T$ and label the mid-point of an edge in $T$ joining $v_i, v_j$ by $v_iv_j$. These vertices will be referred to simply as the mid-point vertices of $S_1(T)$.
For each vertex $v$ of $T$ we define the {\bf half-star} 
$\hs (v) \subset S_1(T)$ of $v$ to be the  (usual) star of $v$ in $S_1(T)$. 

\begin{defn}\label{def-blowup} 
	Let $i: T \to \CC(S)$ be  a tight tree of non-separating curves.
	The {\bf blow-up} $\BU (T)$ of $T$ is a tree obtained from $S_1(T)$ by replacing each half-star $\hs(v)$ by the tree-link $T_v$. 
	
	More precisely, we proceed in two steps:\\ First, attach for each $v$, the  metric tree-link $T_v$ to $S_1(T)$ by gluing $\P(v_i)$ to the mid-point vertex $v_iv$ as $\P(v_i)$ ranges over all the terminal vertices  of $T_v$. In the second step remove the interiors
	of the half-stars $\hs(v)$ from $S_1(T)$ for all $v\in V$.
	
	We retain the labels of the mid-point vertices of  $S_1(T)$  in $\BU (T)$ and refer to them as the mid-point vertices of  $\BU (T)$.
\end{defn}

The topological model for the tight tree is obtained by gluing together the topological building blocks $M_v$  corresponding to $v$ according to the combinatorics of the blow-up $\BU (T)$. Topologically this is simply the product:
\begin{defn}\label{def-topmodeltree}
	Let $i: T \to \CC(S)$ be  a tight tree of non-separating curves.
	The {\bf topological model corresponding to  $T$ } is $$M_T =S \times \BU(T).$$
\end{defn}

Let $P: M_T  \to \BU(T)$ denote the natural projection. 
Note that $\BU(T)$ has distinguished finite subtrees corresponding to the tree-links $T_v$. We also identify  $P^{-1}(T_v)$ with $M_v$. Note that a mid-point vertex $vw$ in $\BU(T)$ is the intersection of the tree-links of $v,w$: $$\{vw\} = T_v \cap T_{lk(i(w))}.$$ We shall denote $P^{-1}(vw)$ by $S_{vw}$ and refer to them as {\bf mid-surfaces}.

\subsection{Balanced trees} We now indicate the modifications necessary for a general tight tree. Let $i: V(T) \to \ccd(S)$ be  a tight-tree. Tree-links $T_v$ are defined as in Definition \ref{def-treelink} with the understanding that for $i(v)$ separating, the weak convex hull $CH(\ilkv)$ is constructed in the curve complex $\CC(Y_v)$ of the distinguished component $Y_v$ of $S\setminus i(v)$. It remains to construct tree-links for the secondary component $Y_v'$ when $i(v)$ is separating. To construct tree-links for the secondary component $Y_v'$ we need to restrict the class of tight trees we are considering. 

For $w$ adjacent to $v$ let $T_w'$ denote the connected component of $T\setminus \{v\}$ containing $w$. Let $\Pi'_v (T_w')$ denote the subsurface projection of $i(V(T_w'))$ onto $\CC(Y_v')$.

\begin{defn}\label{def-balancedtree}
	A tight tree  $i: V(T) \to \ccd(S)$ is said to be a {\bf balanced} tree with parameters $D,k$ if 
	\begin{enumerate}
		\item For every separating vertex $v$ of $T$, \[\d(\Pi'_v (T_w')) \leq D.\]
		\item Let $\ilkv' \subset \CC(Y_v')$ denote the collection of  curves $w_0 \in \Pi'_v (T_w')(\subset \CC(Y_v'))$ as $w$ ranges over all vertices adjacent to $v$ in $T$. Let $CH(\ilkv')$ denote the weak convex hull of $\ilkv' $ in $\CC(Y_v')$. We require that there exists a  surjective $k-$quasi-isometry \[\P' : CH(\ilkv') \to T_v\] to the tree-link $T_v$ such that for a vertex $w$ of $T$ adjacent to $v$, \[\P'( \Pi'_v (T_w')) = \P (w),\] (where $\P$ is the projection defined in Definition \ref{def-treelink}).
	\end{enumerate} 
	
\end{defn}

\noindent {\bf Building blocks for balanced trees:} The notions of topological building block (Definition \ref{def-topbb}) in the case of balanced trees, blow-up (Definition \ref{def-blowup}), topological model (Definition \ref{def-topmodeltree})  now go through exactly as before.
	The notion of a balanced tree (Definition \ref{def-balancedtree}) ensures that the weak convex hulls $CH(\ilkv) \subset \CC(Y_v)$ and $CH(\ilkv') \subset \CC(Y_v')$ are coarsely quasi-isometric to each other and to the tree-link $T_v$.

\section{Geometry of building blocks}
\label{sec:bb}
The purpose of this section is to construct a model geometry on the topological building blocks $M_v$ (Definition \ref{def-topbb}) and the  topological model $M_T = S \times \BU (T)$ (Definition \ref{def-topmodeltree}) corresponding to a tight tree of non-separating vertices, and more generally for a balanced tree $T$ (Definition \ref{def-balancedtree}). 

\subsection{Model geometries of doubly degenerate 3-manifolds}\label{sec:model3mfld}
It will be convenient to recall some model geometries on doubly degenerate 3-manifolds as these form the motivation and the background for the model geometry on $M_v$.

\subsubsection{A quick summary}\label{sec:summarydm}
 
 \begin{ing}\label{3ings} The model geometry on doubly degenerate 3-manifolds $M$ that is relevant to that on the topological building block $M_v$ is built from the following ingredients:
 \begin{enumerate}
\item the general combinatorial model in \cite{minsky-elc1} built from the hierarchy machinery of tight geodesics and hierarchy paths \cite{masur-minsky,masur-minsky2};
\item the  model for bounded geometry doubly degenerate 3-manifolds built from a thick Teichm\"uller geodesic in \cite{minsky-jams}.
\item The main theorem of \cite{minsky-bddgeom}  establishing a combinatorial model for bounded geometry doubly degenerate 3-manifolds along with a dictionary between the combinatorics of such a model and the model geometry from a a thick Teichm\"uller geodesic in Item (2) above.
\end{enumerate}
 \end{ing}
We  briefly describe these 3 ingredients in this section and use them in Definition \ref{def-splsplitcombin} to define the geometry that will lead to the model geometry on $M_v$. As usual $\Teich(S)$ will denote the Teichm\"uller space of $S$. \\ 

\noindent {\bf  Item(1): The combinatorial model of \cite{minsky-elc1}.} The general combinatorial model on a doubly degenerate 3-manifold $M$  in \cite{minsky-elc1} is built as follows. Let $\LL_\pm$ denote the ending laminations of $M$. Identify $\LL_\pm$ with a pair of points on the boundary $\partial \CC(S)$ of the curve complex (using Klarreich's theorem \cite{klarreich-el}). Let $\gamma \subset \CC(S)$ be a tight geodesic joining $\LL_\pm$. The hierarchy path joining $\LL_\pm$ is built inductively by \cite{minsky-elc1,masur-minsky2}. One starts with $\gamma$ as the base geodesic. For every vertex (or simplex) $v$ in $\gamma$, one (roughly speaking) constructs geodesics in $\CC(S\setminus \{v\})$ joining the predecessor of $v$ to its successor. This process is repeated inductively for the geodesics constructed at this second stage and so on. The general combinatorial model in \cite{minsky-elc1} is built from standard building blocks ($S_{0,4} \times I$ or $S_{1,1} \times I$ equipped with some standard metrics) by assembling them according to the combinatorics 
dictated by the hierarchy path joining  $\LL_\pm$.\\

\noindent {\bf  Item (2): the model for bounded geometry 3-manifolds $M$ with ending laminations $\LL_\pm$.} Recall that $M$ is homeomorphic to $S \times \R$. We first replace the laminations  $\LL_\pm$ on $S$ by  singular foliations $\FF_\pm$ (differing from $\LL_\pm$ by bounded homotopies) and equip $S$ with a singular Euclidean metric where the $x-$(resp. $y-$)co-ordinate is given by $\FF_+$ (resp. $\FF_-$). Note that fixing co-ordinates implicitly converts $\FF_\pm$  into {\it measured}  singular foliations. Then the model geometry on $M$ is locally given by a singular Sol-type metric (see \cite{CTpub} or \cite[p. 567]{minsky-jams}) \[ds^2 = e^{2t}dx^2 +e^{-2t}dy^2 + dt^2,\] where $t$ parametrizes the $\R-$direction in 
$M=S \times \R$. So far we have not used the bounded geometry hypothesis. There exists a more canonical parametrization of the $\R-$direction when $M$ has bounded geometry. In \cite{minsky-top,minsky-jams}, Minsky showed that when $M$ has bounded geometry the Teichm\"uller geodesic $\gamma$ in $\Teich(S)$ joining $\FF_\pm \in \partial \Teich(S)$ is thick, i.e.\ it projects to a geodesic lying inside a compact region in moduli space $Mod(S)$:

\begin{defn}\label{def-thickgeod}
A geodesic $\gamma$ in $\Teich(S)$ is said to be {\bf $\ep_0-$thick} if the systole of any surface $S_x , \ x \in \gamma$ (thought of as a hyperbolic surface) is bounded below by $\ep_0$. 

A geodesic $\gamma$ in $\Teich(S)$ is {\bf thick} if it is $\ep_0-$thick for some $\ep_0 > 0$.
\end{defn}
For a thick Teichm\"uller geodesic $\gamma$, joining $\FF_\pm \in \partial \Teich(S)$ the parameter $t$ may be identified with the arc-length of the Teichm\"uller geodesic $\gamma$. 

Rafi \cite{rafi-gt} characterized thick 
Teichm\"uller geodesics in terms of subsurface projections.
To state this characterization we  recall that in Definitions \ref{def-subsurfproj} and \ref{def-dists} the notion of subsurface projections was defined. As pointed out after Definition \ref{def-subsurfproj}  these notions  can be naturally   extended to $d_Y(\lambda, \mu)$ for  laminations $\lambda, \mu$ on $S$ and arbitrary essential subsurfaces $Y$ of $S$ \cite[p. 150-151]{minsky-bddgeom}.

\begin{theorem}\cite{rafi-gt} \label{thickgeod}
	Let $\gamma$ be a bi-infinite geodesic in $\Teich(S)$ with end-points   $\LL_\pm \in \PML (S) = \partial \Teich(S)$. Then $\gamma$ is of bounded geometry if and only if there exists $D>0$ such that for every essential subsurface $W$ of $S$ (including annular domains), $d_W(\LL_+,\LL_-) \leq D$.
\end{theorem}

\begin{defn} \label{def-thickminmod}
For a  bounded geometry doubly degenerate 3-manifold $M$ without parabolics (homeomorphic to $S \times \R$)  with ending laminations $\LL_\pm$, the {\bf thick Minsky model} on $M$ is given by the singular Sol-type metric \[ds^2 = e^{2t}dx^2 +e^{-2t}dy^2  + dt^2,\] where $t$ parametrizes (according to arc length)  the Teichm\"uller geodesic $\gamma$ joining $\LL_\pm$ and $x,y$ are co-ordinates for  singular  foliations boundedly homotopic to $\LL_\pm$.

For $S=S_{g,n}$ with marked points, let $M^h$  (homeomorphic to $S \times \R$) be  a  bounded geometry doubly degenerate 3-manifold  with ending laminations $\LL_\pm$ and let $M$ denote $M^h$ minus a small neighborhood of cusps. The singular Sol-type metric $ds^2 = e^{2t}dx^2 +e^{-2t}dy^2  + dt^2$ on $S \times \R$ is given as before; but the latter contains a distinguished set of geodesics through each of the marked points $p_1,\cdots, p_n$ given by $(p_i,t)$. we refer to these as {\bf cusp geodesics}. 
\end{defn}

This will be elaborated upon in Section \ref{sec:bddgeo} below. \\

\noindent {\bf  Item (3): the relationship between the thick Minsky  model of a bounded geometry manifold $M$ as per Definition \ref{def-thickminmod} in Item (2) and its combinatorial model given in Item (1).}
The main theorem of \cite{minsky-bddgeom} establishes the necessary dictionary (note the similarity with Theorem \ref{thickgeod}).

\begin{theorem}\cite[p. 144]{minsky-bddgeom} \label{thickcombin}
Let $M$ be a doubly degenerate hyperbolic 3-manifold $M$ with ending laminations $\LL_\pm$. Then $M$ is of bounded geometry if and only if there exists $D>0$ such that for every essential subsurface $W$ of $S$ (including annular domains), $d_W(\LL_+,\LL_-) \leq D$.
\end{theorem}

We turn now to a special geometry that will be relevant to this paper.
We describe in terms of subsurface projections the conditions that define the model relevant to the geometry on the topological building block $M_v$ (Definition \ref{def-topbb}). 

\begin{defn}\label{def-splsplitcombin}
	Let $M$ (homeomorphic to $S \times \R$) be a doubly degenerate hyperbolic 3-manifold with ending laminations $\LL_\pm$. Then $M$ will be said to be of {\bf special split geometry}  with parameters $L, R$ if it satisfies the following conditions:\\
	1) Let $\gamma$ be a tight geodesic joining $\LL_\pm$ in $\CC(S)$. Then for every simplex  $v$ of $\gamma$ and every component $Y$ of $S\setminus v$,  \[d_Y(\LL_+,\LL_-) \geq L.\] We refer to the components $Y$ of $S\setminus v$, $v \in \gamma$ as {\bf principal component domains}. \\
	2) For every proper essential subsurface $W$ of $S$ that is {\em not} a principal component domain,
	 \[d_W(\LL_+,\LL_-) \leq R.\]
	 
	 Further, if each simplex of the base tight geodesic $\gamma$ is a single vertex $v$ corresponding to a non-separating simple closed curve on $S$ then $M$ is said to be of  special split geometry {\bf with non-separating curves}. 
\end{defn}

It follows from \cite[Theorem 8.1]{minsky-elc1} that for $L$ large enough, each vertex $v$ gives a Margulis tube:

\begin{lemma}\label{v-thin} For every $\epsilon_0 > 0$, 
	there exists $L_0$ such that for $L\geq L_0$ the following holds. Let $M$ be of special split geometry with parameters $L, R$ as in Definition \ref{def-splsplitcombin}. Then every vertex of the tight geodesic $\gamma$ in Definition \ref{def-splsplitcombin} gives an $\epsilon_0-$Margulis tube in $M$.
\end{lemma}

 \subsubsection{The bounded geometry model}\label{sec:bddgeo} We turn now to the second item of Ingredients \ref{3ings} and adapt it to bundles over quasiconvex subsets of $\CC(S)$ or $\Teich(S)$. Let $N^h$  be a doubly degenerate hyperbolic 3-manifold corresponding to a surface $S$ with or without punctures. Let $N$ denote $N^h$ minus a small neighborhood of the cusps. We normalize so that the boundary components of $N$ are isometric to products $S^1_e \times \R$, where $S^1_e$ are round circles of radius $e$.
 We define the {\bf systole} of a manifold or more generally a length space to be  the infimum of the length of closed geodesics (thus ignoring cusps).
 If there exists $\ep > 0$ such that the systole of $N^h$ (and hence $N$)  is bounded below by $\ep$, then $N^h$ is said to be  of {\bf bounded geometry}.
 
 The following theorem is essentially  due to Minsky
 \cite{minsky-jams} (see, however, the paragraphs following Theorem \ref{minsky-elcbddgeo} for references to the literature, from where the refinement we need can be culled). It establishes a bi-Lipschitz equivalence between the hyperbolic structure on a bounded geometry doubly degenerate hyperbolic 3-manifold and its thick Minsky model:
 
 \begin{theorem}\cite[Cor. 5.10]{minsky-jams} For $S=S_{g,n}$ a surface of genus $g$ and $n$ punctures, and $\ep >0$, let $N^h$ be  a doubly degenerate hyperbolic 3-manifold corresponding to $S$ with injectivity radius bounded by $\ep >0$, and $N$ denote $N^h$ minus a small neighborhood of the cusps as above. Then there exists $L \geq 1$ such that the following holds:\\
 	Let $\LL_\pm$ as above be the ending laminations of $N^h$, let $l$ be the bi-infinite geodesic in $\Teich(S)$ joining  $\LL_\pm$. Let $Q^h$ denote the thick Minsky model as in Definition \ref{def-thickminmod}.
  Let $Q$ denote $Q^h$ minus a small neighborhood of the cusp geodesics (Definition \ref{def-thickminmod}) with boundary components  normalized to be isometric to products $S^1_e \times \R$. Then $Q$ and $N$ are $L-$bi-Lipschitz homeomorphic.
 	\label{minsky-elcbddgeo}
 \end{theorem}

We point out that Minsky established a quasi-isometric map at the level of  universal covers that is a lift of a possibly non-injective map between $N, Q$ in the case that $N^h=N$, i.e.\ in the absence of cusps.
 We refer the reader to
\cite[Proposition 8]{lott} for the relevant refinement  in the absence of cusps. The conclusion also follows from the full strength of
the ending lamination theorem \cite{minsky-elc1, minsky-elc2} applied to this
special case. 

A word about cusps. A coarse model in the case of surfaces
with punctures may be found in \cite{bowditch-ct}, \cite[Section 1.1]{mahan-pared} or \cite[Section 5]{mahan-bddgeo}. One removes a small neighborhood of the cusps from $N^h$ to obtain $N$ as in the statement of Theorem \ref{minsky-elcbddgeo}. Then one proves the existence of a sequence of pleated surfaces in $N^h$, such that the intersections of these pleated surfaces 
with $N$ give a  sequence of equispaced pleated surfaces with boundary. 
 The main theorem of \cite{minsky-top} applies equally to the punctured surface case to show that the Teichm\"uller distance between successive
 pleated surfaces with boundary
  is uniformly bounded both above and below. However,
 as in \cite{minsky-jams}, these pleated surfaces may be immersed, and not embedded. To upgrade immersed surfaces to embedded surfaces, Lott's
 argument in \cite[Proposition 8]{lott} applies equally to the manifold with boundary $N$, concluding the proof.

 Theorems \ref{thickcombin} and \ref{minsky-elcbddgeo} thus establish two different descriptions of bounded geometry doubly degenerate hyperbolic 3-manifolds. \\

 \noindent {\bf Universal bundles:} The thick Minsky model in Definition \ref{def-thickminmod} will need to be generalized to a situation where the base space is a quasiconvex subset of $\Teich(S)$ rather than a geodesic. To do this it will be more convenient to obtain a description in terms of hyperbolic metrics on $S$ rather than the singular Euclidean metric in Definition \ref{def-thickminmod}. The natural structure is given in terms of {\bf universal bundles or universal curves} over $\Teich(S)$. The following remark recalls the necessary notion from \cite{wolpert-univcurve}.
 
 \begin{rmk}\label{ubundle}{\rm 
The moduli space $Mod(S)$ is a quasiprojective variety \cite{mumford-ens}. A finite-sheeted cover of $Mod(S)$ is actually a manifold (and also a quasiprojective variety) and can be naturally equipped with a K\"ahler metric: the Weil-Petersson metric 
 \cite[p. 420]{wolpert-univcurve}. We specialize to the case where $g\geq 2, n=1$ and denote $S=S_{g,0}$. Then $Mod(S_{g,1})$ admits a natural bundle structure fibering over $Mod(S)=Mod(S_{g,0})$ with fiber over $x\in Mod(S_{g,0})$ the curve $x$. This is called the universal curve \cite[p. 419]{wolpert-univcurve}. The cover of  $Mod(S_{g,1})$ corresponding to the fundamental group $\pi_1(S_{g,0})$ of the fiber is then called the 
 {\bf  universal bundle} $U\Teich(S)$ over $\Teich(S)$. The fiber $S_x$ over $x \in \Teich(S)$ is then the marked hyperbolic structure given by $x$. The metric induced on $S_x$ is the restriction of the Weil-Petersson metric on  $U\Teich(S)$ and equals the hyperbolic metric (up to a global scale factor).
 
 Finally, for any $X \subset \Teich(S)$ the restriction of $U\Teich(S)$ to $X$ gives topologically a product $UX= S \times X$. The  metric on $UX$ is the path-metric induced from $U\Teich(S)$ on $UX$.}
  \end{rmk}
  
   There is a natural fiberwise uniformization map $\Phi$ from the thick Minsky model to the universal curve over a thick Teichm\"uller geodesic $l$. Since the systole of every fiber $S_x, \ x \in l$ is uniformly bounded below, there exists $K\geq 1$, depending only on the lower bound on systole, such that $\Phi^{-1}$ is  $K-$bi-Lipschitz on $S_x$ for every $x$.  It follows that the universal bundle over $l$ with its metric is  bi-Lipschitz homeomorphic to the thick Minsky model under a fiber-preserving  homeomorphism.
  
  \begin{rmk}
An alternate coarse description of universal bundles for $S$ closed may be given as follows. Let $X \subset \Teich(S)$ be contained in the $\ep-$thick part $\Teich_\ep(S)$ of $\Teich(S)$, i.e.\ for every $x \in X$ the hyperbolic surface $S_x$ has systole at least $\ep$. Further, suppose that $X$ is quasiconvex (with respect to the Teichm\"uller metric). Note that the quotient $\Teich_\ep(S)/MCG(S)$ by the mapping class group is compact and hence the inclusion $MCG(S).o \subset \Teich_\ep(S)$ is a quasi-isometry where 
 $o \in \Teich_\ep(S)$ is some  base-point.
 Hence there is a subset $K \subset MCG(S)$ such that $K.o $ is quasi-isometric to $X$ with the same constants. Further, if $$1 \to \pi_1(S) \stackrel{i}\rightarrow MCG(S,*) \stackrel{q}\rightarrow MCG(S) \to 1$$ denote the Birman exact sequence, then $q^{-1}(K)$ projects  to $K$ under $q$ and there is a coarsely  fiber-preserving quasi-isometry between $q^{-1}(K)$ and the universal cover of the universal curve over $X$  (see the notion of metric bundles in \cite[Definition 1.2]{mahan-sardar} for more details). 
    \end{rmk}

  \subsubsection{Relations between model geometries}\label{sec-3models} In what follows, we shall need to go between three different geometries of doubly degenerate hyperbolic manifolds:
  
  \begin{enumerate}
  \item The hyperbolic metric.
  \item The combinatorial model \cite{minsky-elc1}.
  \item A model obtained by interbreeding the thick model of Theorem \ref{minsky-elcbddgeo} above with the combinatorial model {\it in a certain special case} that we shall amplify below. This last model will be called a {\bf special split geometry} model following \cite{mahan-split}.
  \end{enumerate}
  
 It follows essentially from the ending lamination theorem \cite{minsky-jams,minsky-elc1,minsky-elc2} that these three different geometries will give us metrics that are uniformly bi-Lipschitz to each other. We shall say more about Item (3) above below (see especially Theorem \ref{thickhierrachy}).

 \subsubsection{The special split geometry model}\label{sec-thin} We shall now proceed to elaborate on the special split geometry model given in Definition \ref{def-splsplitcombin} and provide an alternate description of the model that we shall need later. The alternate description  is culled out of \cite{minsky-elc1,mahan-split} (see especially   \cite[Sections 1.1.2, 1.1.3, 4.1]{mahan-split}, and the description of models of `graph amalgamation geometry' in \cite{mahan-amalgeo}).  
 
 A piecewise smooth embedded incompressible surface $S$ in a hyperbolic 3-manifold is said to have $(\ep, D)-$bounded geometry if, with respect to the induced path metric
 \begin{enumerate}
 	\item the systole of $S$ is bounded below by $\ep$
 	\item the diameter of $S$ is bounded above $D$.
 \end{enumerate}

 We summarize the part of the discussion in Section 4.1 of \cite{mahan-split} that will be necessary for us. Let $N$ be a doubly degenerate hyperbolic 3-manifold  of special split geometry (Definition \ref{def-splsplitcombin}) with ending laminations $l_\pm$. Let $E_\pm$ be the two ends of $N$ and $\gamma =\{ \cdots, v_{i-1}, v_i, v_{i+1}, \cdots \}$ be the tight geodesic of simplices joining $l_\pm$ occurring in Definition \ref{def-splsplitcombin}. Then Proposition 4.2 of \cite{mahan-split} gives a sequence of bounded geometry surfaces $\{S_i\}, \ i \in \Z$ exiting the ends $E_\pm$. Proposition 4.3 of \cite{mahan-split} now shows that the region between $S_i, S_{i+1}$ has a finite number of Margulis tubes corresponding to the simple closed curves occurring as vertices of the simplex $v_i$. Further, away from these Margulis tubes, the systole of $N$ is uniformly bounded away from zero. We summarize the conclusions of this construction (see p. 36 of \cite{mahan-split}) as follows.

 \begin{prop}
 	\label{splsplit} For all $ R$ there exist $\ep, C, D > 0$ such that the following holds.\\ Let $N$ be a doubly degenerate hyperbolic 3-manifold  of special split geometry with parameters $L\geq 3, \ R$ and ends $E_\pm$. Then (see figure below):
 		\begin{enumerate}
 		\item  There exists a sequence
 		$\{ S_i \}, i \in \Z$  of   disjoint,  embedded, incompressible, $\ep, D-$bounded geometry surfaces exiting  the ends $E_\pm$ as $ i \to \pm \infty$ respectively. The surfaces are ordered so that  $i<j$ implies that
 		$S_j$ is contained in the unbounded component of $E_+ \setminus S_i$. The topological product region  between $S_i$ and $S_{i+1}$ is denoted $B_i$ and is termed a {\bf split block}.
 		\item corresponding to each such product region $B_i$, there exists a finite number of  Margulis tubes  corresponding to disjoint simple closed curves on $S_i$. The disjoint union of these  Margulis tubes is called a multi-Margulis tube and denoted as $\tube_i$. Then $\tube_i \subset B_i$.  Further, 
 		$\tube_i \cap S_i$ and  $\tube_i \cap S_{i+1}$ are (multi-)annuli on $S_i$ and $S_{i+1}$ respectively, with core curves homotopic to the  core curve of $\tube_i$. We think of $\tube_i$ as splitting the $i$th {\bf split block} $B_i$ and call it a {\bf splitting tube}. The complementary components $K_{ij}$ 
 		of $B_i \setminus \tube_i$  and their lifts $\til{K_{ij}}$ to $\til N$  are called {\bf split components}. The top and bottom boundary surfaces
 		$ S_{i+1}, S_i$ of $B_i$ are called {\bf split surfaces}.
 		\item The core curves of $\tube_i$ correspond to a  simple closed multicurve $\tau_i$ on $S$.
 		\item Further 
 	 $B_i \setminus \tube_i$  has systole uniformly bounded below by $\ep$ for all $i$. 
 		\item The  geometry of the Margulis tubes $\tube_i$ is as follows. For a component of a splitting tube $\tube_i$ in a split block $B_i$,
 		the vertical boundaries $A_{i}^\pm$, corresponding to the left and right vertical annuli
 		in the figure below,    are $C-$bi-Lipschitz homeomorphic to  products, $A_{i}^\pm=S^1 \times [0,l_{i}^\pm]$, of the unit circle(a  normalization condition) and an interval of length $l_i^\pm$. 
 		The horizontal boundaries of $\tube_i$ are $C-$bi-Lipschitz homeomorphic to (each other and to) products, $S^1 \times [-e,e]$ for a fixed (small) $e$ independent of $i$. 
 	\end{enumerate}
 \end{prop}

 \begin{center}
 	
 	\includegraphics[height=6cm]{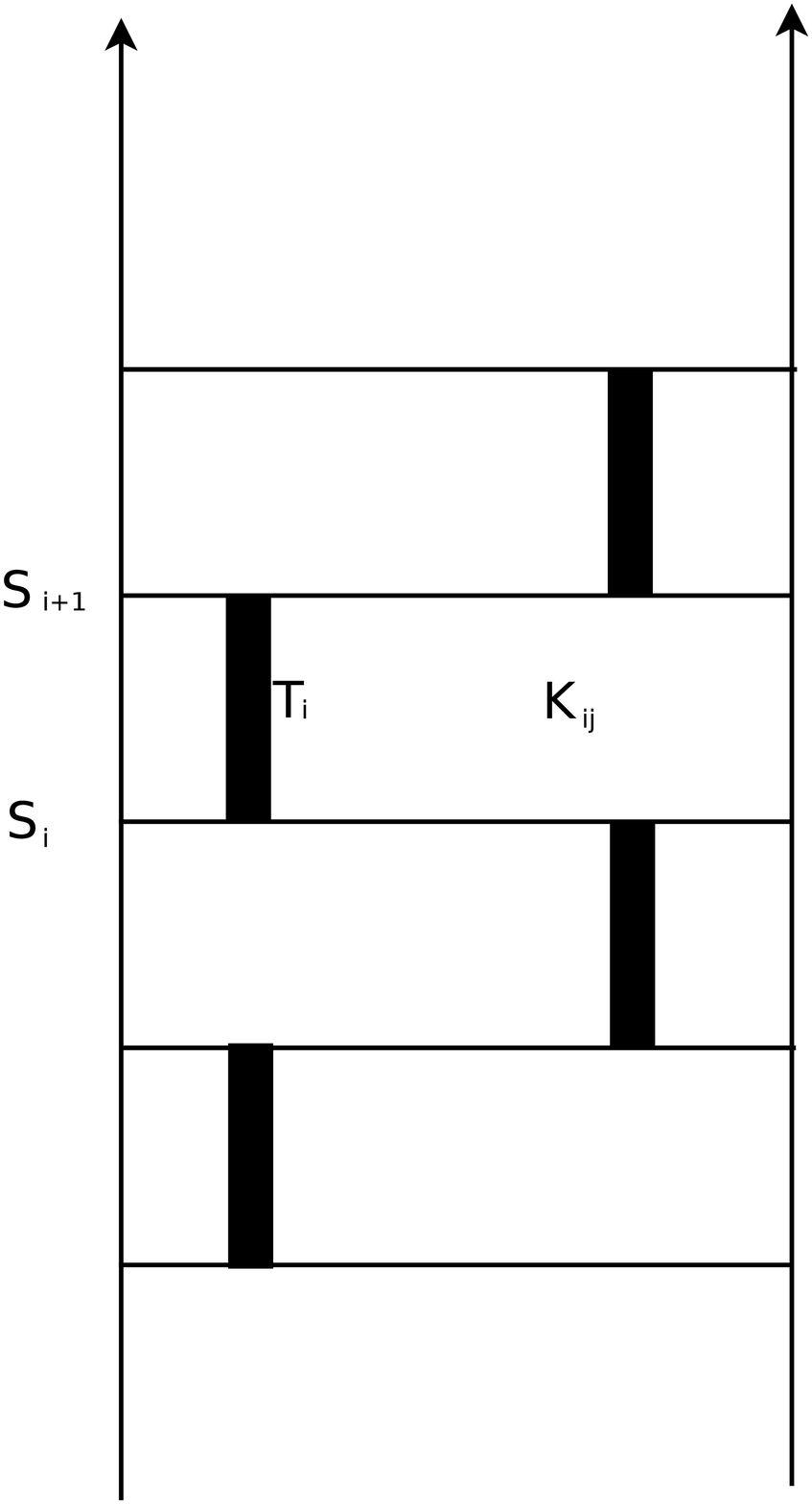}

 	\smallskip
 	
 	\underline{Figure:  {\it A schematic representation of the model geometry} }
 	
 \end{center}
 
 \begin{rmk}
 When the tight geodesic $\gamma$ of Definition \ref{def-tighttree} joining the ending laminations $\LL_\pm$ of $N$ in Proposition \ref{splsplit} consists of simple closed curves $\tau_i$, then each multi-Margulis tube $\tube_i$ is in fact a Margulis tube with core curve isotopic to $\tau_i$. 
 \end{rmk}
 
 \begin{rmk}
 The geometry of the Margulis tubes in Proposition \ref{splsplit}  really originates in the geometry of such tubes in the combinatorial model of \cite{minsky-elc1}. Using the  bi-Lipschitz homeomorphism of \cite{minsky-elc2} between the combinatorial model and the hyperbolic  metric we obtain the structure of  Margulis tubes given in Proposition \ref{splsplit}. 
 \end{rmk}
 
 \begin{rmk}
 We remark that  the general case of {\it weak split geometry} described in \cite[Remark 4.9]{mahan-split} allows for each multi-Margulis tube $\tube_i$ to split a uniformly bounded number of blocks. For special split geometry, this number is precisely one.
\end{rmk}
 
 Let $l_i = \min(l_i^+,l_i^-)$
 Let $\Phi_i^\pm : S^1 \times [0,l_i^\pm] \to S^1 \times [0, l_i]$ be maps that are identity in the first factor and affine surjective maps in the second factor.
 \begin{defn}\label{weld}\cite[p. 38]{mahan-split}
 	A {\bf welded split block} ${B_{i,weld}}$ (homeomorphic to $S \times [0,1]$) is a split block equipped with the following quotient path metric on each splitting tube:
 	\begin{enumerate}
 		\item horizontal boundaries $S^1 \times [-e,e]$ quotiented down to $S^1 \times \{ 0\}$ by projecting the second co-ordinate to $0$,
 		\item the vertical boundaries of splitting tubes are identified with each other  via the   maps $\Phi_i^\pm$. 
 	\end{enumerate}
 	The resulting annuli in ${B_{i,weld}}$ after the identification shall simply be called  {\bf standard annuli in ${B_{i,weld}}$}. The resulting metric on ${B_{i,weld}}$ will be denoted by ${d_{i,weld}}$.
 	We shall also refer to $l_i$ as the {\bf height} of the standard annulus in $B_i$, or simply the  {\bf height} of $B_i$.
 	
 	The composition of the two maps above give a quotienting map
 	$f_i: \partial T_i \to S^1 \times [0,l_i]$.
 \end{defn}

 The definition of  a welded manifold we have used here is slightly different from the one in \cite{mahan-split}, where all the $l_i$'s were equal to one.

 We shall equip ${B_{i,weld}}$ with a new pseudometric.
 Equip the standard annulus $S^1 \times [0,l_i]$  with the product of the zero metric on the $S^1$-factor
 and the Euclidean metric on the $[0,l_i]$ factor.  Let $(S^1 \times [0,l_i], d_0)$ denote the resulting pseudometric.
 
 \begin{defn}\label{tubeel}\cite[p. 39]{mahan-split}
 	The {\bf tube-electrified metric} $d_{te}$   is defined to be the pseudometric metric that
 	agrees with ${d_{weld}}$ away from the standard annuli in ${B_{weld}}$ and with $d_0$ on 
 	the standard annuli in ${B_{i,weld}}$.
 \end{defn}
 
 To distinguish it from $({B_{i,weld}}, {d_{i,weld}})$
 the new space and  pseudometric will be denoted as $(B_{i, te},d_{i, te})$.
 Note that all the top and   bottom  split surfaces of  split blocks $B_i$ (before or after tube-electrification)
 are homeomorphic to  a fixed hyperbolic $S$ via uniformly bi-Lipschitz homeomorphisms.

 Gluing successive welded blocks along common split surfaces we obtain the {\bf welded model manifold} 
 $({N_{weld}}, {d_{weld}})$ homeomorphic to
 $S \times  \mathbb{R}$ corresponding to the original doubly degenerate manifold $N$.
 
 \subsection{Model geometry of topological building blocks $M_v$}\label{sec-interbred} The purpose of this section is twofold.
 First, it furnishes an alternate explicit model geometry (the special split model geometry) for the split blocks of Proposition \ref{splsplit} by interbreeding the thick Minsky model (Theorem \ref{minsky-elcbddgeo}) with the combinatorial model of \cite{minsky-elc1}. 
 Secondly,  armed with the model geometries of bounded geometry doubly degenerate  3-manifolds (Theorem \ref{minsky-elcbddgeo}) and the special split geometry model (Proposition \ref{splsplit}), we 
 describe a model geometry (i.e.\ a metric) on the topological building blocks $M_v$ (Definition \ref{def-topbb}). The metric on $M_v$ shall be denoted as $d_v$ and the metrized building block $(M_v,d_v)$ shall be called the {\bf geometric building block}.

 \begin{rmk}\label{rmk-splsplit}{\rm 
 Suppose that the tree $T$ of Definition \ref{def-tighttree-sep} is a simplicial tree $l$ with underlying space $\R$ and with vertices at $\Z$.
   Let $l_\pm$ denote the ending laminations corresponding to the end-points of $i(l) \subset \ccd(S)$ in $\partial \CC(S)$. Let $M_l$ be the (unique up to isometry \cite{minsky-elc1,minsky-elc2}) doubly degenerate hyperbolic 3-manifold  with  ending laminations  $l_\pm$. Then $M_l$ is of special split geometry as in Proposition \ref{splsplit} Let
 $v $ be a vertex in the vertex set $\Z$.
Then
  the model geometries using $\dw, \dt$ that we describe below on $M_v$ will respectively be uniformly bi-Lipschitz to the metrics on the welded split block (Definition \ref{weld}) and the tube-electrified metric $\dt$ of Definition \ref{tubeel}.}
  \end{rmk}

 Recall that the  topological building block corresponding to $v$  is given by $M_v =S \times T_v.$

 \begin{defn}\label{def-geometricbb}
 A {\bf special split geometry} on $M_v$ with parameters $k,\ep$ is built from the following:
 \begin{enumerate}
 \item  A $k-$bi-Lipschitz section $\sigma_W: T_v \to \Teich_\ep(W)$ for  each component $W$ of $S \setminus {i(v)}$ such that $\sigma_W(T_v) $ is $k-$quasiconvex in $\Teich(W)$. Note that by Definition \ref{def-balancedtree}, $T_v$ is coarsely independent of the component $W$.
 \item The Margulis riser $i(v)\times  T_v$ corresponding to $v$ is metrized by equipping it with the product metric so that each circle of $i(v)$ is a round circle $S^1_e$ of radius $e>0$.
  \item  Let $W_v$ denote the universal metric bundle (see Remark \ref{ubundle} and the discussion following it) over  $\sigma_W(T_v)$ with a neighborhood of the cusps removed. We further demand that each annular boundary component of $W_v$ (corresponding to circular boundary components of $W$) is a metric product $S^1_e \times \sigma_W(T_v)$ (equivalently, we excise the cusps of a fiber over any $x \in \sigma_W(T_v)$ in such a way that the boundary curves are isometric to $S^1_e$).
 \item Let $A_W$ be an annular boundary component of some $W_v$ ($W$ ranges over components  of $S \setminus {i(v)}$). Then there exists a {\bf simple} closed curve $v_A \subset v$ such that $A_W$ corresponds to the Margulis riser $v_A \times T_v$ and is isometric to the metric product $S^1_e \times \sigma_W(T_v)$. We glue the annular boundary component $A_W$ to the Margulis riser $v_A \times T_v$   via the map $(Id, \, \sigma_W^{-1})$.
 \item We do this for every component $W$ of $S \setminus {i(v)}$.
 \end{enumerate}
 The  resulting quotient metric on $M_v$ is denoted $d_v$.
 $M_v$ equipped with $d_v$ will be called the {\bf  building block of special split geometry} corresponding to $v$.   The natural projection from  $(M_v, d_v)$ to $T_v$ will be denoted by $P_v$.
 \end{defn}

\begin{defn}\label{def-unifproper}
We shall say that a map $f: (A, d_A) \to (B, d_B)$ of metric spaces is 
{\bf $c-$proper} if for any  $B_1 \subset B$ of diameter at most one, $f^{-1} (B_1)$ has diameter at most $c$. If $f$ is $c-$proper for some $c$ we shall simply say that it is uniformly proper.
\end{defn}
We observe an immediate consequence of Definition \ref{def-geometricbb}.

\begin{lemma}\label{proper-block}
Given $k, \ep$, there exists $c$ such that  if $M_v$, as in Definition \ref{def-geometricbb}, is of special split geometry with parameters $k,\ep$, then $P_v: (M_v, d_v)\to T_v$ is $c-$proper.
\end{lemma}
 
 \begin{rmk}\label{onesepgeomodel}{\rm {\bf Special case of a single non-separating curve:}\\
 We describe a quick informal way of thinking about  the geometric building block $(M_v,d_v)$ when $v$ consists of a single non-separating curve, so that  $W = S\setminus v$ is connected. Here, $\sigma_W(T_v)$ is a $k-$quasiconvex tree in the $\ep-$thick part $\Teich(W)$. The metric on $W_v$ away from the cusps is the universal bundle metric over $\sigma_W(T_v)$. Thus, away from the cusps, the metric on $W_v$ is like  the bounded geometry metric given by Theorem \ref{minsky-elcbddgeo}. After excising the cusps this bundle is glued to the metric product Margulis riser  $S^1_e \times T_v$ by a map that is identity in the first co-ordinate and $\sigma_W^{-1}$ in the second.}
 \end{rmk}

 We proceed to define a tube-electrified (pseudo-)metric on $M_v$ following Definition \ref{tubeel}. 
 Equip each Margulis riser $S^1_e \times T_v$  with the product of the zero metric on the $S^1_e$-factor
 and the usual (tree) metric on the $T_v$ factor.  Let $(S^1 \times T_v, d_0)$ denote the resulting pseudometric.
 
 \begin{defn}\label{def:tubeelbb}
 	The {\bf tube-electrified metric} $d_{te}$ on $M_v$  is defined to be the pseudometric  that
 	agrees with ${d_{v}}$ away from the Margulis risers in $M_v$ and with $d_0$ on 
 	the Margulis risers in $M_v$.
 \end{defn}
 
 $M_v$ equipped with the tube-electrified metric $d_{te}$
 will be denoted as $(M_v,d_{te})$ (as in Definition \ref{tubeel}).

 An alternate description of the model geometry on $M_v$ (Definition \ref{def-geometricbb}) can be given in terms of hierarchy paths along the lines of the dictionary established by Theorem \ref{thickcombin}. We give a quick informal recapitulation following \cite{minsky-bddgeom}. Let $\MM(S)$ and $\PP(S)$ denote respectively, the marking complex and the pants complex of $S$. Fix a base-point $o \in  \Teich_e(S)$ and let $MCG(S)$ denote the mapping class group of $S$ acting on $\Teich_e(S)$. Note that  $MCG(S)$  (with respect to a word metric for a finite generating set) and  $\MM(S)$ are quasi-isometric.
 Let $\P_M: \Teich_e(S) \to \MM(S)$ denote a projection (coarsely well-defined,
see \cite{masur-minsky,masur-minsky2}) taking a point $x$ of $\Teich_e(S)$ to a nearest point $g.o$ in the mapping class group orbit $MCG(S).o$ and hence via a quasi-isometry to $\MM(S)$. Also, let $\P_C: \Teich_e(S) \to \CC(S)$ denote a projection (again coarsely well-defined,
see \cite{masur-minsky,minsky-elc1}) taking a point $x$ of $\Teich_e(S)$ to the collection of short curves (where shortness is defined by a Bers' constant). We may and will assume that $\P_C$ factors through $\P_M$.

 We shall need a slight generalization of Theorems \ref{thickgeod} and \ref{thickcombin} due to Rafi \cite{rafi-gt} and Minsky \cite{minsky-bddgeom}. Using the projection $\P_C$, subsurface projections $\pi_W (x)$ of points  $x\in \Teich(S)$ onto the curve complex $\CC(W)$ of an essential subsurface $W$ and distances $d_W(x,y)$ between $x,y\in \Teich(S)$ can be defined in a straightforward fashion \cite{minsky-bddgeom,rafi-gt}. The hierarchy machinery of Masur-Minsky in the papers \cite{masur-minsky2,minsky-bddgeom} is needed to state the Theorem below.  Theorems \ref{thickgeod} and \ref{thickcombin} have been stated for bi-infinite geodesics. However, in  \cite{minsky-bddgeom,rafi-gt} these are proven  for geodesic segments and rays as well using the projection $\P_C$ above. We restate these in the form we need them (see Section 2.6 and the Bounded Geometry Theorem on p. 144 of \cite{minsky-bddgeom}):
 
 \begin{theorem}\cite{minsky-bddgeom,rafi-gt} \label{thickhierrachy}
 	For $K \geq 0$ and $\ep >0$, there exists $R > 0$ such that if
 	  $H$ is a bounded $K-$quasiconvex subset of $\Teich_\ep(S)$ then for any $x,y \in H$ and any proper essential subsurface $W$ of $S$, the hierarchy path in $W$ subordinate to any tight geodesic joining $\P_C(x), \P_C(y)$ is either empty or has length at most $R$.
 	  
 	  Conversely, for any $R > 0$ there exists $\ep, K >0$, such that the following holds. Suppose that
 	  \begin{enumerate}
 	  \item  $u, v \in \ccd(S)$ are maximal simplices equipped with transversals $t(u), t(v)$, 
 	  \item  for any proper essential subsurface $W$ of $S$ (including annular domains), the hierarchy path in $W$ subordinate to any tight geodesic joining $\P_C(x), \P_C(y)$ is either empty or has length at most $R$.
 	  \end{enumerate} 
 	  Then
 	  \begin{enumerate}
 	  	\item the set of points $x$ (resp. $y$) in $\Teich(S)$ where $u, t(u)$ (resp. $v, t(v)$) are short (bounded by the Bers' constant, say) lies in a ball of radius $K$ in $\Teich_\ep(S)$.
 	  	\item The Teichm\"uller geodesic joining such pairs $x, y$ lies in 
 	  	$\Teich_\ep(S)$.
 	  \end{enumerate}
\end{theorem}

\begin{defn}\label{def-rethickss}
 A subset $X$ of $\ccs$ is $R-$thick, if for any $v \in X$ and $v_1, v_2 \in X$ adjacent to $v$, and any component $W$ of $S \setminus v$, \begin{enumerate}
\item any geodesic $\gamma$ joining $v_1, v_2$ in $\CC(W)$ is of length at most $R$,
\item any geodesic in a hierarchy path joining $v_1, v_2$ and subordinate to a geodesic $\gamma$ as in the previous condition  is of length at most $R$.  
\end{enumerate}
\end{defn}

As an immediate consequence of Theorem \ref{thickhierrachy} we have the following:

\begin{cor}\label{pA} For $S=S_{g,n}$, let $\phi$ be a pseudo-Anosov homeomorphism. Then there exists $R>0$ such that any tight geodesic $\gamma$ in $\ccs$ preserved 
by $\phi$ is $R-$thick.

More generally, let $\phi_1, \cdots, \phi_k$ freely generate a free convex cocompact subgroup $Q=F_k$. There exists $R$ such that if $Q$ preserves a quasi-isometrically embedded tree $T_Q \subset \CC(S)$, then $T_Q$ is also $R-$thick. 
\end{cor}

\begin{defn}\label{def-subordgeod}  Let $i:\vT\to \ccd(S)$ be a balanced tree (see Definition \ref{def-balancedtree}) and $v\in T$. Let $W$ be a component  of $S \setminus i(v)$ and let $T_{v,W}$  denote a    bi-Lipschitz embedded image of the tree-link $ T_v$ of $v$ in $\CC(W)$
	(with parameters as in  Definition \ref{def-balancedtree}).

	For any two terminal vertices $u, w$ of $T_{v,W}$, any tight geodesic $\gamma_W$ joining them in $\CC(W)$, and any proper essential subsurface $W'$ of $W$, a tight geodesic supported on $W'$ and occurring in a hierarchy of geodesics subordinate to $\gamma_W$ will be called a {\bf geodesic subordinate to the tree-link $T_{v,W}$}. 
	
	If there exists a component $W$  of $S \setminus i(v)$ 
	such that  $\gamma$ is a geodesic  subordinate to the tree-link $T_{v,W}$, then $\gamma$ is  called a {\bf geodesic  subordinate to the tree-link $T_{v}$}.
	
	If there exists a vertex $v$ of $T$ such that  $\gamma$ is a geodesic  subordinate to the tree-link $T_{v}$, then $\gamma$ is simply called a {\bf geodesic  subordinate to the tree $T$}.
\end{defn}

As a consequence of Theorem \ref{thickhierrachy}, we have the following alternate description of a building block $M_v$ of special split geometry corresponding to $v$. The Corollary follows by applying Theorem \ref{thickhierrachy} to the tree-link of $v$.

\begin{cor}\label{subordinatehierarchysmall} For all $k,\ep>0$, there exists $R>0$ such that the following holds:\\ If a model building block of special split geometry has parameters 
$k,\ep>0$ then every geodesic  subordinate to the tree-link $T_{v}$ has length at most $R$.

Conversely, given $R>0$, there exists $k,\ep>0$ such that the following holds.\\ For a topological building block $M_v$ with tree-link $T_v$ if every geodesic  subordinate to the tree-link $T_{v}$ has length at most $R$ then $M_v$ admits a special split geometry structure with parameters 
$k,\ep>0$.
\end{cor}

The advantage of Corollary \ref{subordinatehierarchysmall} over Definition \ref{def-geometricbb} is that the problem is reduced to looking only at the curve complex rather than varying Teichm\"uller spaces.

\begin{rmk}\label{rmk-interval}{\rm 
We observe that the  welded split  block in Definition \ref{weld} is a  special case of  a model building block of special split geometry when the tree link $T_v$ is an interval of the form $[0,n]$ with vertices at the integer points. 

A word of caution: The   split  block of Proposition \ref{splsplit} may be quite different from the  welded split  block in Definition \ref{weld} as far as the geometry of the tubes $\tube_i$ are concerned. In the split block, the Margulis tubes have the geometry of  solid hyperbolic tori. In the welded split block, these are replaced by  flat annuli. }
\end{rmk}

We expand on Remark \ref{rmk-splsplit} and
 explicitly state here the relationship between the geometry of split blocks in totally degenerate 3-manifolds (Proposition \ref{splsplit}) and 
the special split geometry of $M_v$ as in Definition \ref{def-geometricbb}.
Let $i:\vT\to \ccd(S)$ be a balanced tree and $v\in T$. Let $l$ be a bi-infinite geodesic in $T$ through $v$. We further equip $l$ with the simplicial tree structure induced by $T$. Let $\but$ denote the blown-up tree and let $\BU(l)$ denote the blow up of $l$. Let $T_v(l)$ denote the tree-link of $v$ in $\BU(l)$ and let $M_v(l)$ denote the associated geometric building block. Let $P_v: M_v \to T_v$ and $P_v(l): M_v(l) \to T_v(l)$ denote the natural projections.

\begin{lemma}\label{bb-deg-tree} Given $R, D, k\geq 1$, $n \geq 2$ there exists $C\geq 1$ such that the following holds: \\
	Let $i:\vT\to \ccd(S)$ be an $L-$tight $R-$thick balanced tree with parameters $D, k$ (see Definition \ref{def-balancedtree}) such that each vertex of $T$ has valence at most $n$. Let $T_v, M_v, l, T_v(l), M_v(l), P_v, P_v(l)$ be as above. Then there exist
	\begin{enumerate}
		\item a $C-$bi-Lipschitz embedding $\psi_v: T_v(l) \to T_v$ taking the end-points of $T_v(l)$ to the corresponding  end-points of $T_v$.
		\item a $C-$bi-Lipschitz embedding $\phi_v: M_v(l) \to M_v$
	\end{enumerate}
such that $\psi_v \circ P_v(l) = P_v\circ\phi_v$, i.e.\ $\phi_v$ preserves fibers.
\end{lemma} 

\begin{proof}
The construction of the tree-link in Definition \ref{def-treelink} guarantees the existence of a  $C-$bi-Lipschitz embedding $\psi_v: T_v(l) \to T_v$ taking the end-points of $T_v(l)$ to the corresponding  end-points of $T_v$, where $C$ depends only on $n$.

The construction of the model geometry on $M_v$ in Definition \ref{def-geometricbb} now guarantees the bi-Lipschitz embedding $\phi_v$ with constant $C$ depending only on the parameters $k, \ep$ of the model geometries of $M_v(l), M_v$. Since $k, \ep$ depend only on $R$ by Corollary \ref{subordinatehierarchysmall}, the Lemma follows.
\end{proof}

For doubly degenerate manifolds of special split geometry, the height $l_i$ of the block $B_i$ has a nice interpretation that we now
recall.
From the construction of the Minsky model for such manifolds,
\cite[Theorem 8.1]{minsky-elc1} (see the summary in \cite[Sections 1.1.2 and 3]{mahan-split}) $l_i$ may be taken to be approximately equal to $d_{\CC(S\setminus v_i)} (v_{i-1},v_{i+1})$:

\begin{prop}\label{length=link} Given $R>0$, there exists $c_0$ such that the following holds.
	Let $l$ be an $L-$tight $R-$thick tree whose underlying topological space is homeomorphic to $\R$ and whose vertices $v_i$ are simple non-separating curves. Let $M_l$ be the corresponding model manifold of special split geometry. Then for every vertex $v_i$  of $T$, the height $l_i$ of the $i$th split block $B_i$  may be chosen to equal $l_i^+ = l_i^-$ (thus $C=1$ in Proposition \ref{splsplit}) and
	$$\vert d_{\CC(S\setminus v_i)} (v_{i-1},v_{i+1})-l_i\vert \leq c_0.$$
\end{prop}

 \subsection{Model geometry on the topological model $M_T = S \times \BU (T)$} We now describe how to glue the geometric building blocks together to obtain a model geometry on  $M_T = S \times \BU (T)$. Since the model geometry will be quite similar to the metric in Definition \ref{weld}, the resulting metric on $M_T$ will also be denoted as $d_{weld}$. There are two points of view one can adopt in describing the model geometry: hierarchy paths or geodesics in Teichm\"uller space. It will be more convenient to define the model using hierarchy paths as observed after Corollary \ref{subordinatehierarchysmall}.

 \begin{defn}\label{def-rthick}
 A balanced tree $i: V(T)\to \ccd(S)$ is said to be $L-$tight and $R-$thick if
 \begin{enumerate}
 \item it is $L-$tight in the sense of Definition \ref{def-tighttree}, and
 \item all  geodesics  subordinate to the tree $T$ have length at most $R$.
 \end{enumerate} 
 \end{defn}
 
 To recover the model geometry on $M_T = S \times \BU (T)$ from Definition \ref{def-rthick} we shall need the model geometry used in the Ending Lamination Theorem \cite{minsky-elc1,minsky-elc2} of Brock-Canary-Minsky.
 Note that for any $v\in T$, Corollary \ref{subordinatehierarchysmall} furnishes a model building block $M_v$ of special split geometry as a bundle over the tree-link $T_v$. To construct the model geometry on $M_T $, it remains to assemble the pieces given by $M_v$. Note also that:
 \begin{enumerate}
 \item Every terminal vertex of $T_v$ corresponds to a mid-point vertex $vw$ of the blown-up tree $\BU(T)$ (Definition \ref{def-blowup}), where $w$ is adjacent to $v$ in $T$.
 \item For every terminal vertex $vw$ of $T_v$, the mid-surface $S_{vw}$ (Definition \ref{def-topmodeltree}) is of (uniformly, independent of $v,w$) bounded geometry, i.e.\ it has injectivity radius uniformly bounded below and diameter uniformly bounded above.
 \end{enumerate}
 
 In order to assemble the pieces given by $M_v$ therefore, it suffices to determine (at least coarsely) the gluing maps between $M_v$ and $M_w$ at $S_{vw}$ as $v,w$ range over adjacent vertices in $T$. Since $S_{vw}$ is of uniformly bounded geometry, it will suffice to show that, up to a choice of a base-point in $\Teich_\ep(S)$ (where $\ep$ is as in Corollary \ref{subordinatehierarchysmall}), $S_{vw}$ lies in a uniformly (independent of $v,w$) bounded ball in $\Teich_\ep(S)$. It is precisely this fact that is furnished by the Minsky model as summarized and explained in Sections 1.1.2 and 1.1.3 of \cite{mahan-split}. 
 
 We briefly recall the necessary facts and the argument for completeness. We shall find it convenient to think of $T$ as rooted, with root vertex $\ast$. Let $l$ be any bi-infinite geodesic in $T$ through $\ast$. Then $i(l)$ is a tight geodesic in $\CC(S)$ by our hypothesis on $i:T \to \ccd (S)$ and gives a bi-infinite tight geodesic in $\CC(S)$ converging to ending laminations $l_\pm \in \EL(S) = \partial \CC(S)$ \cite{klarreich-el}. Given such a tight geodesic, Minsky \cite{minsky-elc1} constructs a combinatorial model $M_l$ for a hyperbolic 3-manifold $N_l$ with ending laminations $l_\pm$. Finally, Brock-Canary-Minsky \cite{minsky-elc2}  prove that $M_l$ is uniformly bi-Lipschitz homeomorphic to $N_l$. The construction of $M_l$ in \cite[Theorem 8.1]{minsky-elc1} shows in particular that the bounded geometry surfaces in $M_l$ correspond to markings and hence give coarsely well-defined points of $\Teich(S)$ (once a base surface is chosen and identified with a base-point of $\Teich(S)$).
 
 Proposition \ref{splsplit} now shows that if moreover $l$ is  $L-$tight (for some $L\geq 3$) and $R-$thick, then 
 \begin{enumerate}
 \item $M_l$ admits a  bi-Lipschitz homeomorphism to a model of special split geometry (Definition \ref{def-splsplitcombin}).  Further, the bi-Lipschitz constant and the parameters $\ep, D>0$ occurring in Proposition \ref{splsplit} depend only on $R$.
 \item The split surface (Item (2) of Proposition \ref{splsplit}) between split blocks  corresponding  to adjacent vertices $v,w$ in $l$ gives a coarsely well-defined element $S(v,w)$ of $\Teich(S)$. 
 \end{enumerate}
 
 We restate the last conclusion more precisely.
 Given $R>0$  there exists $r, \ep>0$ such that the following holds:\\ Let $i:\vT \to \ccd(S)$ be $L-$tight  and $R-$thick. Then
 for any pair of adjacent vertices $v,w\in T$,    and any bi-infinite  geodesic $i(l)$,  passing through $i(v), i(w)$ and $\ast$, the split surface  between split blocks  corresponding  to  $v,w\in l$ lies in $N_r(S(v,w)) \subset \Teich_\ep(S)$.
 Note that $r, \ep>0$ depend on $R$ but not $L$. 
 
 Thus we have a coarsely well-defined element $S(v,w)$ of $\Teich(S)$ corresponding to the mid-surface $S_{vw}$ independent of the bi-infinite geodesic $l$  passing through $v,w$.
 We summarize the above discussion as follows:
 \begin{theorem} 
 	\label{midsurfminskymodel} There exists $C_0 \geq 1$ depending only on the topology of $S$ and given $R>0$, $D_0, k_0 \geq 1$ there exist $r, \ep,  >0, C, D, k \geq 1$ such that the following holds:\\
 	Suppose that $i: \vT\to \ccd(S)$ is an $L-$tight  $R-$thick balanced tree with  parameters $D_0, k_0$ as in Definition \ref{def-balancedtree}.  Let $\ast$ be a root of $T$. Let $l$ be  any bi-infinite tight geodesic  in $i(T)$ through $\ast$ with end-points $l_\pm \in \EL(S) =\partial \CC(S)$.  Then
 	\begin{enumerate}
 	\item The  doubly degenerate hyperbolic 3-manifolds $N_l$ with end-invariants $l_\pm$ are of special split geometry with constants $\ep, D>0, C\geq 1$ as in Proposition  \ref{splsplit}.
 	\item The model manifold $M_l$ is $C_0-$bi-Lipschitz homeomorphic to $N_l$. 
 	\end{enumerate} 
 	Further, for any pair of adjacent vertices $v,w\in T$, there exists $S(v,w) \in \Teich_\ep(S)$  such that for any   geodesic $l$ in $T$,  passing through $i(v), i(w), \ast$, the split surface  between split blocks in $M_l$ corresponding  to  $v,w\in l$ lies in $N_r(S(v,w)) \in \Teich_\ep(S)$.
 \end{theorem}
 
 \begin{defn}\label{def-mg}
 Theorem \ref{midsurfminskymodel} implies in particular that the mid-surfaces $S_{vw}$ of $\BU(T)$ are (coarsely) well-defined points of $\Teich(S)$. Thus the image of $lk(v) (\subset \BU(T))$ in $\Teich(S)$ is (coarsely)  well-defined under a qi-section as a finite set of points (of uniformly bounded cardinality). Interpolating the model building blocks $(M_v,d_v)$ of special split geometry finally gives us the {\bf model metric} $\dw$ on $M_T$. The pair $(M_T,\dw)$ will be called the {\bf model of special split geometry} on the topological model $M_T$.
 
 Replacing each $(M_v, d_v)$ in $(M_T,\dw)$ with the tube-electrified (pseudo-)metric $(M_v, \dt)$ (Definition \ref{def:tubeelbb}) gives us the  tube-electrified metric $\dt$ on $M_T$. The pair $(M_T,\dt)$ will be called the {\bf tube electrified model of special split geometry} on the topological model $M_T$.  $P: (M_T,\dw) \to \but$
 and $P: (M_T,\dt) \to \but$ will denote the natural projections.
 
The lift of the metric $\dw$ (resp. $\dt$)  to the universal cover $\til M_T$  is also denoted by $\dw$ (resp. $\dt$).
Also, $P: \tmtdw \to \but$
and $P: \tmtdt \to \but$ will denote the natural projections.
\end{defn}
 
We should remind the reader of the caveat in Remark \ref{rmk-interval}:  the model metrics on $(M_v,d_v)$ differ from the model metrics on the split blocks of Proposition  \ref{splsplit} at the Margulis tubes.

Lemma \ref{proper-block} and Theorem \ref{midsurfminskymodel} give us the following:

\begin{lemma}\label{proper-proj}  Given a surface $S$, $D, k \geq 1$ and $R>0$, there exist $c\geq 1$ such that the following holds:\\
	Suppose that $i: \vT\to \ccd(S)$ is an $L-$tight  $R-$thick balanced tree with  parameters $D, k$ as in Definition \ref{def-balancedtree}.
	Then $P: \mtdw \to \but$ and  $P: \mtdt \to \but$ are $c-$proper.
\end{lemma}

\begin{proof} By
Corollary \ref{subordinatehierarchysmall}, there exist $k, \ep$ depending on $R$  such that each $M_v$ is of split geometry with parameters  $k, \ep$. Theorem \ref{midsurfminskymodel} now shows that the mid-surfaces $S_{vw}$ of $\BU(T)$ are coarsely well-defined points of $\Teich(S)$:  the constant $r$ occurring in the conclusion of Theorem \ref{midsurfminskymodel} depends only on $R$. Hence $P: \mtdw \to \but$   is $c-$proper. It follows that $P: \mtdt \to \but$   is $c-$proper.
\end{proof}

\subsection{The Main  Theorems}\label{sec-maintech} We are now in a position to present the main  theorems of this paper. We carry forward the notation from the discussion preceding Lemma \ref{bb-deg-tree}:
$l$ is a bi-infinite geodesic in $T$ and $\BU(l)$ denotes the bi-infinite geodesic in $\but$ after blowing up $l$ in $T$. Further, let $\VV(l)$ denote the collection of vertices of $T$ on $l$, $N_l$ denote the doubly degenerate hyperbolic 3-manifold with ending laminations given by $l_\pm$, the ideal end-points of $i(l)$. Let $\T_v$ denote the Margulis tube in $N_l$ corresponding to $v$. Let $N_l^0 = N_l \setminus \bigcup_{v \in \VV(l)} \T_v$. Also let $M_l$ denote the bundle over $\BU(l)$ induced from $\Pi: M_T \to \but$. Let $M_l^0 = M_l \setminus \bigcup_{v \in \VV(l)} \RR_v$.

\begin{theorem} \label{model-str} Given $R > 0$, $D, k \geq 1$, there exist $K, c\geq 1, e >0$ such that the following holds. Let $i:\vT \to \ccd(S)$ be an $L-$tight  $R-$thick balanced tree  with parameters $D, k$ as in Definition \ref{def-balancedtree}. There exists a  metric $\dw$ on $M_T$ such that $P: M_T \to \but$  satisfies the following:
	\begin{enumerate}
		\item The induced metric on a Margulis riser $\RR_v$ is the metric product $S^1_e \times T_v$, where $S^1_e$ is a round circle with radius $e$. 
		\item  For any bi-infinite geodesic $l$ in $T$, $N_l^0$ and $M_l^0$ are $K-$bi-Lipschitz homeomorphic.
		\item Further, if there exists a subgroup $Q$ of $MCG(S)$ acting cocompactly and geometrically on $i(T)$, then this action can be lifted to an isometric fiber-preserving isometric action of $Q$ on $(M_T,\dw)$.
		\item $P:\mtdt \to \but$ is $c-$proper.
	\end{enumerate}
\end{theorem}  

\begin{proof}
	Item (1) follows immediately from the construction in Definition \ref{def-geometricbb} and Lemma \ref{bb-deg-tree}.\\
	
		Item (2) follows from Proposition \ref{splsplit} and Lemma \ref{bb-deg-tree}.\\
		
		Item (3) follows from the observation that the constructions  of the tree-link in Definition \ref{def-treelink}, the blow-up in  Definition \ref{def-blowup}, and the model geometry in Definition \ref{def-geometricbb} can all be 
done equivariantly with respect to the  action of  $Q$.\\

Item (4) follows from Lemma \ref{proper-proj}.
\end{proof}

 The lift of the pseudometric $\dt$ on $( M_T,\dt)$ to $\til M_T$ is also denoted by $\dt$. 

\begin{theorem}\label{maintech} Given $R > 0$, $D, k \geq 1$, there exists $\delta_0, L_0 \geq 0$ such that the following holds. Let $i:\vT \to \ccd(S)$ be an $L-$tight  $R-$thick balanced tree  with $L\geq L_0$
	and parameters $D, k$ as in Definition \ref{def-balancedtree}. Then
 $(\til M_T,\dt)$ is $\delta_0-$hyperbolic.
\end{theorem}
In the statement of Theorem \ref{maintech} we have explicitly mentioned the constant $L_0$ from Standing Assumption \ref{assumption}.
The proof of Theorem \ref{maintech} will occupy the rest of the paper.

\section{Effective combination theorems and relative  hyperbolicity}\label{sec:prelims} Before proving  
Theorem \ref{maintech}  we shall  recall, organize and
 adapt some known material on combination theorems and relative hyperbolicity. The fundamental combination theorem in the context of trees of spaces is due to Bestvina and Feighn \cite{BF}. Its converse is due, in various forms, to Gersten \cite{gersten}, Bowditch
 \cite{bowditch-ct} and others. This was generalized to the context of relative hyperbolicity in \cite{mahan-reeves,mahan-sardar}. An effective (i.e.\ with constants) generalization is due to  Gautero \cite[Theorem 2]{gautero-conv} \cite[Theorem 2.20]{gautero} (see especially Sections 7, 8 of the last paper), \cite{gautero-weidmann}. In the context that we are interested in, the base-tree will be a metric  tree where some of the edges (corresponding to edges of the tree-links $T_v$, see Definition  \ref{def-balancedtree}) might have non-integral length. Strictly speaking, therefore we are in the context of a metric bundle in the sense of \cite{mahan-sardar} where the fibers are uniformly hyperbolic (see Remark \ref{mbdl-mgbdl} below for going back and forth between trees of spaces and metric bundles). The combination theorem and its converse for metric bundles  are proven in \cite[Theorem 4.3, Proposition 5.8]{mahan-sardar}. It is also shown in \cite{mahan-sardar} that the metric bundle is (with effective uniform constants) quasi-isometric to a metric graph bundle.

 We shall be specifically interested in the following bundles:

\begin{enumerate}
\item The universal cover  $\tmtdw$ of the bundle $\mtdw$,
\item The universal cover  $\tmtdt$ of the bundle $\mtdt$.
\end{enumerate}

Both have as base the blown-up tree $\BU(T)$ (see Definition \ref{def-blowup}). We shall denote the projection map to the base as
$P: \tmtdw \to \but$ or $P: \tmtdt \to \but$.
Metric bundles over trees are examples of trees of spaces (Section \ref{sec-effct} below) as well as metric bundles in the sense of \cite{mahan-sardar} and both points of view will be important. We shall in Section \ref{effcombbut} use the terminology of metric bundles and adapt the statements of \cite{gautero,gautero-conv,mahan-sardar} to the context of $P: \tmtdt \to \but$.

\subsection{Trees of hyperbolic spaces and effective combination theorem}\label{sec-effct}
We recall the notion of a tree of spaces.
\begin{defn} \label{tree}
 \cite{BF} Let $(X,d)$ be a geodesic
 metric space and $T$ a simplicial tree with vertex set $\VV(T)$ and edge set $\EE(T)$. $P: X \rightarrow T$ is said to be a  tree of geodesic
	metric spaces satisfying
	the \emph{quasi-isometrically embedded condition} (or \emph{qi condition}) if  there exists
	 a map $P : X \rightarrow T$, and constants
 $K \geq 1, \epsilon \geq 0$ satisfying the
	following: 
	\begin{enumerate}
	\item For all vertices $v\in{\VV(T)}$,
	$X_v = P^{-1}(v) \subset X$ with the induced path metric $d_{v}$ is
	a geodesic metric space $X_v$. Further, the
	inclusions ${i_v}:{X_v}\rightarrow{X}$
	are uniformly proper, i.e. for all $M > 0$, $v\in{T}$ and $x, y\in{X_v}$,
	there exists $N > 0$ such that $d({i_v}(x),{i_v}(y)) \leq M$ implies
	${d_{X_v}}(x,y) \leq N$.
	\item Let $e \in \EE(T)$ with initial and final vertices $v_1$ and
	$v_2$ respectively.
	Let $X_e$ be the pre-image under $P$ of the mid-point of  $e$.
	There exist continuous maps ${f_e}:{X_e}{\times}[0,1]\rightarrow{X}$, such that
	$f_e{|}_{{X_e}{\times}(0,1)}$ is an isometry onto the pre-image of the
	interior of $e$ equipped with the path metric. Further, $f_e$ is fiber-preserving,
	i.e. projection to the second co-ordinate in ${X_e}{\times}[0,1]$ corresponds via $f_e$
	to projection to the tree $P: X \rightarrow T$.
	\item Identifying $e$ with $[0,1]$, ${f_e}|_{{X_e}{\times}\{{0}\}}$ and
	${f_e}|_{{X_e}{\times}\{{1}\}}$ are $(K,{\epsilon})$-quasi-isometric
	embeddings into $X_{v_1}$ and $X_{v_2}$ respectively.
	${f_e}|_{{X_e}{\times}\{{0}\}}$ and
	${f_e}|_{{X_e}{\times}\{{1}\}}$ will occasionally be referred to as
	$f_{e,v_1}$ and $f_{e,v_2}$ respectively. 
	\end{enumerate}
	
	$K, \epsilon$ will be called the constants or parameters of the  qi-embedding condition.
\end{defn}

A tree of spaces $P: X \rightarrow T$ 
as in Definition \ref{tree} above is said to be a tree of
hyperbolic metric
spaces, if there exists $\delta > 0$ such that the vertex and edge spaces $X_v, X_e$ are all
$\delta$-hyperbolic for all vertices $v$ and edges $e$ of $T$.

\begin{defn}\label{def-hallway}\cite{BF}
	A disk $f : [-m,m]{\times}{I} \rightarrow 
	X$ is a {\bf hallway} of length $2m$ if it satisfies:
	
	\begin{enumerate}
		\item $f^{-1} ({\cup}{X_v} : v \in T) = \{-m,  \cdots , m \}{\times}
		I$
		\item $f$ maps $i{\times}I$ to a geodesic in  $X_v$ for some vertex
		space $X_v$.
		\item $f$ is transverse, relative to condition (1) to $\cup_e X_e$.
	\end{enumerate}
\end{defn}

\begin{defn}\label{def-rhothin}\cite{BF} A hallway $f : [-m,m]{\times}{I} \rightarrow 
	X$ is {\bf $\rho$-thin} if 
	$d({f(i,t)},{f({i+1},t)}) \leq \rho$ for all $i, t$.
	
	A hallway $f : [-m,m]{\times}{I} \rightarrow 
	X$ is said to be {\bf $\lambda$-hyperbolic}  if 
	$$\lambda l(f(\{ 0 \} \times I)) \leq \, {\rm max} \ \{ l(f(\{ -m \} \times I)),
	l(f(\{ m \} \times I)).$$

	The quantity ${\rm min_i} \, \{ l(f(\{ i \} \times I))\}$ is called the {\bf girth} of the hallway.

	A hallway is {\bf essential} if the edge path in $T$ 
	resulting from projecting the hallway under $P\circ f$
	onto $T$ does not backtrack (and is therefore a geodesic segment in
	the tree $T$).
\end{defn}

\begin{defn}\label{def-flare} {\bf Hallways flare condition \cite{BF}:}
	The tree of spaces, $X$, is said to satisfy the {\bf hallways flare}
	condition if there are numbers $\lambda > 1$ and $m \geq 1$ such that
	for all $\rho$ there is a constant $H:=H(\rho )$ such that  any
	$\rho$-thin essential hallway of length $2m$ and girth at least $H$ is
	$\lambda$-hyperbolic. In general, $\lambda, m$ will be called the constants of the hallways flare condition. If, in addition $\rho$ is fixed,
	$H$ will also be called a constant of the hallways flare condition.
\end{defn}

We  recall  the notion of a metric bundle from \cite{mahan-sardar}:

\begin{defn}\label{def-mbdl}
	Let $(X,d_X)$ and $(B, d_B)$ be geodesic metric spaces. Let $c, K\geq 1$ be  constants and 
	$h:{\mathbb R}^+ \rightarrow {\mathbb R}^+$  a function.
	$P: X\to B$ is called an $(h,c,K)-$ {\bf metric bundle} if
	\begin{enumerate}
		\item $P$ is 1-Lipschitz.
		\item For each  $z\in B$, $X_z=P^{-1}(z)$ is a geodesic metric space
		with respect to the path metric $d_z$ induced from $(X,d_X)$. Further, we require that the inclusion maps
		$i_z: (X_z,d_z) \rightarrow X$ are uniformly metrically proper as measured with respect to $h$, i.e.\ for  all $z\in B$ and $u, v\in X_z$,  $d_X(i_z(u),i_z(v))\leq N$ implies that $d_z(u,v)\leq f(N)$. 
		\item For $z_1,z_2\in B$ with $d_B(z_1,z_2)\leq 1$, let $\gamma$ be
		a geodesic in $B$ joining them. 
		Then for any $z\in \gamma$ and $x\in X_z$,  there is a path in $p^{-1}(\gamma)$
		of length at most $c$ joining $x$ to both $X_{z_1}$ and $X_{z_2}$.
		\item For $z_1,z_2\in B$ with $d_B(z_1,z_2)\leq 1$ and $\gamma \subset B$
		a geodesic  joining them,
		let $\phi: X_{z_1}\rightarrow X_{z_2}$, be a(ny) map such that
		for all $ x_1\in X_{z_1}$ there is a path of length at most $c$ in $P^{-1}(\gamma)$
		joining $x_1$ to $\phi(x_1)$. Then $\phi$ is a $K-$quasi-isometry.
	\end{enumerate}
If in addition, there exists $\delta'$ such that each $X_z$ is $\delta'-$hyperbolic, then $P: X\to B$ is called an $(h,c,K)-$  metric bundle
of $\delta'-$hyperbolic spaces.
\end{defn}
It is pointed out in \cite{mahan-sardar} that  condition (4) follows from the previous three (with some $K$); but it is more convenient to have it as part of our definition. For any hyperbolic metric space $F$ with more than two points in its Gromov boundary $\partial F$, there is a coarse {\bf barycenter map} $\phi : \partial^3 F \rightarrow F$ mapping any unordered triple $(a,b,c)$ of distinct points in $\partial F$ to a centroid of the ideal triangle spanned by $(a,b,c)$. We shall say that the
barycenter map $\phi : \partial^3 F \rightarrow F$ is  $N-$coarsely surjective if $F$ is contained in the $N$-neighborhood of the image of $\phi$. 
A {\bf $K-$qi-section} $\sigma: B\to X$ is a $K-$qi-embedding from $B$ to $X$ such that $P \circ \sigma $ is the identity map.
The following Proposition guarantees the existence of qi-sections for metric bundles.

\begin{prop}\label{qi-section}\cite[Section 2.1]{mahan-sardar} For all $\delta^{'},N, c, K\geq 0$ and 
	proper $f:{\mathbb{N}} \rightarrow {\mathbb{N}} $ there exists $K_0$ such that the following holds.\\
	Suppose $p : X \rightarrow B$ is an $(f,c, K)$-metric bundle of $\delta'-$hyperbolic spaces such that the barycenter maps $\phi_b : \partial^3 F_b \rightarrow F_b$ are uniformly $N-$coarsely surjective, 
	Then there is a $K_0$-qi section through each point of $X$.
\end{prop}

\begin{rmk}\label{mbdl-mgbdl} A word of clarification is necessary regarding the relationship between 
	\begin{enumerate}
	\item Metric bundles over trees in the sense of Definition \ref{def-mbdl}, and
	\item A tree of spaces satisfying the qi-embedded condition in the sense of Definition \ref{tree} with the additional restriction that the edge-space to vertex-space maps in Item (3) of Definition \ref{tree} are $(K,\ep)$-quasi-isometries rather than just $(K,\ep)$-quasi-isometric embeddings. We refer to such a tree of spaces as a {\bf homogeneous} tree of spaces.
	\end{enumerate}
It is clear that a  homogeneous tree of spaces is an example of a metric bundle over a tree. The converse is not, strictly speaking, true as the metric on fibers $F_b$ in Definition \ref{def-mbdl} is allowed to change continuously. 

However,  all the underlying trees $\but$ of metric bundles (Definition \ref{def-blowup}) occurring in this paper can be assumed to be simplicial trees (with edges of length one) as they approximate geodesic polygons in curve complexes. Further, as shown in \cite[Lemma 1.21]{mahan-sardar}, any metric bundle over a tree can be approximated by a homogeneous tree of spaces. (In \cite{mahan-sardar} a more general result was proven approximating  general metric bundles  by  metric graph bundles.) The constants $(K,\ep)$ occurring  in Definition \ref{def-mbdl} are then  determined by the parameters $(h,c,K)$ occurring  in Definition \ref{def-mbdl}.

We shall thus assume henceforth, without mentioning it explicitly, that whenever we are talking of a metric bundle over a tree as a homogeneous tree of spaces, we have approximated  the former by the latter as  in \cite[Lemma 1.21]{mahan-sardar}.
\end{rmk}

We shall now state the main theorem of \cite{BF}  in an effective form, using   \cite[Theorem 2.20]{gautero} where the proof does not
require uniform properness of the
space. 
A converse may be found in  \cite[Theorem 2]{gautero-conv} (see also \cite{gersten,bowditch-ct}). We shall however, state the theorem and its converse \cite[Section 5.3]{mahan-sardar} in the restrictive setting of a metric bundle over a tree, where it is easier to state.
\begin{theorem}\label{effectiveBF} Suppose that there exist $\delta_0 \geq 0$ and $\rho \geq 1$  such that $P:X\to T$ is a metric bundle over a tree satisfying the following conditions:
	\begin{enumerate}
		\item  $X_z$ is $\delta_0-$hyperbolic, for every $z \in T$. 
		\item through every $x \in X$ there is a $\rho-$qi-section $\sigma_x:T \to X$.
	\end{enumerate} 
Then given $K_0, \ep_0, \lambda_0, m_0, H_0$ there exists $\delta > 0$ such that the following holds:\\
If $X$  satisfies the 
	qi-embedded condition with constants $K \leq K_0, \ep \leq \ep_0$ and the  hallways flare
	condition with constants $\lambda \geq \lambda_0, m \leq m_0, H\leq H_0$ for hallways bounded by $\rho-$qi-sections, then $X$ is  $\delta-$hyperbolic. \\
	
Conversely, given $\delta>0$,  there exist $K_0 \geq 1, \ep_0 \geq 0$ and $\lambda_0 > 1, m_0 \in \natls, H_0 \geq 0$ such that if 
 $X$ is $\delta-$hyperbolic,
then as a tree of hyperbolic metric spaces $X$ satisfies
\begin{enumerate}
\item  the 
qi-embedded condition with constants $K \leq K_0, \ep \leq \ep_0$.
\item  hallways bounded by $\rho-$qi-sections satisfy the flare
condition with constants $\lambda \geq \lambda_0, m \leq m_0, H \leq H_0$.
\end{enumerate}  
\end{theorem}

\subsection{Effective relative hyperbolicity}\label{sec-effrh}

We shall also need to quantify relative hyperbolicity. If $X$ is strongly hyperbolic relative to a collection $\HH$ of parabolic subsets (see \cite{farb-relhyp, bowditch-relhyp} for definitions) we can attach a hyperbolic cone $H_h$ to each $H \in \HH$ as follows.

\begin{defn}
	For any geodesic metric space
	$(H,d)$, the {\em hyperbolic cone} (analog of a horoball)
	$H^h$ is the metric space
	$H\times [0,\infty) = H^h$ equipped with the
	path metric $d_h$ obtained from two pieces of
	data \\
	1) $d_{h,t}((x,t),(y,t)) = 2^{-t}d_H(x,y)$, where $d_{h,t}$ is the induced path
	metric on $H\times \{t\}$.  Paths joining
	$(x,t),(y,t)$ and lying on  $H\times \{t\}$
	are called {\em horizontal paths}. \\
	2) $d_h((x,t),(x,s))=\vert t-s \vert$ for all $x\in H$ and for all $t,s\in [0,\infty)$, and the corresponding paths are called
	{\em vertical paths}. \\
	3)  for all $x,y \in H^h$,  $d_h(x,y)$ is the path metric induced by the collection of horizontal and vertical paths. \\
\end{defn}

\begin{defn}
	Let $X$ be a geodesic metric space and $\HH$ be a collection of mutually disjoint uniformly separated subsets of $X$.
	$X$ is said to be  strongly hyperbolic relative to $\HH$, if the quotient space $\GG (X, \HH)$,  obtained by attaching the hyperbolic cones
	$ H^h$ to $H \in \HH$  by identifying $(z,0)$ with $z$
	for all $H\in \HH$ and $z \in H$,
	is a complete hyperbolic metric space. The collection $\{ H^h : H \in \HH \}$ is denoted
	as ${\HH}^h$. The induced path metric is denoted as $d_h$.
\end{defn}

 As per Bowditch's definition of relative hyperbolicity  \cite{bowditch-relhyp} following Gromov \cite{gromov-hypgps},  $X$ is strongly hyperbolic relative to   $\HH$ if $\GG (X, \HH)$ is hyperbolic.
 We make this effective as follows: 

\begin{defn}\label{def-srh} We say that 
	$X$ is strongly $\delta-$hyperbolic relative to a collection $\HH$ of parabolic subsets if $\GG (X, \HH)$ is $\delta-$hyperbolic.
\end{defn}

\subsubsection{Partial Electrification}
In this subsection, we give a quantitative version of the notion of  partial electrification following
\cite{mahan-reeves,mahan-pal,mahan-sardar}.

\begin{defn}
	Let $(X, \HH , \GG , \LL )$ be an ordered quadruple such that the
	following holds for some $K, \epsilon, \delta >0$:
	\begin{enumerate}
		\item $X$ is a geodesic metric space. $\HH$ is a collection of subsets $H_\alpha$ of $X$.
		$X$ is  strongly $\delta-$hyperbolic relative to $\HH$.
		\item  $\LL$ is a collection of $\delta$-hyperbolic metric spaces
		$L_\alpha$
		and $\GG$ is a collection of  coarse $(K, \epsilon)-$Lipschitz maps
		$g_\alpha : H_\alpha \rightarrow L_\alpha$. Note that the indexing set for $H_\alpha, L_\alpha, g_\alpha$ is common.
	\end{enumerate}
	The {\bf partially electrified space} or
	{\em partially coned off space}  $\PEX $
	corresponding to $(X, \HH , \GG , \LL)$
	is
	obtained from $X$ by gluing in the (metric)
	mapping cylinders for the maps
	$g_\alpha : H_\alpha \rightarrow L_\alpha$. The metric on $\PEX$ is denoted by
	$d_{pel}$.
	\label{pex}
\end{defn}

In the particular case that each $L_\alpha$ is a point and $g_\alpha$ is a constant map, this gives back the electrified, or {\bf coned-off} space $\EXH$ in the sense of Farb \cite{farb-relhyp}.
For the next two statements, see  \cite[Lemmas 1.20. 1.21]{mahan-pal}, (also \cite{mahan-reeves},\cite[Lemma 1.50]{mahan-sardar}).
\begin{lemma} For $K, \epsilon, \delta >0$ there exists $\delta', C$ such that the following holds:\\
	Let $(X, \HH , \GG , \LL )$  be an ordered quadruple as in Definition
	\ref{pex} above with constants $K, \epsilon, \delta >0$. Then
	$(\PEX ,d_{pel})$ is a $\delta'-$hyperbolic metric space and the sets $L_\alpha$
	are $C-$quasiconvex.
	\label{pel}
\end{lemma}

\begin{lemma}
	Let $(X, \HH , \GG , \LL )$  be an ordered quadruple with constants as in Definition
	\ref{pex} above.
	Given $K_0, \epsilon_0 \geq 0$, there exists $C_0 > 0$ such that the following
	holds: \\
	Let $\gamma_{pel}$ and $\gamma$ denote respectively a $(K_0, \epsilon_0 )$
	partially electrified quasigeodesic in $(\PEX,d_{pel})$ and a
	$(K_0, \epsilon_0 )$quasigeodesic in $(\GXH ,d_h)$ joining $a, b$. Then $\gamma \setminus
	\bigcup_{H_\alpha\in\HH} H_\alpha$
	lies in a  $C$-neighborhood of (any representative of)
	$\gamma_{pel}$ in $(X,d)$. Further, outside of  a $C$-neighborhood of the horoballs
	that $\gamma$ meets, $\gamma$ and $\gamma_{pel}$ track each other, i.e. lie in a $C$-neighborhood
	of each other.
	\label{pel-track}
\end{lemma}

\subsection{Effective relatively hyperbolic combination theorem}\label{sec-effrhcomb}
We follow \cite{gautero} and  \cite{mahan-reeves}  here and subsequently indicate the modifications needed for us.

\begin{defn}\label{def-treerh} A tree $P: X \rightarrow T$ of geodesic
	metric spaces is said to be a  tree of relatively hyperbolic metric spaces
	if in addition to the conditions of Definition \ref{tree}
	\begin{enumerate}
		\item[(4)] each  vertex space $X_v$ is strongly hyperbolic relative to a collection of subsets $\HH_v$ and  each 
		edge space $X_e$ is strongly hyperbolic relative to a collection of subsets $\HH_e$.  The individual sets
		$H_{v,\alpha}\in \HH_{v}$ or $H_{e,\alpha}\in \HH_{e}$ will be called {\bf horosphere-like sets}.
		\item[(5)]  the maps $f_{e,v_i}$ above ($i = 1, 2$) are {\bf
			strictly type-preserving}, i.e. $f_{e,v_i}^{-1}(H_{v_i,\alpha})$, $i =
		1, 2$ (for
		any $H_{v_i,\alpha}\in \HH_{v_i}$)
		is
		either empty or some $H_{e,\beta}\in \HH_{e}$. Also, for all
		$H_{e,\beta}\in \HH_{e}$, there exists $v$ and
		$H_{v,\alpha}$, such that $f_{e,v} ( H_{e,\beta}) \subset H_{v,\alpha} $.
		\item[(6)]  There exists $\delta > 0$ such that each $\EE (X_v, \HH_v )$ is $\delta$-hyperbolic.
		\item[(7)]  The induced maps (see below) of the coned-off edge spaces into the
		coned-off vertex spaces $\hhat{f_{e,v_i}} : \EE ({X_e}, \HH_e ) \rightarrow
		\EE ({X_{v_i}}, \HH_{v_i})$ ($i = 1, 2$) are uniform quasi-isometries. This is called the
		{\bf qi-preserving electrification condition}
	\end{enumerate}
\end{defn}

Given the tree of spaces with vertex spaces $X_v$ and edge spaces $X_e$ there exists a naturally associated  tree whose vertex spaces are
$\EE (X_v, {\HH}_v)$ and edge spaces are
$\EE (X_e, {\HH}_e)$ obtained by simply coning off the respective horosphere like sets.
Condition (4) of the above definition ensures that we have natural inclusion maps of edge spaces
$\EE (X_e, {\HH}_e)$ into adjacent vertex spaces $\EE (X_v, {\HH}_v)$.

The resulting tree of coned-off spaces $P: \TC (X) \rightarrow T$
will be called the {\bf induced
	tree of coned-off spaces}. The resulting space will thus be denoted as
$\TC (X)$ when thought of as a tree of spaces.
The {\bf cone locus} of $\TC (X)$ is
the graph (actually a forest) whose vertex set $\VV$ consists of the
cone-points $c_v$ in the vertex set and whose edge-set $\EE$
consists of the
cone-points $c_e$ in the edge set.

 Each such connected component of the
 cone-locus will be called a {\bf maximal cone-subtree}. The collection
 of {\em maximal cone-subtrees} will be denoted by $\TT$ and elements
 of $\TT$ will be denoted as $T_\alpha$. Further, each maximal
 cone-subtree $T_\alpha$ naturally gives rise to a tree $T_\alpha$ of
 horosphere-like subsets depending on which cone-points arise as
 vertices and edges of $T_\alpha$. The metric space that $T_\alpha$
 gives rise to will be denoted as $C_\alpha$ and will be referred to as
 a {\bf maximal cone-subtree of horosphere-like spaces}. The induced tree of horosphere-like sets will be denoted as $g_\alpha : C_\alpha
 \to T_\alpha$. 
  The collection of these maps will be denoted as $\GG$.
 The collection
 of $C_\alpha$'s will be denoted as $\CC$. 
  Note thus that each $T_\alpha$ thus appears in two guises:\\
 1) as a subset of $\TC (X)$ \\
 2) as the underlying tree of $C_\alpha$\\

An essential  hallway of length $2m$ is {\bf cone-bounded} if
$f(i \times {\partial I})$
lies in the cone-locus for $i = \{ -m, \cdots , m\}$.

\begin{defn} {\bf Cone-bounded hallways strictly flare
		condition:} 
	The tree of spaces, $X$, is said to satisfy the {\em cone-bounded hallways flare}
	condition if there are numbers $\lambda > 1$ and $m \geq 1$ such that
	any
	cone-bounded hallway of length $2m$  is
	$\lambda$-hyperbolic. $\lambda, m$ will be called the constants of the strict flare condition.
\end{defn}

\begin{theorem}\label{effectiverelBF}\cite{mahan-reeves,gautero}
 Given $K_0 \geq 1, \ep_0 \geq 0, \delta_0 \geq 0, \lambda_0 > 1, m_0 \geq 1, \rho_0 > 1, H_0 \geq 0$ 
 there exists $\delta > 0$ such that the following holds:
	Let $P:X\to T$ be a metric bundle over a tree such that
	\begin{enumerate}
		\item  $X_z$ is $\delta_0-$relatively hyperbolic, for every $z \in T$. 
		\item through every $x \in X$ there is a $\rho_0-$qi-section $\sigma_x:T \to X$.
	\end{enumerate}  
If $X$ satisfies the
	qi-embedded condition with constants $K \leq K_0, \ep \leq \ep_0$,  the  hallways flare
	condition with constants $\lambda \geq \lambda_0, m \leq m_0, H \leq H_0$ with respect to hallways bounded by $\rho_0-$qi-sections, and the cone-bounded hallways strictly flare condition with parameters $\lambda \geq \lambda_0, m \leq m_0$, then $X$ is  $\delta-$relatively hyperbolic. 
\end{theorem}

\subsection{$M_T$ as a bundle over $\BU(T)$}\label{effcombbut} We shall now specialize and adapt the above results to the case that will be of relevance to us:
\begin{enumerate}
\item $P: \tmtdw \to \but$, and
\item  $P: \tmtdt \to \but$.
\end{enumerate}

\begin{rmk}\label{tmtdtcaveat}{\rm 
A word of caution is necessary here. It is easy to see that  $P:\tmtdw\to \but$ is a metric bundle  as per Definition \ref{def-mbdl}. However, 
 $P:\tmtdt\to \but$ violates the properness condition (Item 2 in Definition \ref{def-mbdl}). It also violates condition 5 (the strictly type-preserving condition) and hence condition 7 (the qi-preserving electrification condition) of Definition \ref{def-treerh}. We therefore need a way around these conditions. Instead of doing this in the fullest possible generality we shall simply focus on the relevant example, namely $$P:\tmtdt \to \but,$$ and proceed to check the properties of metric bundles and trees of relatively hyperbolic spaces that go through. Much of the discussion in the remainder of this subsection is aimed at addressing the issue just discussed and pointing out adaptations of existing arguments in the literature (particularly \cite{gautero,gautero-conv,mahan-sardar}) that help us circumvent it. }
\end{rmk}

We first observe that through every point of $\tmtdw, \tmtdt$ there exist uniform qi-sections.

\begin{lemma}\label{qisecexists} Given $g \geq 2$, there exists $\rho_0$ such that the following holds. 
Let $P : \tmtdw \to \but$ and $P : \tmtdt \to \but$ be as in Definition \ref{def-mg} with fiber $S$ of genus $g$. Then through every $x \in \tmtdw$   and  $x \in \tmtdt$ there exists a $\rho_0-$qi-section.
\end{lemma}

\begin{proof}
Since the partial electrification map from $\tmtdw$ to $\tmtdt$ is 1-Lipschitz, it suffices to prove the Lemma for $P: \tmtdw\to \but$.
Further, since sections can be lifted from $\mtdw$ to $\tmtdw$,  it suffices to prove the Lemma for $P: \mtdw\to \but$.

For a building block $M_v$ of special split geometry and $P: M_v \to T_v$ the natural projection onto the associated tree-link, there is an isometric section $\sigma_v: T_v \to M_v$ lying inside the Margulis riser $\RR_v = S^1 \times T_v$ since the latter is a metric product. 
The fibers $P^{-1}(z)$ of $P: \mtdw \to \but$ have diameter bounded
by some $D = D(g)$, by the Gauss-Bonnet Theorem. Choosing $\rho_0 = 2D +1$, we can construct a $\rho_0-$qi-section from $\but$ to $\mtdw$ by connecting the sections $\sigma_v$ using paths lying in the mid-surfaces.
\end{proof}

\noindent {\bf Ladders in trees of spaces:} We shall need the technology of ladders from \cite{mitra-trees,mitra-ct} below. We extract the necessary features from the ladder construction of \cite{mitra-trees,mitra-ct} and adapt it here to the language of hallways. The following is a restatement of \cite[Theorem 3.6]{mitra-trees}  in our context (see also the construction of the ladder in \cite[Section 3]{mitra-trees}). The corresponding statement
for $\tmtdt$ follows from Lemma \ref{qisecexists}.

\begin{theorem}\label{qcladder} Given $\delta \geq 0, K\geq 1, \ep\geq 0$ there exists $D$ such that the following holds.\\ Let 
	\begin{enumerate}
	\item  $(X,d)$ be either a tree of $\delta-$hyperbolic spaces  as in Definition \ref{tree} with parameters $K, \ep$ and let $X_v$ be a vertex space,
	\item or $X = \tmtdt$ with $P: \tmtdt \to \but$, and $X_v = P^{-1}(v)$ for some $v \in \but$.
	\end{enumerate}  Then for every geodesic segment $\mu \subset (X_v, d_v)$ there exists a $D-$qi-embedded subset $\LL_\mu$ of $X$ such that the following holds.
	\begin{enumerate}
		\item $X_v \cap \LL_\mu = \mu$,
		\item For $X = \tmtdt$ and every $w \in \but$, $X_w \cap \LL_\mu$ is a geodesic $\mu_w$ in $(X_w, d_w)$.
		\item For $X$ a tree of hyperbolic metric spaces and every $w \in T$, $X_w \cap \LL_\mu$ is either empty or a geodesic $\mu_w$ in $(X_w, d_w)$. Further, there exists a subtree $T_1 \subset T$ such that the collection of vertices $w \in T$ satisfying $ X_w \cap \LL_\mu \neq \emptyset$ equals the vertex set of $T_1$.
		\item There exists $ \rho_0\geq 1$ such that  through every $z \in \LL_\mu$, there exists a $ \rho_0-$qi-section $\sigma_z$ of $[v, P(z)]$ contained in $\LL_\mu$ satisfying $$ \sigma_z (P(z)) = z, \quad
		\sigma_z (v) \in \mu.$$
		\item There exist constants $\lambda_0, m_0, H_0$ such that for every $\mu_w = X_w \cap \LL_\mu$ the following holds:\\
		There is a hallway  $\HH_w$ bounded by $ \rho_0-$qi-sections as in (4) above containing $\mu_w$ satisfying the hallways flare condition with $\lambda\geq \lambda_0, m\leq m_0, H\leq H_0$. Further, $\HH_w\cap X_v$ is a geodesic subsegment of $\mu$.
	\end{enumerate}
Further, there exists a $D-$coarse Lipschitz retraction $\Pi_\mu:X \to \LL_\mu$, i.e.\
\begin{enumerate}
\item $d(\Pi_\mu(x), \Pi_\mu(y)) \leq D d(x, y) + D, \, \forall \, x, y \in X$,
\item $\Pi_\mu(x) = x, \, \forall \, x \in \LL_\mu$.
\end{enumerate}
\end{theorem}

The qi-embedded set $\LL_\mu$ is called a {\bf ladder} in \cite{mitra-ct,mitra-trees}. 
Theorem \ref{qcladder} shows in particular that there is a $(2D,2D)-$  quasigeodesic of $(X,d_X)$ joining the end-points of $\mu$ and lying on 
$\LL_\mu$. 

\begin{rmk}
Note that in Theorem \ref{qcladder}, we have {\bf not} assumed that $X$ is hyperbolic: no assumptions on the global geometry of $X$ are necessary here.
\end{rmk}

\noindent {\bf Ladders in $\tmtdt$ or $\tmtdw$:}
Given two $\rho_0-$qi-sections $X_1, X_2 \subset \tmtdt$ or $\tmtdw$ as in Lemma \ref{qisecexists}, we construct a {\bf ladder} $C(X_1, X_2)$ by joining the points 
$X_1 \cap F_b$ and $X_1 \cap F_b$ by a geodesic in $F_b$ (see \cite[Section 2.2]{mahan-sardar}). 	The  coarse Lipschitz retraction  property of Theorem \ref{qcladder} goes through in this context also. Further, in  Theorem \ref{qcladder}, the constant $D$  depends only on $\delta, K, \ep$. By Remark \ref{mbdl-mgbdl} we can pass from a metric bundle to a homogeneous tree of spaces. Unraveling definitions, $K, \ep$ depend on $R$ and parameters $D, k$ in Definition \ref{def-balancedtree} of an $L-$tight $R-$thick balanced tree. Now, for  $\tmtdt$ or $\tmtdw$,
$\delta$ depends only on the genus $g$. Thus we have the following:

\begin{lemma}\label{lem-qisxninhallway}
	Given $g \geq 2$ and $R, D, k \geq 1$, there exists $\rho_0$ such that the following holds. \\ Let $i: \vT \to \ccd (S)$ be an $L-$tight, $R-$thick balanced tree with parameters $D, k$ as in Definition \ref{def-balancedtree}. Let $ C(X_1, X_2)$ be a ladder in $\tmtdw$ or $\tmtdt$ as above. Then through every $x \in C(X_1, X_2)$, there exists a $\rho_0-$qi-section contained in $C(X_1, X_2)$.
\end{lemma}

The definition of hallways (Definition \ref{def-hallway}) now continues to make sense  for $P:\tmtdt \to \but$ and $P:\tmtdw \to \but$ with the following modification:
 the maps $f:[-m,m] \times \{0\} \to \tmtdt$, $f:[-m,m] \times \{1\} \to \tmtdt$ (or $\tmtdw$) in Definition \ref{def-hallway} are restrictions of $\rho_0-$qi-sections from $\but$ to $\tmtdt$, where $\rho_0$ is as in Lemma \ref{lem-qisxninhallway}. Note that  $P\circ f$ is an isometry onto its image. To distinguish from the hallways of Definition \ref{def-hallway}, we shall call them {\bf qi-section bounded hallways}. With this clarification, the  flaring condition
 of Definition  \ref{def-flare} continues to make sense for $P:\tmtdt \to \but$ and qi-section bounded hallways. We now state the following consequence of Theorem \ref{effectiveBF} in the form that we shall need it:

\begin{cor}\label{effectiveBFfwdtmtdt} 
	Given $ \lambda_0, m_0, H_0, \delta_0$ there exists $\delta > 0$ such that the following holds:\\
 For $b \in \but$, 	let $F_b = P^{-1} (b)$ equipped with the induced path metric and suppose that  $F_b$ is $\delta_0-$hyperbolic for all $b \in \but$.
	If  $P:\tmtdt \to \but$ satisfies the   flare
	condition with constants $\lambda \geq \lambda_0, 1\leq m \leq m_0, 1\leq H \leq H_0$ for  qi-section bounded hallways, then $\tmtdt$ is  $\delta-$hyperbolic. 
\end{cor}

\begin{proof} The proof is a transcription of the relevant steps from 
	\cite{mahan-sardar} and \cite[Theorem 5.2]{gautero} and we only give a sketch.\\
	
	\noindent {\bf Step 1:} Lemma \ref{qisecexists} guarantees the existence of $\rho_0-$qi-sections through every point and $\rho_0-$qi-sections in ladders. This replaces \cite[Proposition 2.12]{mahan-sardar}.\\
	
		\noindent {\bf Step 2:}  
		Now \cite[Theorem 3.2]{mahan-sardar} shows that $C(X_1, X_2)$ is $C-$qi-embedded in $\tmtdt$ where $C$ depends only on $\delta_0$
		and the parameters $R, D, k$ of the $L-$tight $R-$thick balanced tree  (see Definition \ref{def-balancedtree} and also Lemma \ref{lem-qisxninhallway} for the dependence on constants).\\
		
	\noindent {\bf Step 3:} Then  $C(X_1, X_2)$ is a bundle over $\but$ with fibers closed intervals. Further, it satisfies the flare condition with respect to qi-section bounded hallways.
		We now invoke  Theorem \ref{effectiveBF} to conclude that there exists $\delta_1$ such that each  $C(X_1, X_2)$ is $\delta_1-$hyperbolic. Note that it is at this step that we are circumventing the  use of properness of the metric bundle as in \cite[Section 3]{mahan-sardar} by using \cite{gautero} instead, cf.\ Remark \ref{tmtdtcaveat}. We recall that the proof given by Gautero of Theorem \ref{effectiveBF} in \cite[Theorems 2.20, 5.2]{gautero}  does not use properness  of the total space and proceeds by directly   deducing effective hyperbolicity from exponential divergence of geodesics. The last condition (exponential divergence of geodesics) in turn is an immediate consequence of flaring. In particular, the proof in \cite{gautero} does not go via the original linear isoperimetric inequality proof of \cite{BF}.\\
		
		\noindent {\bf Step 4:} The rest of the proof  follows \cite[Section 4]{mahan-sardar}. Given any 3 points, $x, y, z \in \tmtdt$, let $X_x, X_y, X_z$ be 	$\rho_0-$qi-sections through $x, y, z$ respectively. The union of the ladders $C(X_x,X_y)$, $C(X_y,X_z)$, $C(X_x,X_z)$ is denoted as $C(X_x,X_y, X_z)$ is called a {\bf tripod-bundle} in \cite[Definition 4.1]{mahan-sardar}. Let $\phi_b(x,y,z) $ denote a barycenter in $F_b$ of
		$(X_x\cap F_b), (X_y\cap F_b), (X_z\cap F_b)$. Then the set $$X_b:=\{\phi_b(x,y,z)\vert b \in \but\}$$ gives a qi section  \cite[Proposition 4.2]{mahan-sardar}. The tripod-bundle  $C(X_x,X_y, X_z)$
	 can be $\delta_0-$approximated  by the union of three ladders $C(X_b,X_x), C(X_b,X_y), C(X_b,X_z)$ and any two of them intersect along $X_b$.\\
	 
	 \noindent {\bf Step 5:} By Step (3) above, each of $C(X_b,X_x), C(X_b,X_y), C(X_b,X_z)$ is $\delta_1-$hyperbolic and they all intersect along the qi-embedded subset $X_b$. Hence by Theorem \ref{effectiveBF}, there exists $\delta_2$ depending only on $\delta_1$ and the qi-embeddedness constant $\rho_0$ of $X_b$ (see Lemma \ref{lem-qisxninhallway}) such that
	 $$C(X_b,X_x) \cup C(X_b,X_y) \cup C(X_b,X_z)$$ is $\delta_2-$hyperbolic.\\
	 
	  \noindent {\bf Step 6:} Finally, by a standard path-family argument (see \cite[Theorem 4.3]{mahan-sardar})  $\tmtdt$ is  $\delta-$hyperbolic, where $\delta$  depends only on $\delta_0$ (the hyperbolicity constant of fiber spaces) $R, D, k$ (the parameters of the $L-$tight $R-$thick balanced tree). 
\end{proof}

For the converse direction, we refer the reader to Section 5.3 of \cite{mahan-sardar}, which proves the necessity of flaring. We  briefly indicate how to adapt the argument here. First, $\delta-$hyperbolicity of $\tmtdt$ guarantees that there exists $H$ (depending on $\delta$) such that qi-section-bounded hallways of girth (cf.\ Definition \ref{def-rhothin}) lying between $H$ and $H+1$ flare (see Lemma 5.9 of \cite{mahan-sardar}) so long as $\rho_0$ is chosen (again depending on $\delta$) to ensure that $\rho_0-$thin hallways exist connecting a point of $P^{-1}(z_1)$ to some point of $P^{-1}(z_2)$  for any $z_1, z_2 \in \but$ with $d_\but (z_1,z_2)\leq 1$.
Next, \cite{mahan-sardar}  (see the paragraph in  
\cite[Section 5.3 ]{mahan-sardar}
called `Flaring of general ladders') shows how to decompose a general hallway into flaring hallways of girth  between $H$ and $H$. Thus we conclude the converse direction of Theorem \ref{effectiveBF} for $P: \tmtdt \to \but$:

\begin{cor}\label{effectiveBFreversetmtdt} Given $\delta>0, \rho_0$,  there exist $\lambda_0, m_0, H_0$ such that the following holds:\\
	If $\tmtdt$ is $\delta-$hyperbolic and $\rho_0$ is as in Lemma \ref{qisecexists}, then $\tmtdt$  satisfies the
		  hallways flare condition with respect to $\rho_0-$qi-section bounded hallways,
		 with constants $\lambda \geq \lambda_0, m \leq m_0, H \leq H_0$.
\end{cor}

Finally, we shall combine Corollary \ref{effectiveBFreversetmtdt} with Lemma \ref{pel}. To do this, observe that for $P: (M_T, \dw) \to \but$, any tree-link $T_v \subset \but$, $z\in T_v$, the pre-image $P^{-1}(z) =S_z$ is of uniformly bounded geometry. Hence, 
\begin{enumerate}
\item The fibers $(\til{S_z}, \dw)$ of $P:\tmtdw \to \but$ are uniformly hyperbolic.
\item The fibers $(\til{S_z}, \dt)$ of $P:\tmtdt \to \but$ are uniformly hyperbolic as these are obtained by electrifying uniformly separated (independent of $z$) uniform quasigeodesics (again with constant independent of $z$) in $(\til{S_z}, \dw)$.
\end{enumerate}

We denote the collection  of Margulis risers as $$\RR_\MM \ := \ \{v \times T_v \vert T_v \subset \but \ {\rm is \ a \, tree-link} \},$$ 
and the set of all lifts of $\mr$ to $\til M_T$ as $\tmr$.

\begin{prop}\label{relhypconvtmt} Given $\delta, \rho_0 >0$,  there exist $\lambda_0 > 1, m_0\geq  1, H_0 \geq 0$ and $C\geq 0$ such that the following holds:\\
If $\tmtdw$ is strongly $\delta-$hyperbolic  relative to $\tmr$ and  $\rho_0$ is as in Lemma \ref{qisecexists}, then $\tmtdt$  satisfies the
hallways flare
condition  with respect to $\rho_0-$qi-section bounded hallways, with constants $\lambda \geq \lambda_0, m \leq m_0, H \leq H_0$. Further, each element of $\tmr$ is $C-$quasiconvex in $\tmtdt$.
\end{prop}

\begin{proof} We first observe that $\tmtdt$ is obtained from $\tmtdw$ by partially electrifying the $\R-$directions in $\R \times T_v$ for every lift $\R \times T_v$ of a Margulis riser to $\til M_T$. We  now   consider the quadruple $(X,\HH,\GG,\LL)$ with
	\begin{enumerate}
		\item $\tmtdw$ in place of $X$,
		\item $\tmr$ in place of $\HH$, 
		\item Indexing the elements of $\tmr$ by  $\tmr_\alpha$, define
		$$g_\alpha: (\tmr_\alpha,\dw) \to (\tmr_\alpha,\dt)$$ to be the map that partially electrifies the $\R-$directions in $\R \times T_v$ for every lift $\R \times T_v$ of a Margulis riser. Then $\GG$ is the collection of maps $g_\alpha$ and $\LL$ is the collection of spaces $(\tmr_\alpha,\dt)$.
	\end{enumerate}
	Lemma \ref{pel} applied to this quadruple $(X,\HH,\GG,\LL)$ then shows that there exist $\delta_0, C \geq 0$ such that 
	\begin{enumerate}
	\item $\tmtdt$ is $\delta_0-$hyperbolic.
	\item $(\tmr_\alpha,\dt)$ is $C-$quasiconvex in $\tmtdt$ for every $\alpha$.
	\end{enumerate}
	This proves the last statement of the proposition. The first statement now follows from Corollary \ref{effectiveBFreversetmtdt}.
\end{proof}

\subsection{Effective quasiconvexity and flaring}\label{sec-effecqc} The main purpose of this subsection is to prove  Proposition \ref{effectiveqc}. We shall apply it in its full strength in the companion paper 
\cite{mms}. For the purposes of this paper, it is used mildly in the proofs of
Propositions \ref{telunifhyp} and \ref{thinbdlhyp}.  Proposition \ref{effectiveqc} may  be regarded as
a fact supplementing the effective hyperbolicity and relative hyperbolicity
Theorems \ref{effectiveBF} and \ref{effectiverelBF}.

For the purposes of this subsection, $X$ will be 
\begin{enumerate}
\item Either a tree ($T$) of hyperbolic metric spaces satisfying the 
qi-embedded condition with constants $K, \ep$ and the  hallways flare
condition with constants $\lambda_0, m_0$. Further, if $\rho_0$ is given we shall assume an additional constant $ H_0$ as a lower bound  for girths of $\rho_0-$thin hallways. $X$ is equipped with the usual projection map $P: X \to T$.
\item OR $\tmtdt$ corresponding to an $L-$tight $R-$thick tree $T$. $P: \tmtdt \to \but$ will denote the usual projection map. The constant $\rho_0$ will be as in Lemma \ref{qisecexists} and the constants $\lambda_0, m_0, H_0$ will be as in Corollary \ref{effectiveBFreversetmtdt}.
\end{enumerate}

Also $(X_v,d_v)$ will, respectively, be a vertex space of $X$ (in the tree of spaces case) or $P^{-1}(v)$ (in the $P: \tmtdt \to \but$ case) and $Y\subset (X_v,d_v)$ will be a $C-$quasiconvex subset of $(X_v,d_v)$.

\begin{defn}\label{def-flareinall}
We shall say that $Y$ {\bf flares in all directions with parameter $K$} if for any geodesic segment $[a,b] \subset (X_v,d_v)$ with $a, b \in Y$ and any $\rho-$thin hallway $f:[0,k] \times I \to X$ satisfying 
\begin{enumerate}
	\item $\rho \leq \rho_0$,
\item $f(\{0\} \times I) = [a,b]$,
\item $l([a,b]) \geq K$,
\item $k \geq K$,
\end{enumerate}
the length of $f(\{k\} \times I)$ satisfies $$l(f(\{k\} \times I)) \geq \lambda l( [a,b]).$$
\end{defn}

Proposition \ref{effectiveqc} below is probably well-known to  experts (at least for trees of spaces) but we could not find an explicit statement in the literature.

\begin{prop}\label{effectiveqc} Given $K, C$, there exists $C_0$ such that the following holds.\\
Let $P: X \to T$ (or $P: \tmtdt \to \but$) and $X_v$ be as in Theorem \ref{qcladder} above. If $Y$ is a $C-$quasiconvex subset of $(X_v,d_v)$ and flares in all directions with parameter $K$, then $Y$ is  $C_0-$quasiconvex in $(X,d_X)$.

Conversely, given $C_0$, there exist $K, C$ such that the following holds.\\
For $P: X \to T$ (or $P: \tmtdt \to \but$) and $X_v$  as above, if $Y \subset X_v$ is   $C_0-$quasiconvex in $(X,d_X)$, then it is  $C-$quasiconvex subset in $(X_v,d_v)$ and flares in all directions with parameter $K$.
\end{prop}

\begin{proof} We first prove the forward direction. If the conclusion fails, then though $Y$ flares in all directions, it is not quasiconvex in $(X,d_X)$. In particular, for every $n \in \natls$, there exists $\mu \subset X_v$ with end-points in $Y$ such that there exists a $(2D,2D)-$  quasigeodesic $\mu^R$  (of $(X,d_X)$) joining the end-points of $\mu$, lying on 
	$\LL_\mu$ and leaving the $n-$neighborhood of $\mu$. 
	Hence there exists a vertex $w$ of $T$ such that 
	\begin{enumerate}
	\item $d_T(v,w) = O(n)$,
	\item $\mu^R\cap X_w$ contains a pair of points  $a',b'$ such that $d_w(a',b') $ is minimal amongst lengths that exceed the minimal girth ($H(\rho_0)$ in Definition \ref{def-flare}) required for flaring (see figure below). Since the flaring constant $\lambda$ is fixed,  it follows that
	$d_w(a',b') \leq \lambda H(\rho_0)$; in particular,  $d_w(a',b')$
	is uniformly bounded.
	\end{enumerate}

\begin{center}
	
	\includegraphics[height=6cm]{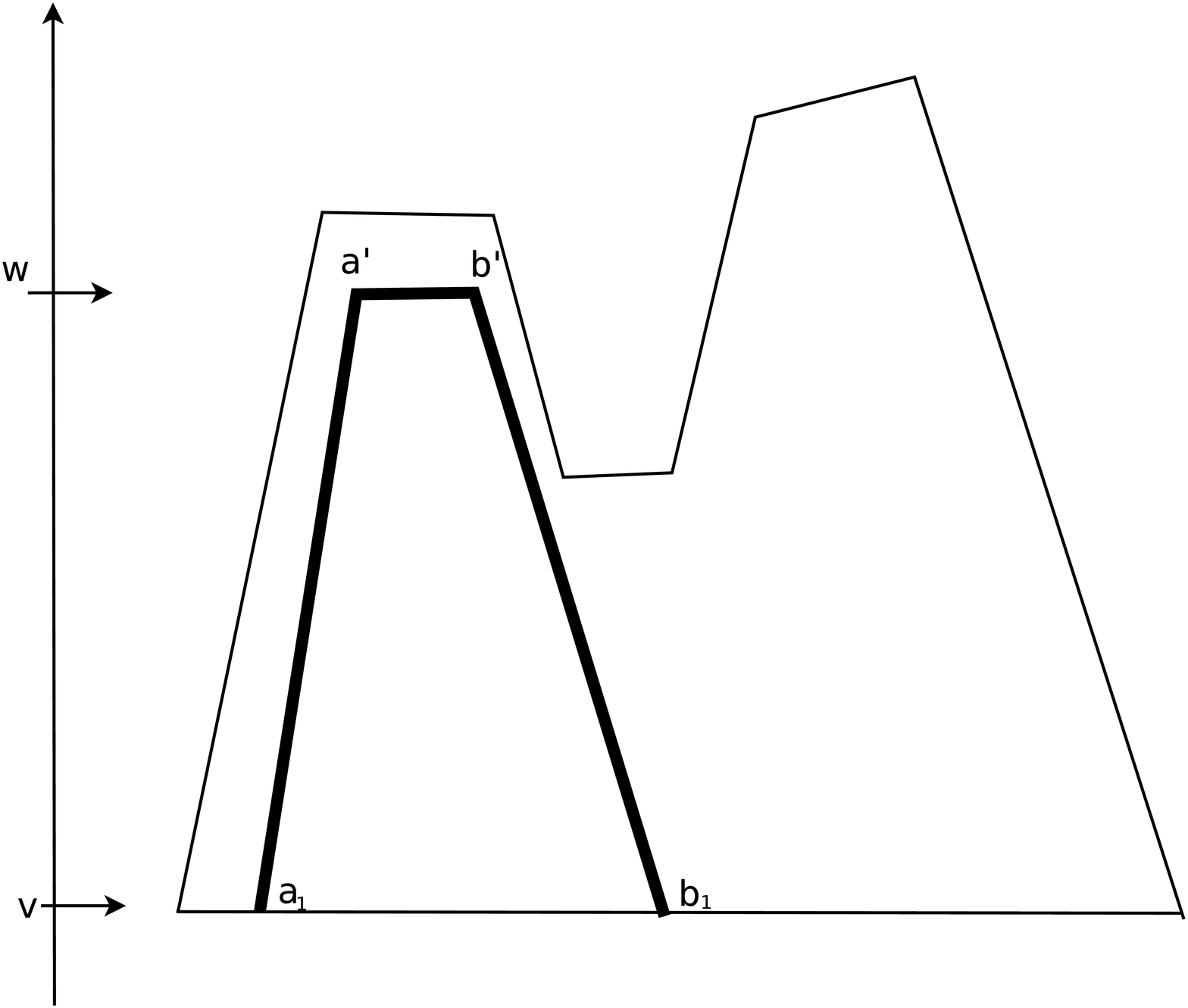}

	\smallskip
	
	\underline{Figure:  {\it Finding a flaring trapezium} }
	
\end{center}
Let $\mu_w$ be a geodesic in $(X_w,d_w)$ joining $a',b'$. By Theorem \ref{qcladder} it is contained in a hallway	$\HH_w$ such that $\HH_w\cap X_v$ is a geodesic subsegment $\mu_0$ of $\mu$. Since $Y$ is $C-$quasiconvex in $(X_v,d_v)$, there exist $a_1, b_1\in Y$ close to the end-points of $\mu_0$.

Hence there exists a hallway $\HH_w'$ (with slightly worse constants than $\HH_w$, see Lemma \ref{lem-qisxninhallway})) such that 
\begin{enumerate}
\item $\HH_w' \cap X_w = [a',b']$,
\item $\HH_w' \cap X_v = [a_1,b_1]$.
\end{enumerate}

In particular (since $d_w(a',b') =O(1)$ is uniformly bounded), the geodesic $[a',b']$ {\it does not} flare in the direction $[v,w]$ (choosing $n$ large enough). This contradiction proves the forward direction.

	We now prove the converse direction. Since $Y$ is $C_0-$quasiconvex in  $(X,d_X)$,  it is  $C_0-$quasiconvex  in $(X_v,d_v)$ the latter being a subspace of the former. Next, since $Y$ is $C_0-$quasiconvex in  $(X,d_X)$, the following holds.\\  Let
	\begin{enumerate}
		\item  $a, b \in Y$ be vertices with $d_v(a,b)$ large enough,
		 \item  $[v,w] \subset T$ (or $\but$) be a geodesic segment starting at $v$. Let $\sigma_a, \sigma_b$ be two qi-sections (with uniform constant $K_0$) of $[v,w]$ for
		 $P: X \to T$ (or $P: \tmtdt \to \but$). 
	\end{enumerate}
Then $\sigma_a, \sigma_b$ must flare with flaring constants depending on 
$K_0$ as soon as $d_T(v,w)\geq K$ (or $d_{\but}(v,w)\geq K$) for some $K$ depending only on $C$. 
This is  a simple quasification of the standard fact that
geodesics diverge 
exponentially in a
hyperbolic metric space (see 
\cite[Proposition 2.4]{mitra-endlam} for instance). Since $w\in T$ (or $\but$) was arbitrary, it follows that $Y$ flares in all directions with parameter $K$.
\end{proof}

\section{Uniform hyperbolicity of $M$}\label{sec-minsky} In this section, we establish {\bf uniform} estimates for the Gromov hyperbolicity of $\tmtdt$. 	We restate Theorem \ref{maintech} in the form that we shall prove it.

\begin{theorem}\label{effectivehyptmtdt}
Given. $R\geq 1$, and $D, k\geq 1$ there exists $\delta >0$ such that the following holds:\\
For  an $L-$tight  $R-$thick balanced tree  $T$ with parameters  $D, k\geq 1$,
\begin{enumerate}
		\item $\tmtdt$ is  $\delta-$hyperbolic.
	\item $\tmtdw$ is strongly $\delta-$hyperbolic relative to the collection $\tmr$ of lifts of Margulis risers,
\end{enumerate} 
\end{theorem}

Note that by Definition \ref{def-geometricbb} and  Corollary \ref{subordinatehierarchysmall}, the hypothesis on existence of $R$ in Theorem \ref{effectivehyptmtdt} is equivalent to the existence of $k_0\geq 1, \ep_0 \geq 0$  such that $M_T$ 
is
a special split geometry model with  parameters $k_0, \ep_0$ corresponding to $T$.

The proof of Theorem \ref{effectivehyptmtdt} will be given in Section \ref{unifhypcx} and will use
\begin{enumerate}
\item The fact that the Minsky model  for doubly degenerate Kleinian surface groups with injectivity radius uniformly bounded below is uniformly bi-Lipschitz to the hyperbolic metric \cite{minsky-jams,minsky-elc1,minsky-elc2}.
\item The Bestvina-Feighn combination theorem \cite{BF} and its converse in the effective form given by Corollaries \ref{effectiveBFfwdtmtdt}, \ref{effectiveBFreversetmtdt} and Proposition \ref{relhypconvtmt}.
\item The special split geometry of building blocks.
\end{enumerate}

For the purposes of this section, $N$ will denote a doubly degenerate hyperbolic 3-manifold corresponding to a surface $S$ and a a doubly degenerate surface Kleinian group $\rho (\pi_1(S)) = \pi_1(N) \subset PSL(2, \mathbb{C})$. The ending laminations of $N$ are denoted as $l_\pm$. Note that by work of Thurston \cite[Chapter 9]{thurstonnotes} and Bonahon \cite{bonahon-bouts},
$N$ is homeomorphic to $S \times \R$. Before dealing with the model $\tmtdw$ of special split geometry  and proving Theorem \ref{effectivehyptmtdt}, it will be convenient to focus on the simpler case of bounded geometry. This furnishes the same result under stronger hypotheses (Proposition \ref{thickbdlhyp}) and will serve to delineate the ingredients of the proof.

 We first recall from the Introduction some of the basics of convex cocompact
 subgroups of the mapping class group and refer the reader to \cite{farb-coco} for details. As before, $S$ is a closed surface of genus $g$ and $MCG(S)$ is its mapping class group.
 A subgroup $H$ of  $MCG(S)$ is said to be {\bf convex cocompact} if some (every) orbit of $H$ in the Teichm\"uller space $Teich(S)$
 is quasiconvex. Associated to any $H \subset MCG(S)$, there is a natural
 associated exact sequence \cite[Section 1.2]{farb-coco} of the form $$1 \to \pi_1(S) \to L_H \to H \to 1.$$ The following characterizes  convex cocompactness:
 
 \begin{theorem} \cite{farb-coco, hamen} \label{charzn}
 	A subgroup $H$ of  $MCG(S)$ is  convex cocompact if and only if the extension $L_H$ occurring in the 
 	associated exact sequence  $1 \to \pi_1(S) \to L_H \to H \to 1$ is hyperbolic.
 \end{theorem}
 
 Theorem \ref{charzn} was proved for free groups by Farb and Mosher \cite{farb-coco} as was the  `if' direction in general.
 Hamenstadt \cite{hamen} proved the only if direction.
 In \cite[Proposition 5.17]{mahan-sardar} this was extended to surfaces with punctures. The proof there was in fact effective (see also \cite{gautero, mahan-reeves}). We shall recall this in Section \ref{sec-thickcusps}.

 The next statement observes the absence of $\Z \oplus \Z$ in extensions of purely pseudo-Anosov subgroups of $MCG(S)$.
 
 \begin{prop}\cite[Theorem 8.1]{kl-survey} \label{klpa} Let $H \subset MCG(S)$. If $H$ is purely pseudo-Anosov, then $L_G$ contains no Baumslag-Solitar subgroups and hence no copy of $\Z \oplus \Z$. 
 \end{prop}
 
 For convenience of the reader, we outline the strategy that will go into the proof of Proposition \ref{thickbdlhyp}. We shall modify this strategy to prove Theorem \ref{effectivehyptmtdt} in Section \ref{unifhypcx}.
 \begin{scheme}\label{scheme-thick}
 	The steps of the proof of Proposition \ref{thickbdlhyp} are:
 	\begin{enumerate}
 		\item The Minsky model (Theorem \ref{minsky-elcbddgeo}) shows that the universal bundles over bi-infinite geodesics are uniformly hyperbolic.
 		\item The converse direction of the combination Theorem \ref{effectiveBF} furnishes effective flaring constants.
 		\item Feeding these  effective flaring constants into the bundle $\til{M_H}$  over $\Gamma_H$ furnishes (effective) hyperbolicity of $\til{M_H}$.
 	\end{enumerate}
 \end{scheme}

\subsection{Thick Minsky model: No cusps}\label{sec-thick} We now turn to proving the analog of Theorem \ref{effectivehyptmtdt} for bounded geometry. By Theorem \ref{minsky-elcbddgeo} the bounded geometry hypothesis is equivalent to the (union of the) assumptions that 
\begin{enumerate}
\item The parameter $R$ in the underlying $L-$tight and $R-$thick tree (cf.\ Definition \ref{def-rthick}) is uniformly bounded above.
\item There exists $L' \geq L$ such all the subsurface projections onto 
$S \setminus i(v)$ are bounded by $L'$ for all $v \in V(T)$. (Note that this
is stronger than in the statement of Theorem \ref{effectivehyptmtdt}, where
only $R$ is bounded above.)
\end{enumerate}

 For the time being, we focus on the case of closed $S$. Let $l$ be an $\ep-$thick  bi-infinite Teichm\"uller geodesic (i.e.\ a bi-infinite Teichm\"uller geodesic contained in $\Teich_\ep (S)$) with end-points $l_\pm \in \partial \Teich(S) = \PML(S)$. By forgetting the underlying measure, we identify $l_\pm $ with the underlying elements of the ending lamination space $\EL (S)$. Let $M_l$ be the universal curve over $l$ equipped with the universal curve metric as in Remark \ref{ubundle}. Then, by \cite{minsky-bddgeom,minsky-jams} $M_l$ is uniformly bi-Lipschitz to  the unique hyperbolic manifold $N(l\pm)$ with ending laminations $l_\pm$.
  As a consequence of Theorem \ref{minsky-elcbddgeo}, we thus have

\begin{cor}\label{minsky-elcbddgeocor}
For $S=S_{g,0}$ a closed surface of genus $g$, and $\ep >0$, there exists $\delta > 0$ such that the following holds:\\ For $l$  an $\ep-$thick  bi-infinite Teichm\"uller geodesic,   the universal cover $\til M_l$ of the  Minsky model $M_l$, equipped with the universal curve metric, is $\delta-$hyperbolic.
\end{cor}

\begin{proof}
This follows from the fact that $M_l$ is $K-$bi-Lipschitz homeomorphic to a hyperbolic manifold $M(l\pm)$, with $K$ depending only on $g, \ep$; and hence $\til M_l$ is  $K-$bi-Lipschitz homeomorphic to  $\Hyp^3$.
\end{proof}

Uniform
hyperbolicity of $\til M_l$ in Corollary \ref{minsky-elcbddgeocor} ensures uniform flaring constants by the converse part of Theorem \ref{effectiveBF}:

\begin{cor}\label{effectiveflare}
	For $S=S_{g,0}$ a closed surface of genus $g$, and $\ep >0$, there exists $ \lambda_0 , m_0, \rho_0, H_0 \geq 1$ such that the following holds:\\
	Let  $l$ be an  $\ep-$thick  bi-infinite Teichm\"uller geodesic and $P: M_l \to l$ denote the universal bundle over $l$. Let  $P:\til{M_l} \to l$ denote the lift to the universal cover. Then through every point of $\til{M_l}$ there exists a $\rho_0-$qi-section of $P:\til{M_l} \to l$. Further, hallways bounded by $\rho_0-$qi-sections in $\til{M_l}$ satisfy the flaring condition with constants  $\lambda \geq \lambda_0$, $n \leq n_0$ and $H \leq H_0$.
\end{cor}

We shall say that 
a subgroup $H$ of  $MCG(S)$ is  {\bf $K-$convex cocompact} if some orbit of $H$ in $Teich(S)$
is $K-$quasiconvex. Hence there exists $o \in Teich(S)$, such that for every $l_\pm \subset \partial H \subset \partial Teich(S)$, the Teichm\"uller geodesic $l$ joining $l_\pm$ lies
at bounded Hausdorff distance $D (=D(K))$ from $H.o$. 
For $a, b \in \partial H \subset \partial Teich(S)$ the Teichm\"uller geodesic $l$ joining $a,b$ is denoted as $l_{ab}$.
Next, assume that $H$ is free. 

\begin{constr}\label{thickbdlcoco} Let $H$ be a free, convex cocompact, purely pseudo-Anosov subgroup of $MCG(S)$.
We can  choose a free generating set for $H$, construct a Cayley graph $\Gamma_H$ of $H$ and also  a map $\Phi: \Gamma_H \to Teich(S)$, such that

\begin{enumerate}
\item $\Phi(1)  = o$,
\item $\Phi$ maps edges of $\Gamma_H$ to geodesic segments,
\item For $a, b \in \partial \Gamma_H$, let $(a,b)$ denote the bi-infinite geodesic joining $a, b$ in $\Gamma_H$. Then $\Phi ((a,b))$
and $l_{ab}$ lie within bounded Hausdorff distance $D (=D(K))$ from each other. Further, we can (after choosing $D$ depending only on $S$ and $K$ appropriately) parametrize $(a,b)$ and $l_{ab}$ proportional to their respective arc lengths, such that $d_{Teich} (\Phi(t), l_{ab}(t)) \leq D$.
\item  The universal curve  over $\Phi ((a,b))$ is denoted as $M_{ab}$.
\end{enumerate}
\end{constr}

The following is now a consequence of Corollary \ref{minsky-elcbddgeocor}
(see also \cite{minsky-bddgeom,rafi-gt}):

\begin{cor}\label{thickTeichimpliesthick3}
	 For $K \geq 0$ and $e >0$, there exists $\delta > 0$ such that if
	 \begin{enumerate}
	 \item  $H$ is a $K-$convex cocompact subgroup, and
	 \item there exists $o \in Teich(S)$ with $H.o \subset Teich_e(S)$,
	 \end{enumerate} 
	 then for all $a,b \subset \partial H \subset \partial Teich(S)$, the universal curve $M_{ab}$ over $\Phi ((a,b))$ with ending laminations  $a, b$ is $\delta-$hyperbolic.
\end{cor}

\begin{prop}\label{thickbdlhyp} Given $K, \ep \geq 0$, there exists $\delta >0$ such that the following holds:\\
	Let $H$ be a free $K-$convex cocompact subgroup and  let $o \in Teich(S)$ with $H.o \subset Teich_\ep(S)$. Let $\Gamma_H$ be a Cayley graph of $H$ with respect to a free generating set and $\Phi: \Gamma_H \to Teich(S)$ be as in Construction \ref{thickbdlcoco} above.
	Let $M_H$ be the universal bundle  over $\Phi (\Gamma_H)$ (equipped with the universal bundle metric as before). Then the universal cover $\til{M_H}$ is $\delta-$hyperbolic.
\end{prop}

\begin{proof} For $a, b \in \partial \Gamma_H \subset \partial \Teich(S)$, let $l_{ab}$ denote the Teichm\"uller geodesics joining $a, b$ and let $(a,b)$ denote the bi-infinite geodesic in $\Gamma_H$ joining $a, b$. By $K-$convex cocompactness and $\ep-$thickness, there exists $e'$ such that $l_{ab}$   lies in the $e'-$thick part of Teichm\"uller space for all $a, b \in \partial H$.
	Let $M_{ab}$ denote the universal curve  over $\Phi ((a,b))$.
	Then by Corollary \ref{thickTeichimpliesthick3}, there exists $\delta'$ such that  $\til M_{ab}$ is  $\delta'-$hyperbolic.
	
	Let $P: \til M_{ab} \to  (a,b)$ denote the natural projection. 
	By Lemma \ref{qisecexists}, there exists $\rho_0$ such that through every point of $\til M_{ab}$ there exists a $\rho_0-$qi-section of $P$.
	By (a straightforward quasification of) Corollary \ref{effectiveflare}, there exist $\lambda_0, m_0, H_0$ (depending only on $K, \epsilon > 0$), such that the
	 universal cover $\til {M_{{ab}}}$ of the universal curve $M$ over $\Phi ((a,b))$
	 satisfies  the flaring condition with $\lambda \geq \lambda_0$, $m \leq m_0$ and $H \leq H_0$ with respect to $\rho_0-$qi-section bounded hallways.
	
	Hence,  by (the forward part of) Theorem \ref{effectiveBF}, there exists $\delta >0$ depending only on  $\lambda_0, m_0, \rho_0$ (and hence only on $K, \epsilon > 0$) such that  $\til M_H$ is $\delta-$hyperbolic. 
	\end{proof}
	
	\begin{rmk}\label{thickbdlhyprmk}
	We note here that Proposition \ref{thickbdlhyp} and its proof  go through if $\Gamma_H$ is replaced by any convex subset of $\Gamma_H$, i.e.\ by a connected sub-tree of $\Gamma_H$. All we need to do is assume that the image of the convex subset (instead of the image of the whole Cayley graph) is $K-$quasiconvex  and that it lies in the $e-$thick part of $Teich(S)$. 
	\end{rmk}

\subsection{Thick Minsky model: Cusped case}\label{sec-thickcusps} We describe now a relative version of Proposition \ref{thickbdlhyp} when $S$ has cusps. Though we shall not need it directly, we provide a statement and a sketch, 
as the proof is a fairly straightforward combination of
\cite[Proposition 5.17]{mahan-sardar} and the proof of Proposition \ref{thickbdlhyp} above. 
We  now state the quantitative version of \cite[Proposition 5.17]{mahan-sardar}:

\begin{prop} \label{coco} Given $K, e$, there exists $\delta$ such that the following holds:\\
	Let $N=\pi_1(S)$  be the fundamental group of a surface $S(=S_{g,n})$ with $n$ punctures.
	Let $N_1, \cdots, N_n$ be the cyclic peripheral subgroups.  Let $H$ be a $K-$convex cocompact subgroup of the  
	{\em pure} mapping class group of $S$ having an orbit $H.o \subset Teich_e(S)$. Let
	$$
	1\rightarrow N \rightarrow G\stackrel{p}{\rightarrow}H\rightarrow 1
	$$ be the induced exact sequence.
	The action of $H$ centralizes each $N_i$.
	Let
	$$
	1\rightarrow N_i\rightarrow Z_G(N_i) \stackrel{p}{\rightarrow} H \rightarrow 1,
	$$
	be the induced short exact sequences of peripheral  groups, where $Z_G(N_i) = N_i \times H$ denotes the normalizer (equal to the centralizer) of $N_i$ in $G$.
	Then
	$G$ is  strongly $\delta-$hyperbolic  relative to the collection $\{ N_G(K_i)\}, i=1, \cdots, n$. 
	
	Conversely, if $G$ is  (strongly) hyperbolic  relative to the collection $\{ N_G(N_i)\}, i=1, \cdots, n$,
	then $H$ is convex-cocompact.  \end{prop}

We now specialize to our case of interest, where $H$ is free:

\begin{cor} Let  $S(=S_{g,n})$ be as in Proposition \ref{coco}. Given $K, e \geq 0$, there exists $\delta >0$ such that the following holds:\\
	Let $H$ be a free $K-$convex cocompact subgroup and let $o \in Teich(S)$ with $H.o \subset Teich_e(S)$. Let $\Gamma_H$ be a Cayley graph of $H$ with respect to a free generating set and $\Phi: \Gamma_H \to Teich(S)$ be as in Construction \ref{thickbdlcoco}.
	Let $M_H$ be the universal bundle  over $\Phi (\Gamma_H)$ (equipped with the universal bundle metric as before) with a neighborhood of the cusps removed.
	Let $S_0$ denote $S$ with the corresponding neighborhoods of the $n$ punctures removed.
	Let $\PP_0$ denote the connected components of $\partial S_0 \times \Phi( \Gamma_H)$. Let $\PP$ denote the collection of lifts of $P_0 \in \PP_0$ to   the universal cover $\til{M_H}$. Then $\til{M_H}$ is strongly $\delta-$hyperbolic relative to the collection $\PP$.
	\label{thickbdlrelhyp}
\end{cor}

\begin{proof}
We sketch a proof of the Corollary  carrying forward the notation from Proposition \ref{thickbdlhyp}.  First, by $K-$convex cocompactness, the universal curves $M_{ab}$ have systole bounded below by some $\ep' (= \ep'(K, \ep))$. 
Next, electrify the cusps of $S$. This gives us 
\begin{enumerate}
\item A tree of spaces where all the vertex and edge spaces are quasi-isometric to $(\til S, d_e)$ with electrified horocycle boundary.
\item The universal cover $(\til M_{ab}, d_{pel})$ of the universal curve $M_{ab}$  over $\Phi ((a,b))$  is consequently equipped with the {\it partially electrified} metric $d_{pel}$.
\end{enumerate} 

As in the proof of \cite[Proposition 5.17]{mahan-sardar} (cf.\ Proposition \ref{coco}
above) and Corollary \ref{effectiveflare}, the resulting tree of spaces satisfies a uniform flaring condition, i.e.\ there exist $ \lambda_0, m_0, \rho_0, H_0$ (depending only on $K, \epsilon > 0$), such that  $(\til M_{ab}, d_{pel})$ 
satisfies  $( \lambda, m, \rho)-$flaring with  $\lambda \geq \lambda_0$, $m \leq m_0$ and $\rho \leq \rho_0, H \leq H_0$.

	Hence,  by (the forward part of) Theorem \ref{effectiverelBF}, there exists $\delta >0$ depending only on $ \lambda_0, n_0, \rho_0$ such that  $\til M_H$ is strongly $\delta-$hyperbolic relative to the collection $\PP$. 
\end{proof}

\subsection{Uniform hyperbolicity of $\tmtdt$}\label{unifhypcx} 
We are now in a position to prove Theorem \ref{effectivehyptmtdt}. 
Starting with a balanced tree $i: \vT \to \ccd (S)$, let $\but$ denote the blown up tree. For $l\subset T$  a bi-infinite geodesic, $BU(l)$ will denote its blow-up in $\but$. The end-points of  $BU(l)$ in $\EL(S) =\partial \CC(S)$ will be denoted by $l_\pm$. We remind the reader of Standing Assumption \ref{assumption} about $L-$tight trees.

\begin{scheme}\label{scheme-split} {\rm We now outline the steps of the proof of Theorem \ref{effectivehyptmtdt} and the modifications  to Scheme \ref{scheme-thick} that we require.
\begin{enumerate}
	\item[Step 1:] Let $(M_l, \dw)$ (resp. $(M_l, \dt)$) denote the bundle $P: (M_T,\dw) \to \but$ (resp. $P: (M_T,\dt) \to \but$) restricted to $P: P^{-1}(BU(l)) \to BU(l)$. Let $\MM_l$ denote the  collection of intersections of Margulis risers with $M_l$.
	By Remark \ref{rmk-interval} and Theorem \ref{model-str}, $(M_l, \dw)$ (resp. $(M_l, \dt)$) is precisely the model metric obtained from the welded split blocks of Definition \ref{weld} (resp. the tube-electrified split blocks of Definition \ref{tubeel}).
	Note also that $\MM_l$ consists precisely of the welded annuli in $P^{-1}(BU(l))$.
	Let $\til{\MM_l}$ denote the collection of lifts of $\MM_l$ to the universal cover $\tmldw$.
	Theorem \ref{minskymodel} below will show that\\
	a) $\tmldw$ is (uniformly) strongly  hyperbolic relative to  the collection $\til{\MM_l}$, and\\
	b) By Lemma \ref{pel}, $\tmldt$ is (uniformly) hyperbolic and the elements of $(\til{\MM_l},\dt)$ are uniformly quasiconvex in it.
	\item[Step 2:] The converse direction of the combination theorem in this context, Corollary \ref{effectiveBFreversetmtdt} then furnishes effective flaring constants for $\tmldt$. Feeding these  effective flaring constants into the bundle $P: \tmtdt\to\but$   furnishes (effective) hyperbolicity of $\tmtdt$ by Corollary \ref{effectiveBFfwdtmtdt} and proves the first conclusion of Theorem
	\ref{effectivehyptmtdt}.
	\item[Step 3:]  Finally we extract effective  hyperbolicity of $\tmtdw$ {\it relative} to the collection $\tmr$ of lifts of Margulis risers
	and prove the second conclusion of Theorem
	\ref{effectivehyptmtdt}.
\end{enumerate}}
\end{scheme}

\noindent {\bf Step 1:} \\
For $BU(l)$  a blown-up bi-infinite geodesic in $\but$   with ending laminations $l_\pm$,  let $N_l$ be the doubly degenerate hyperbolic 3-manifold with end-invariants $l_\pm$. Theorem \ref{midsurfminskymodel}
and Theorem \ref{model-str} yield the following as a consequence.

\begin{theorem}\label{minskymodel}  \cite{minsky-elc1,minsky-elc2} Given $R, k, D_0 \geq 0$, there exists $\ep, D, C_0$ such that the following holds:\\ Let $T$ be an $L-$tight $R-$thick balanced tree with parameters $D_0, k$ and let	 $ (M_l, \dw), (M_l, \dt)$ be as above. Then
\begin{enumerate}
\item There exist   model manifolds  $M_l^m$ of special split geometry with constants $\ep, D>0$ as in Definition \ref{splsplit}.
such that $M_l^m$ is $C_0-$bi-Lipschitz homeomorphic to $N_l$.
\item The welded metrics and tube-electrified metrics  of Definitions \ref{weld}, \ref{tubeel}, \ref{def-mg} associated with  $M_l^m$ are 
$C_0-$bi-Lipschitz homeomorphic to $(M_l, \dw), (M_l, \dt)$ respectively.
\end{enumerate}
\end{theorem}

Since Margulis tubes are convex in any $N_l$ and uniformly separated from each other, it follows (see \cite{bowditch-relhyp} for instance) that  there exists $\delta_0$ such that
\begin{enumerate}
\item $\til{N_l}$ is  uniformly hyperbolic for all $l$ (since $\til{N_l}=\Hyp^3$),
\item  $\til{N_l}$ is strongly $\delta_0-$hyperbolic relative to the collection $\MM_l$ of lifts $\til T$ of Margulis tubes to $\til{N_l}$.
\end{enumerate}

Let $\partial \MM_l$ denote the collection of boundaries $\{ \partial \til{T} \vert  \til{T} \in \MM_l\}$, and let $Int(\MM_l) = \{ Int( \til{T}) \vert  \til{T} \in \MM_l \}$. Let $\til{N_l^0} = \til{N_l} \setminus \bigcup_{ Int( \til{T}) \in Int(\MM_l)}  Int( \til{T})$. 
By strong $\delta_0-$hyperbolicity of  $\til{N_l}$  relative to the collection $\MM_l$, it follows that $\til{N_l^0}$ is strongly $\delta_0-$hyperbolic relative to the collection $\partial \MM_l$. 

 Next, consider a standard annulus  isometric to $S^1 \times [0,l_i]$ in a welded block $B_{i,wel}$ (see Definition \ref{weld}) and let $f_i: \partial T_i \to S^1 \times [0,l_i]$ be the quotienting map defined in Definition \ref{weld}. Let $\til{f_i} :  \partial \til{T_i} \to \til{S^1} \times [0,l_i]$ be lifts of $f_i$ to $\til{N_l^0}$ Further, assume that $\til{S^1} \times [0,l_i]$ has been tube-electrified, by assigning the zero metric to the $\til{S^1}-$direction (note that $\til{S_i}$ is the real line  $\R$), so that after this tube-electrification operation, we obtain the universal cover $\til{N_{l,te}}$ of the tube-electrified model manifold $N_{l,te}$. Since the maps $\til{f_i} :  \partial \til{T_i} \to \til{S^1} \times [0,l_i]$ are clearly 1-Lipschitz, we 
have the following  by Lemma \ref{pel}, Theorem \ref{minskymodel} and
Proposition \ref{effectiveqc}:

\begin{prop}\label{telunifhyp}
{\bf Tube-electrified models are uniformly hyperbolic:} Given $R, k, D_0 \geq 0$, there exist $\delta', C \geq 0$ such that the following holds:\\ Let $T$ be an $L-$tight $R-$thick balanced tree with parameters $D_0, k$. For $BU(l)$ as before and $(M_l, \dw), (M_l, \dt)$ as in Theorem \ref{minskymodel},  the universal cover $(\til{M_l}, d_{te})$ of the tube-electrified model manifold $({M_l}, d_{te})$
is a $\delta'-$hyperbolic metric space. Further, each tube-electrified standard annulus
(or equivalently, each tube-electrified Margulis tube) in  $\MM_l$  equipped with $\dt$
is $C-$quasiconvex. 
\end{prop}

\noindent {\bf Alternate Proof:}
We furnish here an alternate proof of Proposition \ref{telunifhyp}. 
We use the notation of Proposition \ref{splsplit}. 

Recall that each special split block $B_i$ has injectivity radius bounded below by $\epsilon>0$ away from the Margulis tube $\T_i$.  Recall also that the core curve of $\T_i$ is denoted as $\tau_i$. Hence, there exists $K\geq 1$ (independent of $i$), such that $B_i \setminus \T_i$ is $K-$bi-Lipschitz to the thick part of the universal curve (i.e.\ the universal bundle minus a neighborhood of the cusps) over a thick Teichm\"uller geodesic segment $\gamma_i$ in $Teich_\ep(S \setminus \tau_i)$ for some uniform $\ep >0$. 
Then, due to uniform thickness of  Teichm\"uller geodesic segments $\gamma_i$, the bundle $(\til{M}, d_{te})$ satisfies flaring conditions with uniformly bounded constants. Strong relative hyperbolicity of $(\til{M}, d_{te})$ relative to $\MM_l$ now follows from Theorem \ref{effectiverelBF}. \hfill $\Box$
 
\smallskip 

\begin{rmk}
In applications we have in mind, especially \cite{mms}, the full strength of the model from Theorem \ref{minskymodel} used in the first proof of Proposition \ref{telunifhyp} becomes relevant. We have thus included two proofs, even though the alternate proof above does not use the full ending laminations machinery of \cite{minsky-elc1,minsky-elc2}.
\end{rmk}

 This completes Step 1 of Scheme \ref{scheme-split}.\\
 
\noindent {\bf Step 2:} \\
  Effective hyperbolicity of $(\til{M_T}, d_{te})$ now follows the same route as the proof of Proposition \ref{thickbdlhyp} (see also Remark \ref{thickbdlhyprmk} and Proposition \ref{effectiveqc}).

\begin{prop} Given $R>0$, there exists $\delta, C >0$ such that the following holds:\\ Let $T$ be an $L-$tight $R-$thick tree. Then $\tmtdt$ is $\delta-$hyperbolic.
	
	Further, each element of the set of Margulis risers $\tmr$ is
	 $C-$quasiconvex in $\tmtdt$. 
	\label{thinbdlhyp}
\end{prop}

\begin{proof} We follow the proof of Proposition \ref{thickbdlhyp}.
Uniform hyperbolicity of  $(\til{M_l}, d_{te})$ (Proposition \ref{telunifhyp}) ensures uniform flaring constants by Corollary \ref{effectiveBFreversetmtdt}  for $\tmldt$ independent of $l \subset T$.
This gives effective flaring constants for $\tmtdt$ as a bundle over $\but$. Hence by Corollary \ref{effectiveBFfwdtmtdt} there exists $\delta>0$ depending only on $R$ such that  $\tmtdt$ is $\delta-$hyperbolic.

Next, since each Margulis riser in $\tmtdt$ arises as a uniform quasi-isometric section of   a tree-link $T_v$, there exists $C >0$ such that each element of $\tmr$ is
$C-$quasiconvex in $\tmtdt$. 
\end{proof}
This completes Step 2 of Scheme \ref{scheme-split} and proves the first conclusion of Theorem \ref{effectivehyptmtdt}.\\

\noindent {\bf Step 3:}\\ We finally turn our attention to $\tmtdw$
and establish  that $\tmtdw$ is hyperbolic {\it relative to} $\tmr$ 
with effective constants. The argument will be an adaptation of very similar arguments in \cite{gautero,mahan-reeves} and we will provide a road-map through it instead of reproducing all the details.
The proof proceeds by first observing
the analogous statement for $\tmtdt$.

\begin{prop}\label{tmtdtrelhyp}
	Given $R>0$, there exists $\delta > 0$ such that the following holds:\\ Let $T$ be an $L-$tight $R-$thick tree. Then $\tmtdt$ is $\delta-$hyperbolic relative to $\tmr$.
\end{prop}

\begin{proof}
Proposition \ref{thinbdlhyp} shows that $\tmtdt$ is  $\delta-$hyperbolic and the Margulis risers in $\tmr$ are uniformly quasiconvex in $\tmtdt$. Uniform separatedness of the elements of  $\tmr$ is a consequence of the construction of $P:\tmtdt \to \but$. 

The proof of  uniform hyperbolicity of $\tmtdt$ {\it relative to}
the collection $\tmr$ is now
 a replica of  the proof of Theorem \ref{effectiverelBF} (the statement was culled from \cite{mahan-reeves,gautero}).  We omit the details and  mention only   that the elements of  $\tmr$ take the place of cone-loci  in \cite{mahan-reeves};  the rest of the proof is an exact copy.  
\end{proof}

\begin{prop} Given $R>0$, there exists $\delta > 0$ such that the following holds:\\ Let $T$ be an $L-$tight $R-$thick tree. Then 
$\tmtdw$ is strongly $\delta-$hyperbolic relative to the collection $\tmr$ of lifts of Margulis risers.

Further, if there exists $L_1$ such that the diameter of any tree-link $T_v$ is bounded above by $L_1$ for every $v$, then  $\tmtdw$ is  hyperbolic.
	\label{thinbdlrelhyp}
\end{prop}

\begin{proof}  
	
\noindent {\bf First statement of Proposition \ref{thinbdlrelhyp}:}	The 
proof of the first statement of Proposition \ref{thinbdlrelhyp}, i.e.\
 that there exists $\delta$ such that $\tmtdw$ is strongly  $\delta-$hyperbolic relative to $\tmr$ is a replica of the proof of Theorem 2.20 of \cite{gautero}. Instead of reproducing the argument here,
we shall now give specific references to the main steps of the proof from
\cite{gautero} and translate its terminology and  conclusion to our context.

First, we note that 	
	the main technical condition Gautero uses \cite[Definition 2.14]{gautero} is what he calls the exponential separation property. 
	In our context, this is equivalent to the effective flaring condition.
and is provided 	by Corollary \ref{effectiveBFreversetmtdt} applied
to the conclusion of Proposition \ref{tmtdtrelhyp} above.
	
Next, the proof of  \cite[Theorem 2.20]{gautero} have, as its main steps, \cite[Theorem 5.2]{gautero} and \cite[Proposition 7.4]{gautero} (proved in \cite[Section 9.7]{gautero}).  The proofs  of \cite[Theorem 5.2]{gautero} and \cite[Proposition 7.4]{gautero}, in turn,
  depend precisely on the  
exponential separation property hypothesis, which,
 as we have observed is a consequence
of  Proposition \ref{tmtdtrelhyp} and Corollary \ref{effectiveBFreversetmtdt}.
The first statement of the Proposition is now a translation, in the context of this paper, of   \cite[Theorem 2.20]{gautero}.\\

\noindent {\bf Second statement of Proposition \ref{thinbdlrelhyp}:}	The last statement of the Proposition now follows from  Proposition \ref{thickbdlhyp} since the upper bound  $L_1$ forces each bi-infinite geodesic $l$ in $T$ to lift to a geodesic in $\Teich_\ep$ with $\ep$ uniformly bounded away from $0$.
\end{proof}

This completes Step 3 of Scheme \ref{scheme-split} and the proof of 
the second conclusion of Theorem \ref{effectivehyptmtdt}. \hfill $\Box$\\

\section{Generalizations and Examples}\label{sec-gen} The purpose of this section is to generalize Theorem \ref{effectivehyptmtdt} to general $L-$tight, $R-$thick trees (Definition \ref{def-tighttree}) rather than just balanced ones. This comes at a cost. Uniform properness (Conclusion (4) of Theorem \ref{model-str}) is no longer valid. 

The tube electrification operation  (Definitions \ref{def:tubeelbb} and \ref{def-mg} is devised to electrify {\it as little as possible}. In the more general cases below, we are forced to electrify more.

\subsection{Lipschitz trees}\label{sec-lip}  An application of the technology developed in this paper is to prove cubulability of some surface-by-free hyperbolic groups \cite{mms}.  The main theorem of 
\cite{mms} requires the construction of quasiconvex tracks in $\tmtdt$.
This in turn requires that all the distances between end-points (leaves) of any tree-link $T_v$ is large. We thus define:

\begin{defn}\label{def-longtree}
	A finite metric tree $\TT$ is said to be $\lambda-$long if the distance between any two end-points (leaves) of $\TT$ is at least $\lambda$. 
	
	A geodesic from a leaf of a finite tree to another leaf will be called a {\bf long edge}.
	A continuous map $\phi$ from a finite tree $T_1$ to a finite tree $T_2$ will be called {\bf monotonic} if
	\begin{enumerate}
		\item $\phi$ is a bijection on leaves,
		\item $\phi$ maps long edges   monotonically (but not necessarily strictly  monotonically) to long edges.
	\end{enumerate}
\end{defn}

We shall now generalize Definition \ref{def-balancedtree}. We adapt the notation of Definition \ref{def-balancedtree}: $T_v^+$ denotes the tree-link obtained as an approximating tree of $CH(\ilkv)$. Let $T_v^-$ denote an approximating tree of $CH(\ilkv)'$. Note that the constants of approximation depend only on the number of vertices in 
$\ilkv$ and hence only on the valence of $v$. 

\begin{defn}\label{def-liptree}
	An $L-$tight $R-$thick tight tree  $i: V(T) \to \ccd(S)$ is said to be a {\bf Lipschitz} tree with parameters $D,k, \lambda$ if 
	\begin{enumerate}
		\item For every separating vertex $v$ of $T$, \[\d(\Pi'_v (T_w')) \leq D.\]
		\item Let $T_v^+, T_v^-$ be as above. There exists a $\lambda-$long tree $T_v$ with the same cardinality of leaves as  $T_v^+, T_v^-$ and  surjective $k-$Lipschitz monotonic maps $\P^+$ and  $\P^-$ from $T_v^+$ and $T_v^-$ respectively to $T_v$.
	\end{enumerate} 
\end{defn}

We have thus weakened the "coarse bi-Lipschitz" condition (equivalent to the surjective quasi-isometry condition) of Item (2) of Definition \ref{def-balancedtree}  to a coarse Lipschitz condition in Definition \ref{def-liptree} above. The tube-electrification process goes through via Lipschitz maps with the following modifications:

\begin{enumerate}
	\item The tree links $T_v$ are now the $\lambda-$long trees in Definition \ref{def-liptree} above.
	\item The Margulis risers are isometric to $S^1_e \times T_v$.
\end{enumerate}

With these modifications, the proof of Theorem \ref{effectivehyptmtdt} goes through as before to yield:

\begin{theorem}\label{effectivehypliptree}
Given $R\geq 1$, and $D, k, \lambda\geq 1$ there exists $ \delta >0$ such that the following holds:\\
For an  $L-$tight  $R-$thick Lipschitz tree  $T$ with parameters  $ D, k, \lambda$,
\begin{enumerate}
	\item $\tmtdw$ is strongly $\delta-$hyperbolic relative to the collection $\tmr$ of lifts of Margulis risers,
	\item $\tmtdt$ is  $\delta-$hyperbolic.
\end{enumerate} 
\end{theorem} 

Note again that
 hyperbolicity is not an issue in Theorem \ref{effectivehypliptree}, but the tube electrification process electrifies more by
 
 \begin{enumerate}
 \item Electrifying the $\R-$direction as before in Definition \ref{def:tubeelbb},
 \item Contracting the finite directions of Margulis risers  as well via the Lipschitz maps $\P^\pm$.
 \end{enumerate}

\subsection{General tight Trees}\label{sec-imb} We finally turn to the case when no large $\lambda$ is possible. To illustrate what can go wrong, define a tripod $\tau_x(a,b,c, A, B, C)$ to be a tree with a single trivalent vertex $x$ and leaves $a, b, c$ with $|xa|=A, |xb|=B, |xc|=C$.
Now, glue $\tau_x(a,b,c, 1, L, L/2)$ to $\tau_y(d,e,c, 1, L, L/2)$ by identifying only the vertices labeled $c$ to obtain a tree $T(a,b,d,e)$ with 4 leaves $a,b, d, e$ so that $d(a,x) = 1, d(b,x)=L, d(x,y)=L, d(y,d) = 1, d(y,e) = L$; in particular $T(a,b,d,e)$ is $L-$long.  Similarly, glue  tripods $\tau_{x'}(a',e',c', L, 1, L/2)$ to $\tau_{y'}(b',d',c', 1, L, L/2)$ by identifying only the vertices labeled $c'$ to obtain a tree $T(a',b',d',e')$. 
It follows that  $d(a',x') = L, d(e',x')=1, d(x',y')=L, d(y',d') = L, d(y',b') = 1$; in particular $T(a',b',d',e')$ is also $L-$long.

However, any tree $T(a'',b'',d'',e'')$ that receives  monotonic continuous maps $\phi, \phi'$ from both $T(a,b,d,e)$ and $T(a',b',d',e')$ such that $\phi(a) =\phi'(a')=a''$, $\phi(b) =\phi'(b')=b''$, and so on, has to necessarily be a star, i.e.\ the conditions $\phi(x)=\phi(y)$ and  $\phi'(x')=\phi'(y')$ are forced. Let $\phi(x)=\phi(y)=\phi'(x')=\phi'(y')=z$. If further, $\phi, \phi'$ are required to be $1-$Lipschitz, then $d(z,a''), d(z,b''), d(z,d''), d(z,e'')$ are all of length at most $1$.
Thus the only option for $T_v$ is a star where all limbs have length one.

One can arrange so that $T(a,b,d,e)$ and $T(a',b',d',e')$ are approximating trees of $CH(ilkv)$ and $CH(ilkv)'$ in the notation of Definition \ref{def-balancedtree}. Thus, in the general case (when the restrictive hypotheses of Definition \ref{def-balancedtree} is absent or the existence of a large $\lambda$ in Definition  \ref{def-liptree} is not guaranteed), the best we can hope  is for the tree $T_v$ to be a star where each edge has length one. In this case, $\lambda=2$ in Definition  \ref{def-liptree}. 

Let $\mtdt^*$ denote the  bundle with tube-electrified metric in the special case that each $T_v$ 
 in Definition  \ref{def-liptree} is a star  with all edges of length one. Let $\tmtdt^*$ denote  the universal cover. Theorem \ref{effectivehypliptree} then gives:
 
 \begin{cor}\label{effectivehypliptreecor}
 	Given. $R\geq 1$ there exists $ \delta >0$ such that the following holds:\\
 	For  an $L-$tight  $R-$thick tight tree,
  $\tmtdt^*$ is  $\delta-$hyperbolic.
  \end{cor} 

To conclude we note that each riser $\RR_v$ has diameter two in $\tmtdt^*$.
Thus  $\tmtdt^*$ is $(2,2)$-quasi-isometric to the space $\EE(\tmtdw, \tilde{\RR_\MM})$ obtained by electrifying the lifts of Margulis risers 
in $\tmtdw$. Thus, in the special case of balanced trees, Corollary \ref{effectivehypliptreecor} also follows immediately from the first statement of Theorem \ref{effectivehyptmtdt}.

\section*{Acknowledgments}    I gratefully acknowledge
several extremely helpful conversations with Dani Wise.   I thank  Ken Bromberg for the proof of Lemma \ref{ltightimpliesgeod}. 
This is the first part of  a project \cite{mms} in collaboration with Jason Manning and Michah Sageev.  In \cite{mms} we shall apply the model geometry of this paper to show that some surface-by-free hyperbolic groups are cubulable.  Much of the impetus for the present work thus comes from our collaboration and I gratefully acknowledge the contribution of Jason and Michah. Finally, I am  thankful to the anonymous referee(s) for a rather careful reading, and for pointing out the reference \cite{lott}.

This project was initiated during a visit to the McGill Mathematics Department in May 2015 and to MSRI, Berkeley in Fall 2016.
Parts of it were accomplished during  a workshop on Groups, Geometry and Dynamics in November 2017 at the International Centre for Theoretical Sciences, Bengaluru, and  during a visit to Technion, Israel, in May 2018.
I acknowledge the hospitality of these institutions.

\bibliography{modeltree}
\bibliographystyle{alpha}

\end{document}